\documentclass[11pt,letterpaper]{amsart}
\usepackage{amsmath,amssymb,amsthm,amsfonts, amscd, color, mathdots,mathtools}
\usepackage[bookmarkstype=toc, colorlinks=true, linkcolor=blue, citecolor=red, pdfborder={0 0 0}, bookmarks=true, bookmarksnumbered=true]{hyperref}
\usepackage{newtxtext,newtxmath}
\usepackage{mathrsfs}
\usepackage{verbatim}
\numberwithin{equation}{subsection}
\newtheorem{lemma}{Lemma}[section]

\newtheorem{proposition}{Proposition}[section]

\newtheorem{theorem}{Theorem}[section]

\newtheorem{corollary}{Corollary}[section]

\newtheorem{Definition}{Definition}[section]
\newtheorem{Remark}{Remark}[section]
\newtheorem{Conjecture}{Conjecture}
\newtheorem*{Claim}{Claim}
\newtheorem{Assumption}{Assumption}

\newcommand{\mR}{\mathbb{R}}

\newcommand{\mZ}{\mathbb{Z}}
\newcommand{\mA}{\mathbb{A}}
\newcommand{\mC}{\mathbb{C}}

\begin{document}
%
%
%
\title[Gan--Gross--Prasad conjecture and its refinement]{On  the Gan--Gross--Prasad conjecture and its refinement
  for $\left(\mathrm{U}\left(2n\right), \mathrm{U}\left(1\right)\right)$
  } 

\author{Masaaki Furusawa}
\address[Masaaki Furusawa]{
Department of Mathematics, Graduate School of Science,
              Osaka Metropolitan University,
         Sugimoto 3-3-138, Sumiyoshi-ku, Osaka 558-8585, Japan
}
\email{furusawa@omu.ac.jp}
\thanks{The research of the first author was supported in part by 
JSPS KAKENHI Grant Number
19K03407, 
22K03235
and by Osaka Central  Advanced Mathematical
Institute (MEXT Promotion of Distinctive Joint Research Center Program JPMXP0723833165).}
%
%
%
\author{Kazuki Morimoto}
\address[Kazuki Morimoto]{
Department of Mathematics, Graduate School of Science, 
Kobe University, 1-1 Rokkodai-cho, Nada-ku, Kobe 657-8501, Japan
}
\email{morimoto@math.kobe-u.ac.jp}
\thanks{The research of the second author was supported in part by
JSPS KAKENHI Grant Number 17K14166, 21K03164
and by
Kobe University Long Term Overseas Visiting Program for Young Researchers Fund.}
%
%
%
%
%
%
%
%
%
%
\begin{abstract}
We prove the Gan--Gross--Prasad conjecture for 
$\left(\mathrm{U}\left(2n\right), \mathrm{U}\left(1\right)\right)$ 
in the general case and prove  its refinement, namely the
 Ichino--Ikeda-type explicit formula
for the central $L$-values, 
under certain assumptions. 
By a similar argument, 
we also prove its split analogue, namely the Gan--Gross--Prasad conjecture
and its refinement for 
$\left(\mathrm{GL}_{2n}, \mathrm{GL}_1\right)$ in 
general.
\end{abstract}
%
%
%
\dedicatory{Dedicated to Freydoon Shahidi on the occasion 
of his 77th birthday.}
%
%
%
\maketitle
%
%
%
%
%
%
%
%
%
\section{Introduction}
Gross and Prasad~\cite{GP1,GP2} proclaimed a conjecture
relating
 non-vanishing of
certain period integrals of automorphic forms on orthogonal groups 
to
that of central special values of certain tensor product $L$-functions.
This conjecture has been
 extended to all classical groups and metaplectic groups by Gan, Gross and Prasad~\cite{GGP}. 
Meanwhile Ichino and Ikeda~\cite{II} formulated
a refinement of the Gross--Prasad conjecture which 
predicts a precise formula for the central $L$-value 
in terms of the period integral for irreducible
cuspidal tempered automorphic representations in the co-dimension one
case.
Then this refinement has been extended by Harris~\cite{Ha}, Liu~\cite{Liu2} and Xue~\cite{Xue}.
In the present
 paper, we study the Gan--Gross--Prasad conjecture and
 its refinement formulated by Liu~\cite{Liu2}
 in the case of $\left(\mathrm{U}\left(2n\right), \mathrm{U}\left(1\right)\right)$.
%
%

Let us introduce some notation to state our main results.
%
%
%
%
\subsection*{Notation}
Let $F$ be a number field. 
We denote its ring of adeles by $\mA_F$, which is mostly abbreviated as $\mA$ for simplicity. 
Let $\psi_F$ be a non-trivial character of $\mA \slash F$. 
For a place $v$ of $F$, we denote by $F_v$ the completion of $F$ at $v$ and by $\mathcal{O}_v$ its ring of integers
when $v$ is non-archimedean.

Let $E$ be a quadratic extension field of $F$ and 
$\mA_E$ be its ring of adeles. 
For a place $v$ of $F$, we denote $E \otimes_F F_v$ by $E_v$.
We denote by $\mathcal{O}_{E, v}$ its ring of integers when $v$ is non-archimedean.
For $x\in E$, 
let $x \mapsto \bar{x}$ be the unique non-trivial element of $\mathrm{Gal}(E \slash F)$.
Let us denote by $N_{E \slash F}$ the norm map from $E$ to $F$.
Let $\chi_E$ denote the quadratic character of $\mA^\times$ corresponding to $E \slash F$.
We define a character $\psi_E$ of $\mA_E \slash E$ by 
$\displaystyle{\psi_E(x) = \psi_F \left(\frac{x+\bar{x}}{2} \right)}$.
Since $\psi_E\left(x\right)=\psi_F\left(x\right)$
for $x\in\mA$, we simply write $\psi_E$ as $\psi$.
For any algebraic group $X$ defined over $F$ and a place $v$ of $F$, we often denote by $X_v$ the group of $F_v$-rational 
points $X\left(F_v\right)$ of $X$.
\emph{
Throughout the paper, we fix a character $\Lambda$ of $\mA_E^\times \slash E^\times$ whose
 restriction 
to $\mA^\times$ is trivial.}
%
%
%
\subsection{Unitary group $\mathrm{U}\left(V\right)$}
\label{subsection unitary group}
Let $(V, ( \, , \,)_V )$ be a $2n$-dimensional vector
 space over $E$ with a 
non-degenerate Hermitian pairing $(\, , \,)_V$.
We assume that the Witt index of $V$ is at least $n-1$.
Then we have a Witt decomposition $V = \mathbb{H}^{n-1} \oplus L$
where $\mathbb{H}$ is a hyperbolic plane over $E$ and $L$ is a $2$-dimensional Hermitian space over $E$.
Then $\mathcal{G}_n$ is defined as the set of 
$F$-isomorphism classes of the unitary groups $\mathrm{U}(V)$ for such $V$. We let $\mathrm{U}(V)$ act on $V$ from the left.
By abuse of notation, we shall often identify $\mathrm{U}(V)$ with 
its isomorphism class in $\mathcal{G}_n$.
We may decompose $V$ as a direct sum 
\begin{equation}\label{e: dedomposition of V}
V = X^- \oplus L \oplus X^+
\end{equation}
where $X^{\pm}$ are totally isotropic $(n-1)$-dimensional subspaces of $V$
which are dual to each other and orthogonal to $L$.
We take a basis $\{ e_{-1}, \dots, e_{-n+1} \}$ of $X^{-}$
and a basis
$\{ e_1, \dots, e_{n-1} \}$ of $X^{+}$, respectively so that
\begin{equation}
\label{(2)}
(e_{-i}, e_{j})_V = \delta_{i, j}
\end{equation}
for $1 \leq i, j \leq n-1$,
where $\delta_{i, j}$ denotes Kronecker's delta.
Then we shall write elements of $G$ as  matrices employing 
\[
e_{-1}, \dots, e_{-n+1}, \, \text{basis of $L$,} \, e_{n-1}, \dots, e_1,
\]
in this order, as a basis of $V$.

Let us denote by  $\left(\mathbb{V}_n,\left(\, , \, \right)_{\mathbb V}
\right)$ a $2n$-dimensional
Hermitian space with the Witt index $n$,
which is uniquely determined up to a scaling.
We often abbreviate $\mathbb{V}_n$ to $\mathbb V$.
When $V=\mathbb V$, we extend $X^-$ and $X^+$ to 
$V^-$ and $V^+$, respectively so that $V^-$ and $V^+$ are
totally isotropic $n$-dimensional subspaces of $\mathbb V$
which are dual to each other.
We take $e_{-n}\in V^-$ and $e_{n}\in V^+$,
respectively, so that \eqref{(2)} holds for $1 \leq i,j \leq n$.
Let $\mathbb G_n$ denote the unitary group 
$\mathrm{U}\left(\mathbb V_n\right)$.
We often abbreviate $\mathbb G_n$ to $\mathbb G$.
We employ 
\[
e_{-1}, \dots e_{-n+1}, e_{-n}, e_n, e_{n-1}, \dots, e_1,
\]
in this order, as a basis of $V$ for the matrix representation of 
elements of $\mathbb G_n$.
Then we have
\begin{equation}\label{unitary G_n}
\mathbb{G}_n\left(F\right)= \left\{ h \in \mathrm{GL}_{2n}(E) \colon {}^{t}\bar{h} w_{2n} h = w_{2n} \right\}.
\end{equation}
Here for a positive integer $r$, $w_r$ denotes 
the $r$ by $r$ matrix
\begin{equation}\label{w_r}
w_r =  \begin{pmatrix} &&1\\ &\iddots&\\ 1&& \end{pmatrix}.
\end{equation}
When we need to emphasize the distinction between
 $\mathbb G_n$
and the unitary group $\mathbb G_n^-=\mathrm{U}\left(
\mathbb W_n\right)$ for a skew Hermitian form $\mathbb W_n$
defined later by \eqref{e: def of G_n^-},
we denote $\mathbb G_n$ by $\mathbb G_n^+$ or $\mathbb G^+$.

\subsection{Bessel period}
Suppose $G = \mathrm{U}(V) \in \mathcal{G}_n$.
Let $P^\prime$ be the maximal parabolic subgroup of $G$ preserving 
the isotropic subspace $X^{-}$. 
Then $P^\prime$ has the Levi decomposition
$P^\prime=M^\prime S^\prime$ where 
$M^\prime$ is Levi part given by
\[
 M^\prime = 
 \left\{ \begin{pmatrix}g&&\\ &h&\\ &&g^\ast \end{pmatrix} \colon \, g \in \mathrm{Res}_{E \slash F} \,\mathrm{GL}_{n-1} ,\,  h \in \mathrm{U}(L) \right\}
 \]
with $g^\ast = w_{n-1} {}^{t} \bar{g}^{-1} w_{n-1}$
and $S^\prime$ the unipotent radical of $P^\prime$.

Let us take an anisotropic vector $e\in L$.
Then we define a character $\chi_e$ of $S^\prime\left(\mA\right)$
by
\[
\chi_e \begin{pmatrix} 1_{n-1}&A&B\\ &1_2&A^\prime\\ &&1_{n-1} \end{pmatrix}
\coloneqq \psi\left( \left(Ae, e_{n-1}\right)_V\right).
\]
Let us define a subgroup $\mathrm{St}_e$ of $M^\prime$ by
\[
\mathrm{St}_e\coloneqq
\left\{ \begin{pmatrix}p&&\\ &h&\\ &&p^\ast \end{pmatrix} \colon \, p \in \mathcal{P}_{n-1},\,  h \in \mathrm{U}(L), \, h e = e \right\}
\]
where $\mathcal{P}_{n-1}$ denotes the mirabolic subgroup of $\mathrm{Res}_{E \slash F}\, \mathrm{GL}_{n-1}$, i.e.
\[
\mathcal{P}_{n-1} = \left\{ \begin{pmatrix}\alpha&u\\ &1 \end{pmatrix} \colon \alpha \in \mathrm{Res}_{E \slash F}\, \mathrm{GL}_{n-2},
u \in \mathrm{Res}_{E \slash F}\, \mathbb{G}_a^{n-2} \right\}.
\]
Then $t\in \mathrm{St}_e\left(\mA\right)$
stabilizes $\chi_e$, i.e.
$\chi_e\left(ts^\prime t^{-1}\right)=\chi_e\left(s^\prime\right)$
for any $s^\prime\in S^\prime\left(\mA\right)$.

For a positive integer $r$, let $U_{r}$ denote the group of upper unipotent matrices in $\mathrm{Res}_{E \slash F}\, \mathrm{GL}_{r}$.
For $u \in U_{n-1}$, we define $\check{u} \in P^\prime$ by
\[
\check{u} \coloneqq \begin{pmatrix} u&&\\ &1_2&\\ &&u^\ast \end{pmatrix}.
\]
Then we define a unipotent subgroup $S$ of $P^\prime$ by
\[
S \coloneqq S^\prime S^{\prime \prime}
\quad
\text{where}
\quad
S^{\prime \prime} = \left\{ \check{u} \colon u \in U_{n-1}\right\}
\]
and we extend $\chi_e$ to a character of $S(\mA)$ by putting
\[
\chi_e(\check{u}) = \psi(u_{1, 2} + \cdots + u_{n-2, n-1}) \quad \text{for} \quad u =\left(u_{i,j}\right)\in U_{n-1}(\mA).
\]
We define a subgroup $D_{e}$ of $G$ by
\begin{equation}\label{bessel reductive part}
D_{e} \coloneqq \left\{ \begin{pmatrix}1_{n-1}&&\\ &h&\\ &&1_{n-1} \end{pmatrix} \colon h \in \mathrm{U}(L), \, h e = e \right\}
\end{equation}
and let $R_{e} \coloneqq D_{e} S$.
Then the elements of $D_e(\mA)$ stabilize a character $\chi_e$ of $R_e(\mA)$ by conjugation.
We note that 
\[
D_e(F) \simeq \{ a \in E^\times \colon \bar{a}a =1 \}.
\]
We may regard $\Lambda$ as a character of $D_e(\mA)$ by $d \mapsto \Lambda(\det d)$.
Then we define a character $\chi_{e, \Lambda}$ of $R_e(\mA)$ by
\[
\chi_{e, \Lambda}(ts) \coloneqq \Lambda(t) \chi_e(s) \quad \text{for} \quad t \in D_e(\mA), \, s \in S(\mA).
\]
%
%
%
\begin{Definition}\label{def of bessel}
For a cusp form $\varphi$ on $G(\mA)$, we define the $(e, \psi, \Lambda)$-Bessel period of $\varphi$ by 
\[
B_{e, \psi, \Lambda}(\varphi) \coloneqq \int_{D_e(F) \backslash D_e(\mA)} \int_{S(F) \backslash S(\mA)}
\chi_{e, \Lambda}\left(ts\right)^{-1} \varphi\left(ts\right) \, ds \, dt.
\]
For an irreducible cuspidal automorphic representation $(\pi, V_\pi)$ of $G(\mA)$,
we say that $\pi$ has the $(e, \psi, \Lambda)$-Bessel period when
 $B_{e, \psi, \Lambda}\not\equiv 0$  on $V_\pi$.
\end{Definition}
%
\begin{Remark}
Suppose that $e, e^\prime \in L$ satisfy $(e, e)_V = (e^\prime, e^\prime)_V \ne 0$.
Then 
for an irreducible cuspidal automorphic representation $\pi$ of $G(\mA)$,
we note that  $\pi$ has the $(e, \psi, \Lambda)$-Bessel period
if and only if $\pi$ has the $(e^\prime, \psi, \Lambda)$-Bessel period.
\end{Remark}
%
%
%
\subsection{The Gan--Gross--Prasad conjecture}
Our first main result is the implication in one direction
of the Gan--Gross--Prasad conjecture in the case of $(\mathrm{U}(2n), \mathrm{U}(1))$.
%
%
\begin{theorem}
\label{period to L-value thm}
Let $\pi$ be an irreducible cuspidal automorphic representation
of $G(\mA)$ for $G \in \mathcal{G}_n$ whose local component 
$\pi_w$
at some finite place $w$ is generic. 
Suppose that $\pi$ has the $(e, \psi, \Lambda)$-Bessel period.
Then we have
\[
L \left(\frac{1}{2}, \pi \times \Lambda \right) \ne 0.
\]
\end{theorem}
%
%
%
%
\begin{Remark}
When $\pi$ is globally generic, this direction of the
Gan--Gross--Prasad conjecture has been
 proved by Ginzburg, Jiang and Rallis~\cite{GJR} 
when the base change lift of $\pi$ is cuspidal and  then by Ichino and Yamana~\cite[Theorem~5.1]{IY} in general.
Meanwhile
 Jiang and Zhang~\cite[Theorem~5.7]{JZ} proved this direction when $\pi$ has a generic $A$-parameter.
In their case there exists an irreducible cuspidal globally generic representation of $\mathbb{G}(\mA)$ 
which is nearly equivalent to $\pi$,
thanks to the global descent method,
 and in particular, our assumption on local genericity is satisfied.
We note that in a recent preprint, Hazeltine, Liu and Lo~\cite[Theorem~1.8]{HLL} proved that a local $A$-packet
is tempered if and only if the packet contains a generic member,
 in the case of symplectic groups and also in
 the case of odd special orthogonal groups. 
We note that if a similar result holds for $G(F_w)$,  then our assumption is equivalent to the one
that $\pi$ has a generic $A$-parameter.
\end{Remark}
\color{black}
%
When $\pi$ is tempered, we may
 prove the following theorem
as predicted in \cite[Conjecture~24.1]{GGP}. 
\begin{theorem}
\label{opp GGP}
Let $\pi$ be an irreducible cuspidal tempered automorphic representation
of $G(\mA)$ for $G \in \mathcal{G}_n$.
Let us fix a triple $\left(e, \psi,\Lambda\right)$.

Then the following two conditions are equivalent :
\begin{enumerate}
\item $L \left(\frac{1}{2}, \pi \times \Lambda \right) \ne 0$.
\item There exists a pair
$\left(G^\prime, \pi^\prime\right)$
where $G^\prime \in \mathcal{G}_n$ and 
$\pi^\prime$ an irreducible 
cuspidal tempered automorphic representation  of $G^\prime(\mA)$
such that $\pi^\prime$ is nearly equivalent to $\pi$
and $\pi^\prime$ has the $(e, \psi, \Lambda)$-Bessel period.
\end{enumerate}
Also the following two conditions are equivalent :
\begin{enumerate}
\item[(a)]
$L \left(\frac{1}{2}, \pi \times \Lambda \right) \ne 0$
and $\operatorname{Hom}_{R_{e,v}}
\left(\pi_v,\chi_{e, \Lambda, v}\right) \ne 0$
at any place $v$ of $F$.
\item[(b)]
 $\pi$ has the $(e, \psi, \Lambda)$-Bessel period.
\end{enumerate}
\end{theorem}
%
%
%
\begin{Remark}
The first equivalence was proved also in Jiang and Zhang~\cite[Theorem~6.10]{JZ}
under the assumption that \cite[Conjecture~6.8]{JZ} holds for $\pi$.
\end{Remark}
%
%
%
We prove Theorems~\ref{period to L-value thm} and
\ref{opp GGP}
by computing the pull-back of the
Whittaker periods with respect to  the 
theta lifting from $G$ to $\mathbb{G}_n^-$.
This is an adaptation of the method used
in the proof of \cite[Theorem~1]{FM1} and  \cite[Corollary~1]{FM1}
to our situation.
%
%
%
\subsection{Refinement of the Gan--Gross--Prasad conjecture}
\label{intro ref ggp}
We may
prove the refinement of the Gan--Gross--Prasad conjecture
for $(\mathrm{U}(2n), \mathrm{U}(1))$, i.e.
the Ichino--Ikeda-type central $L$-value formula,
utilizing an explicit formula for Whittaker periods as in
\cite{FM2}.
In order to state our result, let us introduce  some additional notation.

Let $(\pi, V_\pi)$ be an irreducible tempered cuspidal automorphic representation
of $G\left(\mathbb A\right)$ for $G\in\mathcal G_n$.
Let us take the Tamagawa measures $dg$ on $G(\mA)$ 
and $dt$ on $D_e(\mA)$, respectively. Then we note that 
\[
\mathrm{Vol} \left(G(F) \backslash G(\mA), dg \right) = \mathrm{Vol} \left(D_e(F) \backslash D_e(\mA), dt \right) = 2.
\]

Let $\langle\, \,,\,\,\rangle$ denote the $G\left(\mathbb A\right)$-invariant
Hermitian inner product on $V_\pi$
given by the Petersson inner product, i.e.
\[
\langle\varphi_1,\varphi_2\rangle=
\int_{G\left(F\right)\backslash G\left(\mathbb A\right)}
\varphi_1\left(g\right)\,\overline{\varphi_2\left(g\right)}\,
dg
\quad\text{for $\varphi_1,\varphi_2\in V_\pi$.}
\]
Since $\pi=\otimes_v \,\pi_v$ where $\pi_v$ is unitary
for each place $v$ of $F$, we may 
choose a $G(F_v)$-invariant Hermitian
inner product $\langle\,\,,\,\,\rangle_v$ on $V_{\pi_v}$, 
the space of $\pi_v$,  
so that we have
 \[
 \langle\varphi_1 ,\varphi_2\rangle=\prod_v\,\langle\varphi_{1,v} ,\varphi_{2,v}\rangle_v
 \quad\text{for}\quad
\varphi_i=\otimes_v\,\varphi_{i,v}\in V_\pi\quad\left(i=1,2\right).
 \]
We choose a local Haar measure 
$dg_v$ on $G(F_v)$ at each place $v$ of $F$
so that $\mathrm{Vol}\left(G(\mathcal{O}_v),dg_v\right)=1$
at almost all finite places $v$.
Let us also choose a local Haar measure $dt_v$
on $D_{e,v}=D_e\left(F_v\right)$ at each place $v$
so that $\mathrm{Vol}\left(D_e(\mathcal{O}_v),dt_v\right)=1$
at almost all finite places $v$.
We define positive constants $C_G$ and $C_e$,
called Haar measure constants in \cite{II},  by
%
%
%
\begin{equation}\label{e: measure comparison}
dg=C_G\cdot\prod_v \,dg_v\quad\text{and}\quad
dt=C_e\cdot \prod_v\, dt_v,
\end{equation}
respectively.
%
%
\subsubsection{Local Bessel period}
\label{local integral sec}
At each place $v$, a local Bessel period
$\alpha_v\left(\varphi_v,\varphi_v^\prime\right)$ for smooth vectors
$\varphi_v,\,\varphi_v^\prime\in V_{\pi_v}$ is defined as follows.
%

Suppose that $v$ is non-archimedean.
We  define $\alpha_v\left(\varphi_v,\varphi^\prime_v\right)$ by
\begin{equation}\label{e: local integral 1}
\alpha_v\left(\varphi_v,\varphi_v^\prime\right)
\coloneqq
\int_{D_{e, v}}\int_{S_v}^{\mathrm{st}}
\langle\pi_v\left(s_vt_v\right)\varphi_v,\varphi_v^\prime\rangle_v\,
\chi_{e, \Lambda}^{-1}\left(s_v t_v\right)\, ds_v\,dt_v.
\end{equation}
Here, the integral over $S_v =S(F_v)$ is the stable integral 
defined as in \cite[Definition~2.1,Remark~2.2]{LM}.
Indeed it is shown in Liu~\cite{Liu2} that for any $t_v\in D_{e, v}$
the inner integral of \eqref{e: local integral 1} stabilizes at a certain open compact subgroup of $S_v$
 \cite[Proposition~3.1]{Liu2} and 
the outer integral of \eqref{e: local integral 1} converges \cite[Theorem~2.1]{Liu2}.
%

Now suppose that $v$ is archimedean.
For $u=\left(u_{i,j}\right)\in S_v$, we put
\[
u_{i} = u_{i, i+1} \quad (1 \leq i \leq n-2), \quad u_{n-1} = \left(u e, e_{n-1}\right)_V.
\]
For $\gamma \geq - \infty$, we define 
\[
S_{v, \gamma} =  \left\{ u \in S_v \colon |u_i| \leq e^\gamma \, (1 \leq i \leq n-1) \right\}.
\]
Then Liu~\cite[Corollary~3.13]{Liu2} showed that 
\[
\alpha_{\varphi_v, \varphi_v^\prime}(u)\coloneqq
\int_{D_{e, v}} \int_{S_{v, -\infty}} \langle \pi_v\left(us_vt_v\right)\varphi_v,\varphi_v^\prime\rangle_v \, ds_v \, dt_v
\]
converges absolutely and it gives a tempered distribution on $S_v \slash S_{v, -\infty} \simeq E_v^n$.

For an abelian Lie group $N$ over $E_v$, we denote by $\mathcal{D}(N)$ (resp. $\mathcal{S}(N)$)
the space of tempered distributions (resp. Schwartz functions) on $N$.
Then we have a natural
bilinear pairing $(\,, \,) \colon \mathcal{D}(N) \times \mathcal{S}(N) \rightarrow \mC$.
We define the Fourier transform $\hat{} \colon \mathcal{D}(N)  \rightarrow \mathcal{D}(N) $ by the formula
\[
\left(\hat{\mathfrak{a}}, \phi\right) = \left(\mathfrak{a}, \hat{\phi}\right) 
\quad\text{for $\mathfrak{a}\in\mathcal{D}(N)$
and $\phi \in \mathcal{S}(N)$}
\]
where $\hat{\phi}$ is the Fourier transform of $\phi$.

Let $\left(\widehat{S_v \slash S_{v, -\infty}}\right)^{\rm reg}$ be the regular locus of the
Pontryagin dual $\widehat{S_v \slash S_{v, -\infty}}$. We note that $\left(\widehat{S_v \slash S_{v, -\infty}}\right)^{\rm reg}$
is an open dense subset of $\widehat{S_v \slash S_{v, -\infty}}$.
For $\phi \in \mathcal{S} \left(\left(\widehat{S_v \slash S_{v, -\infty}}\right) \right)$, we define 
its  Fourier transform $\widehat{\phi}\,$  by 
\[
\widehat{\phi}(n) = \int_{\left(\widehat{S_v \slash S_{v, -\infty}}\right)^{\rm reg}} \,\phi\left(\widehat{u}\,\right) \,\widehat{u}(n)\, d \widehat{u}
\]
with the dual measure $d \widehat{u}$ of $ds_v$ on $S_v \slash S_{v, -\infty}$. In this way, we may regard the Fourier transform of 
$\mathcal{S} \left(\left(\widehat{S_v \slash S_{v, -\infty}}\right) \right)$ as a subspace of $\mathcal{S}(S_v \slash S_{v, -\infty})$.
Then, by \cite[Proposition~3.14]{Liu2}, there exists a smooth function $f_{\varphi_v, \varphi_v^\prime} \in  \mathcal{S} \left(\left(\widehat{S_v \slash S_{v, -\infty}}\right) \right)$
such that 
\[
\left(\widehat{\alpha_{\varphi_v, \varphi_v^\prime}}, \phi\right) = \int_{\left(\widehat{S_v \slash S_{v, -\infty}}\right)^{\rm reg}}\, f_{\varphi_v, \varphi_v^\prime}(t)\, \phi(t) \, dt
\]
for any $\phi \in C_c^\infty \left(\left(\widehat{S_v \slash S_{v, -\infty}}\right)^{\rm reg} \right)$.
In this sense, the Fourier transform $\widehat{\alpha_{\varphi_v, \varphi_v^\prime}}$
is smooth on $\left(\widehat{S_v \slash S_{v, -\infty}}\right)^{\rm reg}$.

Then we may define an archimedean local Bessel period
$\alpha_v\left(\varphi_v,\varphi_v^\prime\right)$ by
\begin{equation}\label{e: local integral 2}
\alpha_v\left(\varphi_v,\varphi_v^\prime\right) \coloneqq \widehat{\alpha_{\varphi_v, \varphi_v^\prime}} \left( \chi_{e, \Lambda, v} \right)
\end{equation}
since $ \chi_{e, \Lambda, v} \in \left(\widehat{S_v \slash S_{v, -\infty}}\right)^{\rm reg}$.
%

We recall that the local multiplicity one property holds at
any place $v$, i.e.
\begin{equation}
\label{uniqueness}
\dim_{\mathbb C}\,\operatorname{Hom}_{R_{e,v}}
\left(\pi_v,\chi_{e, \Lambda, v}\right)\le 1.
\end{equation}
We refer \eqref{uniqueness}
to  Gan, Gross and Prasad~\cite[Corollary~15.3]{GGP}
and Jiang, Sun and Zhu~\cite[Theorem~A]{JSZ}
for the non-archimedean case and the archimedean case, respectively.
Moreover when $v$ is not-split, it is shown by
Beuzart-Plessis~\cite{BP1,BP2}  that
\begin{multline}\label{e: multiplicity one}
\dim_{\mathbb C}\,\operatorname{Hom}_{R_{e,v}}
\left(\pi_v,\chi_{e, \Lambda, v}\right)= 1
\\
\Longleftrightarrow
\text{$\alpha_v\left(\varphi_v,\varphi_v^\prime\right)\ne 0$
for some $\varphi_v,\,\varphi_v^\prime\in V_{\pi_v}^\infty$}
\end{multline}
where
$V_{\pi_v}^\infty$ denotes the space of smooth vectors in $V_{\pi_v}$.
As we remarked in \cite[(1.11)]{FM2}, the 
condition on the right-hand side 
of \eqref{e: multiplicity one} is equivalent to:
\begin{equation}\label{e: multiplicity one2}
\text{$\alpha_v\left(\varphi_v,\varphi_v\right)\ne 0$
for some
$\varphi_v\in V_{\pi_v}^\infty$.}
\end{equation}
%
%
%
%
\subsubsection{Normalization of local integrals}
\label{section: normalization}
We fix maximal compact subgroups $K_G\coloneqq\prod_v K_{G,v}$ 
of $G\left(\mA\right)$ and $K_e\coloneqq\prod_v K_{e,v}$
of $D_e\left(\mA\right)$
such that $K_{G,v}=G(\mathcal{O}_v)$ and $K_{e, v} = D_e(\mathcal{O}_v)$
for almost all finite places $v$.

We say that a place $v$ is  \emph{good}
(with respect to $\pi$ and a decomposable vector $\varphi=\otimes_v\,\varphi_v
\in V_\pi=\otimes_v\, V_{\pi_v}$)
if the following conditions are satisfied:
%
\begin{subequations}\label{e: good place}
\begin{equation}
\text{$v$ is non-archimedean and is not lying over $2$};
\end{equation}
\begin{equation}
\text{$E_v$ is an unramified quadratic extension of $F_v$
or $E_v=F_v\oplus F_v$};
\end{equation}
\begin{equation}
\text{$K_{G,v}$ is a hyperspecial maximal compact subgroup of $G(F_v)$};
\end{equation}
\begin{equation}
\text{$\pi_v$ is an unramified representation of $G(F_v)$};
\end{equation}
\begin{equation}
\text{$\phi_v$ is a $K_{G,v}$-fixed vector such that $\langle\varphi_v,
\varphi_v\rangle_v=1$ and $\chi_{e, \Lambda, v}$ is $K_{e, v}$-fixed};
\end{equation}
\begin{equation}
\text{$K_{e, v}\subset K_{G,v}$
and $\mathrm{Vol}\left(K_{G,v},dg_v\right)=\mathrm{Vol}\left(K_{e, v},
dt_v\right)=1$}.
\end{equation}
\end{subequations}
%
Then Liu's theorem~\cite[Theorem~2.2]{Liu2} implies that when
$v$ is good, one has
\begin{equation}\label{e: liu's theorem}
\alpha_v\left(\varphi_v,\varphi_v\right)=
\frac{L\left(1/2,\pi_v \times \Lambda_v \right)
\,\prod_{j=1}^{2n}L\left(j, \chi_{E, v}^j \right)}{
L\left(1,\pi_v, \mathrm{Ad}\right)
L\left(1,\chi_{E,v}\right)}.
\end{equation}

We define the normalized  local Bessel period
$\alpha_v^\natural\left(\varphi_v,\varphi^\prime_v\right)$
at each place $v$ of $F$
by
\begin{equation}\label{e: normalized}
\alpha_v^\natural\left(\varphi_v,\varphi^\prime_v\right)\coloneqq
\frac{
L\left(1,\pi_v, \mathrm{Ad}\right)
L\left(1,\chi_{E,v}\right)}
{L\left(1/2,\pi_v \times \Lambda_v \right)
\,\prod_{j=1}^{2n}L\left(j, \chi_{E, v}^j \right)}\cdot
\alpha_v\left(\varphi_v,\varphi_v^\prime\right).
\end{equation}
Here, local $L$-factors are defined as follows.
For each place $v$ of $F$, let $\phi_{\pi_v}$ denote
 the local Langlands parameter of $\pi_v$ given by \cite{KMSW,Mok}.
Then we define 
\[
L\left(s,\pi_v,\mathrm{Ad}\right)  \coloneqq L\left(s,\phi_{\pi_v},\mathrm{Ad}\right). 
\]
Since  
\[
L\left(s,\phi_{\pi_v},\mathrm{Ad}\right) = L\left(s, \mathrm{BC} (\phi_{\pi_v}),\mathrm{As}^-\right),
\]
we have 
\begin{equation}
\label{Lfct bc as}
L\left(s,\pi,\mathrm{Ad}\right) = L\left(s,\Pi,\mathrm{As}^-\right)
\end{equation}
where $\Pi$ denotes the base change lift of $\pi$ to $\mathrm{GL}_{2n}(\mA_E)$
and $\mathrm{As}^-$ stands for the Asai $L$-function.
The existence of $\Pi$ follows from \cite{KMSW} and \cite{Mok}.
We know that  $L\left(s,\pi,\mathrm{Ad}\right)$ has a meromorphic continuation to the entire complex plane $\mC$
and is holomorphic and non-zero at $s=1$ by \cite[Corollary~2.5.9]{Mok} and \cite[Theorem~5.1]{Sha0}.
Meanwhile we may define $L\left(s,\pi_v \times \Lambda_v \right)$ 
by the $\gamma$-factors defined by the
doubling method as in Lapid and Rallis~\cite{LR2}.
By the uniqueness of $\gamma$-factors, we have 
\[
L\left(s,\pi_v \times \Lambda_v \right) = L\left(s,\phi_{\pi_v} \times \Lambda_v \right),
\]
which is holomorphic for $\mathrm{Re}\left(s\right)>0$ by Yamana~\cite{Yam}
since $\pi_v$ is tempered.
%
%
%
\subsubsection{Refinement of the Gan--Gross--Prasad conjecture}
Lapid and Mao~\cite{LM} 
conjectured the Ichino--Ikeda-type explicit formula
for the Whittaker periods.
See Conjecture~\ref{main conj} 
in Section~\ref{LM-conjecture}
for the case we need in this paper.
Under certain assumptions,
the second author proved Conjecture~\ref{main conj}
in \cite{Mo2}.
The refined Gan--Gross--Prasad
conjecture for the Bessel period $B_{e,\psi,\Lambda}$
follows by combining it with the pull-back computation of 
the Whittaker periods of the theta lift of $\pi$
to $\mathbb{G}_n^-\left(\mA\right)$ as in
\cite{FM2}.

Indeed we prove the following theorem.
%
%
%
\begin{theorem}
\label{refined ggp thm}
Let $(\pi, V_\pi)$ be an irreducible cuspidal tempered automorphic representation of $G(\mA)$ where $G \in \mathcal{G}_n$.
Assume the following conditions:
\begin{itemize}
\item  $F$ is totally real and $E$ is a  totally imaginary quadratic extension of $F$,
\item $\pi_v$ is a discrete series representation at every archimedean place $v$ of $F$,
\item $\pi_v, \psi_v$ and $\Lambda_v$ are unramified at all finite split places.
\item   every finite place $v$ is not split when $v$ divides $2$.
\end{itemize}
Then for any non-zero decomposable vector $\varphi = \otimes \varphi_v \in V_\pi$, we have 
\begin{multline}
\label{formula uni}
\frac{|B_{e, \psi, \Lambda}(\varphi)|^2}{\langle \varphi, \varphi \rangle} = \frac{C_e}{\left|S(\Psi(\pi))\right|} 
\,
\left(\prod_{j=1}^{2n} L(j, \chi_{E}^j) \right)
\\
\times \frac{L(\frac{1}{2}, \pi \times \Lambda)}{L\left(1, \pi, \mathrm{Ad} \right) L(1, \chi_{E})}
\,
\prod_v \frac{\alpha_v^\natural(\varphi_v)}{\langle \varphi_v, \varphi_v \rangle}
\end{multline}
where $\Psi(\pi)$ denote the $A$-parameter of $\pi$
and $S(\Psi(\pi))$ its $S$-group.
Here we use the notation
\begin{equation}\label{e: abbreviation}
\alpha_v^\natural\left(\varphi_v\right)=\alpha_v^\natural\left(\varphi_v,\varphi_v\right)
\quad\text{for $\varphi_v\in V_{\pi_v}$}.
\end{equation}
\end{theorem}
%
%
%
%
\begin{Remark}
Here we note that \eqref{formula uni} is consistent with \cite[Conjecture~2.5]{Liu2}.
Let us define a $D_e(\mA)$-invariant inner product $\langle \,\,, \,\,\rangle_\Lambda$ of 
$\Lambda$
by the Petersson inner product with respect to 
the Tamagawa measure on $D_e(\mA)$. 
At each place $v$, we take the $D_e(F_v)$-invariant Hermitian inner product 
$\langle \,\,, \,\,\rangle_{\Lambda_v}$ given by 
$\langle a, b\rangle_{\Lambda_v} = a\bar{b}$ with $a, b \in \mC$.
Then we have
\[
\langle\, \,,\, \,\rangle_\Lambda = 2 \cdot \prod_v \langle\, \,, \,\,\rangle_{\Lambda_v}
\]
since the volume of $D_e(F) \backslash D_e(\mA)$ is equal to $2$.
Let $\Psi(\Lambda)$ denote the $A$-parameter of $\Lambda$.
Since $\left|S(\Psi(\Lambda))\right| =2$, we may rewrite \eqref{formula uni}
as 
\begin{multline*}
\frac{|B_{e, \psi, \Lambda}(\varphi)|^2}{\langle \varphi, \varphi \rangle \langle \Lambda, \Lambda \rangle_\Lambda } 
= \frac{C_e}{\left|S(\Psi(\pi))\right| \cdot \left|S(\Psi(\Lambda))\right|} 
\,\left(\prod_{j=1}^{2n} L(j, \chi_{E}^j) \right)
\\
\times \frac{L(\frac{1}{2}, \pi \times \Lambda)}{L\left(1, \pi, \mathrm{Ad} \right) L(1, \chi_{E})}
\prod_v \frac{\alpha_v^\natural(\varphi_v)}{\langle \varphi_v, \varphi_v \rangle  \langle \Lambda_v, \Lambda_v\rangle_{\Lambda_v}}
\end{multline*}
where we regard $\Lambda$ as an automorphic form on $D_e(\mA)$.
Thus \eqref{formula uni} is consistent with \cite[Conjecture~2.5]{Liu2}.
\end{Remark}
%
%
%
%
\begin{Remark}
Recently the Gan--Gross--Prasad conjecture and
its refinement in the co-dimension one
unitary group case has been  proved
in  remarkable papers  by 
Beuzart-Plessis, Liu, Zhang and Zhu~\cite{BPLZZ}
and Beuzart-Plessis, Chaudouard and Zydor~\cite{BPCZ},
in the stable case and in the endoscopic case, respectively.
Note that our $\left(\mathrm{U}\left(2n\right),
\mathrm{U}\left(1\right)\right)$ case
is not included  in the co-dimension one case unless $n=1$.
\end{Remark}
%
%
%
\begin{Remark}
After the submission of
 this paper, Beuzart-Plessis and Chaudouard 
posted an outstanding  paper \cite{BPC} on arXiv, in which they 
prove the
Gan--Gross--Prasad conjecture and its refinement 
for Bessel periods on unitary groups in general
by establishing  the Jacquet--Rallis relative trace formulas.
Though the main results of our paper are subsumed 
in theirs as a special case, the authors believe that
the different approach taken here has its own merit 
because of its simplicity and straightforwardness.
\end{Remark}
%
%
%
By an argument similar to the 
 proof of Theorem~\ref{refined ggp thm},
we may prove a split analogue, i.e. the 
refined Gan--Gross--Prasad conjecture in the case of $(\mathrm{GL}_{2n}, \mathrm{GL}_1)$.
Since an explicit formula for the Whittaker periods on $\mathrm{GL}_m$ holds in general by Lapid and 
Mao~\cite[Theorem~4.1]{LM},
we do not need to impose any assumption,
contrary to Theorem~\ref{refined ggp thm}.
In fact we have the following theorem where we use notation
similar to that in 
Theorem~\ref{refined ggp thm}. 
We refer to  Section~\ref{s:RGGP GL} for the precise definitions
in the split setting.
\begin{theorem}
\label{refined GGP GL}
Let $(\pi, V_\pi)$ be an irreducible cuspidal tempered automorphic representation of $\mathrm{GL}_{2n}(\mA)$
and $\eta$ a character of $\mA^\times \slash F^\times$.
Let us fix a decomposition of the measure $dg$ on $\mathrm{GL}_1(\mA)$
as $dg= C_1 \prod_v dg_v$ and also 
fix a decomposition of the
Petersson inner product $\langle \,\,, \,\,\rangle = \prod_v \langle \,\,, \,\,\rangle_v$.

Then for any non-zero decomposable vector $\varphi = \otimes \varphi_v \in V_\pi$, we have
\begin{multline}\label{refined ggp gl}
\frac{|B_{\psi, \eta}(\varphi)|^2}{\langle \varphi, \varphi \rangle} 
=C_1 \left( \prod_{j=2}^n \zeta_F(j) \right) 
\\
\times
\frac{L \left( \frac{1}{2}, \pi \times \eta \right) L \left( \frac{1}{2}, \pi^\vee \times \eta^{-1} \right)}{L(1, \pi, \mathrm{Ad})\,\mathrm{Res}_{s=1} \zeta_F(s)}
\prod_v \frac{\alpha_v^\natural (\varphi_v) }{\langle \varphi_v, \varphi_v \rangle_v}
\end{multline}
where $B_{\psi, \eta}\left(\varphi\right)$
denotes the $\left(\psi,\eta\right)$-Bessel period
of $\varphi$ defined by \eqref{e: def of bessel for GL}.
\end{theorem}
\subsection{Organization of the paper}
This paper is organized as follows.
In Section~\ref{s:Set up}, we introduce the set up.
In Section~\ref{Pull-back computation 1}, we compute the pull-back of Whittaker periods concerning the theta lift for the dual pair
$\left(G,\mathbb{G}^-\right)$.
In Section~\ref{s:proof thm}, we prove the Gan--Gross--Prasad conjecture, 
Theorem~\ref{period to L-value thm} and
Theorem~\ref{opp GGP},
based on the pull-back computation in Section~\ref{Pull-back computation 1}.
In Section~\ref{s:proof ref}, by fine-tuning the argument in Section~\ref{s:proof thm}, we prove the
refined Gan--Gross--Prasad conjecture in the case of 
$(\mathrm{U}(2n), \mathrm{U}(1))$.
In Section~\ref{s:RGGP GL}, by an argument similar to those in Section~\ref{Pull-back computation 1} and \ref{s:proof thm}, we shall prove the
refined Gan--Gross--Prasad conjecture in the case of $(\mathrm{GL}_{2n}, \mathrm{GL}_1)$.
In Appendix~\ref{s:Similitude unitary groups case}, by adapting the arguments for the proof
of Theorem~\ref{refined ggp thm},
we settle the refined Gan--Gross--Prasad conjecture
in the similitude unitary group case under certain assumptions,
which is indispensable 
in our forthcoming paper~\cite{FM3}
on the Gross--Prasad conjecture with its refinement
for $\left(\mathrm{SO}\left(5\right),\mathrm{SO}\left(2\right)\right)$
and the generalized B\"ocherer conjecture.
%
%
%
%

%
%
%
%
%
%
%
%
%
%
%
%
%
%
%
%
%
%
%
%
\section{Set up}
\label{s:Set up}
\subsection{The group $\mathbb{G}_n^-$}
\label{subsection skew hermtian}
For a positive integer $n$, let $(\mathbb{W}_n, (\, ,\, )_{\mathbb{W}_n})$ be a  $2n$-dimensional vector space over $E$ with non-degenerate skew-Hermitian pairing $(\, , \,)_{\mathbb{W}_n}$
whose Witt index is $n$, 
which is  determined uniquely up to a scaling.
We let its isometry group $\mathrm{U}(\mathbb{W}_n)$
act on $\mathbb W_n$ from the right.
We denote $\mathrm{U}(\mathbb{W}_n)$ by $\mathbb{G}_n^-$.
We  fix a basis of $\mathbb{W}_n$ 
so that the matrix representation of $\mathbb{G}_n^-$ 
is given by
\begin{equation}\label{e: def of G_n^-}
\mathbb{G}_n^-(F)= \left\{ g \in \mathrm{GL}_{2n}(E) \colon {}^{t}\bar{g} J_n g = J_n \right\}
\quad
\text{where}
\quad
J_n = \begin{pmatrix} &1_n\\ -1_n&\end{pmatrix}.
\end{equation}
Let  $\xi\in E^\times$
such that $\bar{\xi}=-\xi$.
Then
$\xi\cdot\left(\, , \,\right)_{{\mathbb W}_n}$
is a Hermitian pairing of Witt index $n$.
We note that hence we have
\begin{equation}\label{pm iso}
\mathbb G_n^-\simeq \mathbb G_n
\end{equation}
where $\mathbb G_n$ is defined by \eqref{unitary G_n}.

For $a \in \mathrm{Res}_{E \slash F}\, \mathrm{GL}_n$, we define $\hat{a} \in \mathbb{G}_n^-$ by
\begin{equation}\label{gl_n part}
\hat{a} \coloneqq \begin{pmatrix}a\\& {}^{t}\bar{a}^{-1} \end{pmatrix}
\end{equation}
and let
\begin{equation}\label{levi}
M_n \coloneqq \{ \hat{a} \colon a \in \mathrm{Res}_{E \slash F}\, \mathrm{GL}_n \}.
\end{equation}
For $b \in \mathrm{Herm}_n \coloneqq \{ X \in \mathrm{Res}_{E \slash F}\,\mathrm{Mat}_{n \times n} \colon {}^{t}\overline{X} = X \}$,
we define 
\begin{equation}\label{unipotent part}
n(b) \coloneqq \begin{pmatrix} 1_n&b\\ &1_n\end{pmatrix} \in \mathbb{G}_n^-
\end{equation}
and let 
\begin{equation}\label{unip}
U_{\mathbb{G}_n^-} \coloneqq  \{ n(b) \colon b \in \mathrm{Herm}_n \}.
\end{equation}
Then $M_n U_{\mathbb{G}_n^-}$ is the Siegel parabolic subgroup of $\mathbb{G}_n^-$.
%
%
%
%
%
%
%
%
%
\subsection{Weil representation and theta correspondence}
\label{ss:weil rep}
Let us briefly recall the Weil representation of unitary groups. 
Let $(X, (\, , \,)_X)$ be an $r$-dimensional Hermitian space over $E$,
and let  $(Y, (\,, \,)_Y)$ be an $s$-dimensional skew-Hermitian space over $E$.
Then we may define the symplectic space 
\[
\left(W_{X, Y}, (\,, \,)_{X, Y}\right) \coloneqq 
\left(\mathrm{Res}_{E \slash F}\left(X \otimes_E Y\right), \mathrm{Tr}_{E \slash F}\left((\,, \, )_X \otimes \overline{(\,,\,)_Y} \,\right) \right).
\] 
This is a $2rs$-dimensional
symplectic space over $F$.
We let its isometry group $\mathrm{Sp}\left( W_{X, Y}\right)$ act on $W_{X, Y}$ from the right. 
Then a homomorphism 
$
\mathrm{U}(X) \times \mathrm{U}(Y) \rightarrow \mathrm{Sp}\left( W_{X, Y}\right)
$ is
given by 
\begin{equation}
\label{dual pair action}
(x \otimes y)(g, h) := g^{-1}x \otimes yh
\quad\text{for $x \in X, y \in Y, g \in \mathrm{U}(X),h \in \mathrm{U}(Y)$}.
\end{equation}

For each place $v$ of $F$, we denote the metaplectic extension of $\mathrm{Sp}\left(W_{X, Y}\right)(F_v)$
by $\mathrm{Mp}\left( W_{X, Y}\right)(F_v)$. Also we denote
by $\mathrm{Mp}\left( W_{X, Y}\right)(\mA)$ the 
metaplectic extension of $\mathrm{Sp}\left(W_{X, Y}\right)(\mA)$.

Let $\chi=\left(\chi_X,\chi_Y\right)$ be a pair of characters of 
$\mA_E^\times \slash E^\times$
such that $\chi_{X}|_{\mA^\times} = \chi_E^r$ and $\chi_{Y}|_{\mA^\times} = \chi_E^s$.
For each place $v$ of $F$, let 
\begin{equation}
\label{local splitting}
\iota_{\chi_v} \colon \mathrm{U}(X)(F_v) \times \mathrm{U}(Y)(F_v) \rightarrow \mathrm{Mp}(W_{X, Y})(F_v)
\end{equation}
be the local splitting given by Kudla~\cite{Ku},
which  depends on $\chi_v =(\chi_{X,v}, \chi_{Y,v})$.
Using this local splitting, we get a global splitting 
\[
\iota_{\chi} \colon \mathrm{U}(X)(\mA) \times \mathrm{U}(Y)(\mA) \rightarrow \mathrm{Mp}(W_{X, Y})(\mA),
\]
which depends 
 on $\chi = (\chi_X, \chi_Y)$. Then by pulling back, we obtain the Weil representation $\omega_{\psi, \chi}$ of $\mathrm{U}(X)(\mA) \times \mathrm{U}(Y)(\mA)$.
When we fix a polarization $W_{X, Y} = W_{X, Y}^+ \oplus  W_{X, Y}^-$, we may realize  $\omega_{\psi, \chi}$
so that its space of smooth vectors is given by $\mathcal{S}(W_{X, Y}^+(\mA))$, the space of Schwartz-Bruhat functions
 on $W_{X, Y}^+(\mA)$.
Then for $\phi \in \mathcal{S}(W_{X, Y}^+(\mA))$, we define the theta function $\theta_{\psi, \chi}^\phi$ on 
$\mathrm{U}(X)(\mA) \times \mathrm{U}(Y)(\mA)$ by
\begin{equation}
\label{theta fct def}
\theta_{\psi, \chi}^\phi\,(g, h) \coloneqq \sum_{w \in W_{X, Y}^+(F)} \omega_{\psi, \chi}(g, h)\phi(w).
\end{equation}
Let $(\sigma, V_\sigma)$ be an irreducible cuspidal automorphic representation of $\mathrm{U}(X)(\mA)$.
Then for $\varphi \in V_\sigma$ and $\phi \in \mathcal{S}(W_{X, Y}^+(\mA))$, we define the theta lift of $\varphi$ by 
\[
\theta_{\psi, \chi}^\phi(\varphi)(h) \coloneqq \int_{\mathrm{U}(X)(F) \backslash \mathrm{U}(X)(\mA)} \varphi(g) \theta_{\psi, \chi}^\phi(g, h) \, dg,
\]
which is an automorphic form on  $\mathrm{U}(Y)(\mA)$.
Further, 
the theta lift of $\sigma$ to $\mathrm{U}\left(Y\right)\left(\mA\right)$,
which we denote by $\Theta_{X, Y}(\sigma, \psi, \chi)$,
is defined by
\[
\Theta_{X, Y}(\sigma, \psi, \chi)\coloneqq
\left<
\theta^\phi_{\psi, \chi}(\varphi) \colon\varphi \in V_\sigma,\, \phi \in \mathcal{S}(W_{X, Y}^+(\mA))
\right>
\]
where the right-hand side means an automorphic representation
of $\mathrm{U}\left(Y\right)\left(\mA\right)$
generated.
When the spaces $X$ and $Y$ which we consider are evident, we simply write $\Theta(\sigma, \psi, \chi)$ for 
$\Theta_{X, Y}(\sigma, \psi, \chi)$.
Conversely, for an irreducible cuspidal automorphic representation $\tau$ of $\mathrm{U}(Y)(\mA)$, its theta lift
$\Theta_{Y, X}(\tau, \psi, \chi) $
to $\mathrm{U}\left(X\right)\left(\mA\right)$
is similarly defined
and we may simply write it as $\Theta(\tau, \psi, \chi)$.

Let us fix a place $v$ of $F$. 
Let $\omega_{\psi_v, \chi_v}$ be the Weil representation of $\mathrm{U}(X)(F_v) \times \mathrm{U}(Y)(F_v)$
defined as above by using the local splitting \eqref{local splitting}.
For an irreducible admissible representation $\pi$ of $\mathrm{U}(X)(F_v)$ (resp. $\mathrm{U}(Y)(F_v)$),
the maximal $\pi^\vee$-isotypic quotient of $\omega_{\psi_v, \chi_v}$ is of the form
\begin{equation}
\label{local theta def form}
\pi^\vee \boxtimes \Theta_{X, Y}(\pi, \psi_v, \chi_v)
\end{equation}
where $\Theta(\pi, \psi_v, \chi_v) = \Theta_{X, Y}(\pi, \psi_v, \chi_v)$ is a smooth representation of $\mathrm{U}(Y)(F_v)$ (resp. $\mathrm{U}(X)(F_v)$).
Let $\theta(\pi, \psi_v, \chi_v) = \theta_{X, Y}(\pi, \psi_v, \chi_v)$ denote the maximal semisimple quotient of $\Theta(\pi, \psi_v, \chi_v)$.
Then $\theta(\pi, \psi_v, \chi_v)$ is either zero or irreducible
by
the Howe duality, which is  proved by Howe~\cite{Ho1} at archimedean places, by 
 Waldspurger~\cite{Wa} at odd finite places and finally by Gan and Takeda~\cite{GT} at all finite places.
 \begin{Remark}
\label{IP rem}
In the case when $F_v =\mR$ and $E_v = \mC$,
we shall use results by Ichino~\cite{Ich20} and Paul~\cite{Pau} 
on local theta correspondences.
Here we would like to remark about the
relation between our local theta correspondence and theirs.

From the definition, the theta lift in \cite{Ich20} is  $\theta(\pi^\vee, \psi_v, \chi_v)$ in our notation,
which is isomorphic to $\theta(\pi, \psi_v^{-1}, \chi_v^{-1})^\vee$
by Gan, Kudla and Takeda~\cite[Proposition~15.10]{GKT}.

On the other hand, Paul~\cite{Pau} studies the
theta lifts between genuine representations of covering groups of
$\mathrm{U}(X)(F_v)$ and $\mathrm{U}(Y)(F_v)$.
In order to apply 
the results in \cite{Pau},
we need to lift representations of $\mathrm{U}(X)(F_v)$ and $\mathrm{U}(Y)(F_v)$ to their covering groups.
We obtain an explicit relation between our theta lifts and those
in \cite{Pau} by considering twists by certain genuine characters,
as explained carefully in the discussion after Lemma~3.5
in Xue~\cite{Xue23} (see also Atobe~\cite[Section~3.3, 3.4]{HA}).
In particular, thanks to the relation \cite[(3.16)]{Xue23}, we may  apply the results in \cite{Pau} directly
to our case. Note that Xue~\cite{Xue23} uses
the same local theta lift  as in \cite{Ich20}.
\end{Remark}

%
%
%
%
%
%
%
%
%
%
%
%
%
%
%
%
%
%
%
%
%
%
%
%
%
%
%
%
%
\section{Pull-back computation of the Whittaker periods}
\label{Pull-back computation 1}
%
For positive integers $m$ and $n$, 
we consider $\mathbb G_m^-$ and
$G= \mathrm{U}(V) \in \mathcal{G}_n$.
We take characters $\chi_{V}$ and $\chi_{\mathbb{W}}$ of $\mA_E^\times \slash E^\times$
such that $\chi_{V}|_{\mA^\times}=\chi_{\mathbb{W}}|_{\mA^\times} =1$.
We may realize the Weil representation $\omega=\omega_{\psi, \chi}$, where $\chi = (\chi_V, \chi_{\mathbb W})$,
of $G(\mA) \times \mathbb{G}_m^{-}(\mA)$
with $\mathcal{S}(V(\mA)^m)$
as its space of smooth vectors.
The actions on $\phi \in \mathcal{S}(V(\mA)^m)$
are explicitly given by,
\begin{align}
\label{weil rep}
\notag \omega(g,1) \phi(x_1, \dots, x_m) &=\chi_V(\det g)  \phi(g^{-1} (x_1, \dots, x_m)), \\
\omega (1, \hat{a}) \phi(x_1, \dots, x_m)&= \chi_{\mathbb{W}}(\det a) |\det a|^{m} \phi((x_1, \dots, x_m) a),\\
\notag \omega(1, n(b)) \phi(x_1, \dots, x_m)& = \psi\left(\frac{1}{2} \,\mathrm{tr}(b\, \mathrm{Gr}(x_1, \dots, x_m)) \right) \phi\left(x_1, \dots, x_m\right),
\end{align}
where $g\in G\left(\mA\right)$,
$a \in \mathrm{GL}_m(\mA_E)$ and $b \in \mathrm{Herm}_m(\mA)$.
Here $\mathrm{Gr}(x_1, \dots, x_m)$
stands for the Gram matrix 
$ \left(\left( x_i, x_j \right)_V\right)_{1\le i,j\le m}$ and $g (x_1, \dots, x_m)\coloneqq( g x_1, \dots,  g x_m)$
for $g\in G\left(\mA\right)$ and $\left(x_1,\dots , x_m\right)
\in V\left(\mA\right)^m$.

Let $(\pi, V_\pi)$ be an irreducible cuspidal automorphic representation of $G(\mA)$. In this section, we study the theta lift $\Theta_{V, \mathbb{W}_m}(\pi, \psi, \chi)$ 
of $\pi$ to $\mathbb{G}_m^-(\mA)$ and the Whittaker period 
on $\Theta_{V, \mathbb{W}_m}(\pi, \psi, \chi)$.

Let $N_m$ be the unipotent subgroup of $\mathbb{G}_m^-$ defined by 
\[
N_m = \left\{ \hat{u}\, n(b) \colon u \in U_m, \, b \in \mathrm{Herm}_m\right\}\]
where we recall
\eqref{gl_n part} and \eqref{unipotent part}.
We also denote by $T_m$ the subgroup consisting of diagonal matrices in $\mathbb{G}_m^-$.
Then $T_m N_m$ is 
a Borel subgroup of $\mathbb{G}_m^-$.

For $\lambda\in F^\times$, a non-degenerate character
$\psi_{N_m,\lambda} \colon N_m\left(\mA\right)\to\mathbb C^\times$
is defined by 
\begin{equation}\label{generic character}
\psi_{N_m, \lambda}(u) = \psi \left( \sum_{i=1}^{m-1}u_{i, i+1} + \frac{\lambda}{2}\, u_{m, 2m}\right)
\quad
\text{for $u=\left(u_{i,j}\right)\in N_m$}.
\end{equation}

For an automorphic form $\varphi$ on $\mathbb{G}_m^-(\mA)$, we define the
$\psi_{N_m, \lambda}$-Whittaker period of $\varphi$
by
\begin{equation}\label{whittaker period}
W_{\psi, \lambda}(\varphi) \coloneqq \int_{N_m(F) \backslash N_m(\mA)} \varphi(u)\, \psi_{N_m, \lambda}(u)^{-1} \, du.
\end{equation}

For an irreducible cuspidal automorphic representation $(\sigma, V_\sigma)$ of $\mathbb{G}_m^-(\mA)$,
we say that $\sigma$ is $\psi_{N_m, \lambda}$-generic (or simply, generic)
if $W_{\psi,\lambda}\not\equiv 0$  on $V_\sigma$.
Now, recall that $L$ is the
$2$-dimensional hermitian space over $E$ with the decomposition \eqref{e: dedomposition of V}
as defined in Section~\ref{subsection unitary group}.
%
%
%
\begin{proposition}
\label{pullback whittaker prp}
Let $(\pi, V_\pi)$ be an irreducible cuspidal automorphic representation of $G(\mA)$ with $G = \mathrm{U}(V) \in \mathcal{G}_n$.
\begin{enumerate}
\item Suppose that $m>n$. Then $W_{\psi, \lambda} \equiv 0$ on $\Theta_{V, \mathbb{W}_m}(\pi, \psi, \chi)$.
\item Suppose that $m=n$.
\begin{enumerate}
%
\item 
Unless there exists $e\in L$ such that $\left(e,e\right)_V=\lambda$,
we have $W_{\psi, \lambda} \equiv 0$ on $\Theta_{V, \mathbb{W}_n}(\pi, \psi, \chi)$.
%
\item Suppose that $e_\lambda \in L$ satisfies 
$\left(e_\lambda, e_\lambda\right)_V = \lambda$. 
Then for $\phi \in \mathcal{S}(V(\mA)^n)$ and $\varphi \in V_\pi$, 
we have
\begin{multline}
\label{pullback whittaker prp identity}
W_{\psi, \lambda}\left(\theta^\phi_{\psi, \chi}(\varphi)\right) =
 \int_{R_{e_\lambda}^\prime(\mA) \backslash G(\mA)}
   \omega_{\psi, \chi}(g, 1)
  \phi(e_{-1}, \dots,e_{-n+1}, e_\lambda)
  \\
  \times  B_{e_\lambda, \psi, \chi_V^{-1}}(\pi(g) \varphi) \, dg
\end{multline}
where 
\[
R_{e_\lambda}^\prime = \{ g \in G \colon g e_{-1} = e_{-1}, \,\dots,\, g e_{-n+1} = e_{-n+1}, \,g e_\lambda = e_\lambda \}.
\]
\end{enumerate}

In particular, $\Theta_{V, \mathbb{W}_n}(\pi, \psi, \chi)$ is $\psi_{N_n, \lambda}$-generic
if and only if 
$\pi$ has the  $(e_\lambda, \psi, \chi_V^{-1})$-Bessel period.
\end{enumerate}
\end{proposition}
\begin{proof}
Our claim is proved by applying the argument
in \cite[p.94--98]{Fu}, mutatis mutandis, to our setting.

From the definition of the Whittaker period, we may write 
\begin{multline*}
W_{\psi, \lambda}\left(\theta^\phi_{\psi, \chi}(\varphi)\right)
= \int_{N_M(F) \backslash N_M(\mA)} \int_{U_{\mathbb{G}_m^-}(F) \backslash U_{\mathbb{G}_m^-}(\mA)}
\\
\times
 \theta^\phi_{\psi, \chi}(\varphi)(vu) \psi_{N_m, \lambda}(vu)^{-1} \, du \, dv
\end{multline*}
where $N_M = N_m \cap M_m$.
Moreover, by the definition of the theta lift, this is equal to
\begin{multline*}
 \int_{N_M(F) \backslash N_M(\mA)} \int_{U_{\mathbb{G}_m^-}(F) \backslash U_{\mathbb{G}_m^-}(\mA)} \int_{G(F) \backslash G(\mA)} 
 \\
 \times
 \sum_{(a_i) \in V(F)^m} 
 \omega_{\psi, \chi}(g, vu) \phi(a_1, \dots, a_m)
\,
\varphi(g)\, \psi_{N_m, \lambda}(vu)^{-1}\, dg \, du \, dv.
\end{multline*}
Further,
this integral is equal to
\begin{multline}\label{lemma1 in F}
 \int_{N_M(F) \backslash N_M(\mA)}  \int_{G(F) \backslash G(\mA)} 
 \\
 \times
 \sum_{(a_i) \in \mathcal{X}^\lambda} \omega_{\psi, \chi}(g, v) \phi(a_1, \dots, a_m)
\varphi(g) 
\,\psi_{N_m, \lambda}(v)^{-1}\, dg  \, dv
\end{multline}
where 
\[
\mathcal{X}^\lambda = \left\{(a_1, \dots, a_m) \in V(F)^m \colon \mathrm{Gr}(a_1, \dots, a_m) = \begin{pmatrix}0&\dots&\dots&0\\ \vdots&&&\vdots\\ 0&\dots&\dots&0 \\ 0&\dots&0&\lambda \end{pmatrix}\right\}.
\]
Then as \cite[Lemma~1]{Fu}, in \eqref{lemma1 in F},
only the terms such that $a_1,\dots , a_m$ are linearly 
independent contribute.

Suppose that $m>n$ and that 
there exists $\left(a_1,\dots , a_m\right)\in \mathcal X^\lambda$
such that $a_1,\dots , a_m$ are linearly independent.
We denote the subspace of $V$ generated by 
$a_1,\dots , a_{m-1}$ by $V^+$.
Then $V^+$ is  an 
$\left(m-1\right)$-dimensional totally isotropic subspace
of $V$.
 Since the Witt index of $V$ is at most $n$,
we have $m=n+1$ and the Witt index of $V$ is $n$.
Then
 $\left(a_i,a_{n+1}\right)_V=0\,\,\left(1\le i\le n\right)$
implies that $a_{n+1}\in V^+$.
This is a contradiction since $a_1,\dots , a_n, a_{n+1}$
are linearly independent.
Hence the integral \eqref{lemma1 in F} must vanish
when $m>n$.

Suppose that  $m=n$.
When  $\left(a_1,\dots, a_n\right)\in \mathcal X^\lambda$ and 
$a_1,\dots, a_n$ are linearly independent,
there exists an element $h\in G\left(F\right)$
such that $ha_i=e_{-i}$ for $1\le i\le n-1$ and $ha_n\in L$
by Witt's theorem.
Hence the integral \eqref{lemma1 in F} vanishes when there does not exist
$e\in L$ such that $\left(e,e\right)_V=\lambda$.

Suppose $m=n$ and $e_\lambda\in L$ satisfies $\left(e_\lambda,e_\lambda\right)_V=\lambda$.
Then by Witt's theorem, the integral \eqref{lemma1 in F} becomes 
\begin{align*}
 &\int_{N_M(F) \backslash N_M(\mA)}  \int_{G(F) \backslash G(\mA)} 
 \\
&\times
\sum_{\gamma \in R_{e_\lambda}^\prime(F) \backslash G(F)} \omega_{\psi, \chi}(g, v)
\phi(\gamma^{-1} e_{-1}, \dots, \gamma^{-1} e_{-n+1}, \gamma^{-1}e_\lambda)
\\
&
\qquad\qquad\qquad\qquad\qquad\qquad\qquad\qquad
 \times 
\varphi(g)\, \psi_{N_n, \lambda}(v)^{-1}\, dg  \, dv
\\
=
 &\int_{N_M(F) \backslash N_M(\mA)}  \int_{R_{e_\lambda}^\prime(F) \backslash G(\mA)} \omega_{\psi, \chi}(g, v)
  \phi(e_{-1}, \dots,e_{-n+1}, e_\lambda)
\\
&
\qquad\qquad\qquad\qquad\qquad\qquad\qquad\qquad
\times
\varphi(g)\, \psi_{N_n, \lambda}(v)^{-1}\, dg  \, dv.
\end{align*}
Here $R_{e_\lambda}^\prime = D_{e_\lambda} S_{e_\lambda}^{\prime}$
with
$D_{e_\lambda}$  defined by \eqref{bessel reductive part}
for $e=e_\lambda$ and
\[
S_{e_\lambda}^{\prime} 
\coloneqq \left\{ \begin{pmatrix}1_{n-1}&A&B\\ &1_2&A^\prime\\ &&1_{n-1} \end{pmatrix} \in S^\prime \colon Ae_\lambda = 0 \right\}.
\]
We note that $R_{e_\lambda}^\prime$ is unimodular.
From the formulas~\eqref{weil rep}, the integral above is written as
\begin{align*}
 &\int_{R_{e_\lambda}^\prime(\mA) \backslash G(\mA)} \int_{D_{e_\lambda}(F) \backslash D_{e_\lambda}(\mA)}
  \int_{S_{e_\lambda}^{\prime}(F) \backslash S_{e_\lambda}^{\prime}(\mA)} \int_{N_M(F) \backslash N_M(\mA)} 
  \\
   &\quad\quad\times
   \omega_{\psi, \chi}(h s^{\prime} g, v) \phi(e_{-1}, \dots,e_{-n+1}, e_\lambda)
\varphi(h s^{\prime} g) \psi_{N_n, \lambda}(v)^{-1} \, dv \, ds^{\prime} \, dh \, dg
\\
=
 &\int_{R_{e_\lambda}^\prime(\mA) \backslash G(\mA)} \int_{D_{e_\lambda}(F) \backslash D_{e_\lambda}(\mA)}
  \int_{S_{e_\lambda}^{\prime}(F) \backslash S_{e_\lambda}^{\prime}(\mA)} \int_{N_M(F) \backslash N_M(\mA)}
  \\
 &\,\,\times   \omega_{\psi, \chi}(g, v)
  \phi(e_{-1}, \dots,e_{-n+1}, e_\lambda) \chi_V(\det h)
\varphi(h s^{\prime} g) \psi_{N_n, \lambda}(v)^{-1}  \, dv \, ds^{\prime} \, dh \, dg.
\end{align*}
When we regard $U_{n-1}$ as a subgroup of $U_n$ by 
$u \mapsto \begin{pmatrix} u&\\ &1\end{pmatrix}$, we may write 
\[
U_n = U_{n-1} U_0^\prime
\]
where 
\begin{equation}
\label{def u_0^prime}
U_0^\prime = \left\{ u_0^\prime=\begin{pmatrix} 1&0&\cdots&0&a_1\\ &1&\ddots&\vdots&\vdots\\ 
&&\ddots&0&a_{n-2}\\ &&&1&a_{n-1}\\ &&&&1\end{pmatrix} \in U_n \right\}.
\end{equation}
%
For $j (1 \leq j \leq n-1)$, let $L_j$ denote
 the subspace of $V$ spanned by $e_{-j}$, $e_\lambda$ and $e_j$.
Let $r_j(a_j)$ be the elements in $\mathrm{U}(L_j, \mA)$ whose matrix representation
with respect to the basis 
$\{ e_{-j}, e_\lambda, e_j \}$ is given by
\[
r_j(a_j) = \begin{pmatrix}1&a_j&-\frac{1}{2}\lambda^{-1} N_{E \slash F}(a_j)\\ &1&-\lambda^{-1}\overline{a_j}\\ &&1 \end{pmatrix}.
\]
We extend $r_j(a_j)$ to an element of $G(\mA)$ by letting it act trivially on $L_j^\perp$.
We note that $r_j(a_j)$ is an adaptation
 of $s_j(a_j)$ in \cite[p.97]{Fu} to our setting.
Then we have 
\begin{align}
\label{act r_i}
 \notag &\omega_{\psi, \chi}(g, \widehat{u_0^\prime u} h) \varphi(e_{-1}, \dots, e_{-n+1}, e_\lambda)
\\
=
 &\omega_{\psi, \chi}(g, \hat{u} h) \varphi(e_{-1}, \dots, e_{-n+1}, e_\lambda+\sum_{j=1}^{n-1}a_j e_{-j})
\\
=
\notag &\omega_{\psi, \chi}(r_1(a_1)^{-1} \cdots r_{n-1}(a_{n-1})^{-1} g, \hat{u} h)
\varphi(e_{-1}, \dots, e_{-n+1}, e_\lambda).
\end{align}
Hence our integral becomes
\begin{align*}
 &\int_{R_{e_\lambda}^\prime(\mA) \backslash G(\mA)} \int_{D_{e_\lambda}(F) \backslash D_{e_\lambda}(\mA)}
  \int_{E^{n-1} \backslash \mA_E^{n-1}}\int_{U_{n-1}(F) \backslash U_{n-1}(\mA)}
    \int_{S_{e_\lambda}^{\prime}(F) \backslash S_{e_\lambda}^{\prime}(\mA)} 
  \\
  &\qquad\times
\omega_{\psi, \chi}(r_1(a_1)^{-1} \cdots r_{n-1}(a_{n-1})^{-1} g, \hat{u})
\phi(e_{-1}, \dots, e_{-n+1}, e_\lambda)
\\
 &\qquad\qquad\times \chi_V(\det h)
\varphi(h s^{\prime} g) \psi_{N_n, \lambda}(\widehat{u_0} \widehat{u})^{-1}  \, ds^{\prime} \, du \, du_0  \, dh \, dg
\\
=
 &\int_{R_{e_\lambda}^\prime(\mA) \backslash G(\mA)} \int_{D_{e_\lambda}(F) \backslash D_{e_\lambda}(\mA)}
\int_{U_{n-1}(F) \backslash U_{n-1}(\mA)}
    \int_{S^{\prime}(F) \backslash S^{\prime}(\mA)} 
  \\
&\qquad\times\omega_{\psi, \chi}(g, \hat{u})
\phi(e_{-1}, \dots, e_{-n+1}, e_\lambda) 
\\
&\qquad\qquad\times
\chi_V(\det h)
\varphi(h s^{\prime} g) \chi_{e_\lambda}(s^\prime)^{-1} 
 \psi_{N_n, \lambda}(\widehat{u})^{-1} \, ds^{\prime} \, du \, dh \, dg
\end{align*}
since we may take $\left\{ r_{n-1}(a_{n-1}) \cdots r_1(a_1)\colon
a_i\in \mathbb A_E \right\}$ as a set of representatives for
$S_{e_\lambda}^{\prime}\left(\mA\right) \backslash S^{\prime}\left(\mA\right)$.
Then by a computation similar to that
 in \cite[p.96--98]{Fu}, we finally have
\begin{align*}
&W_{\psi, \lambda}\left(\theta^\phi_{\psi, \chi}(\varphi)\right)
=
 \int_{R_{e_\lambda}^\prime(\mA) \backslash G(\mA)}
   \omega_{\psi, \chi}(g, 1)
  \phi(e_{-1}, \dots,e_{-n+1}, e_\lambda) 
  \\
 & \qquad\qquad\qquad\qquad\times
  \int_{R_{e_\lambda}(F) \backslash R_{e_\lambda}(\mA)}
  \chi_{e_\lambda, \chi_V^{-1}}(r)^{-1} \varphi(rg)\, dr \, dg
  \\
  =
 &\int_{R_{e_\lambda}^\prime(\mA) \backslash G(\mA)}
   \omega_{\psi, \chi}(g, 1)
  \phi(e_{-1}, \dots,e_{-n+1}, e_\lambda)  B_{e_\lambda, \psi, \chi_V^{-1}}(\pi(g) \varphi) \, dg.
 \end{align*}

%
The last assertion concerning the non-vanishing follows from this formula by
an argument similar to that for the proof of \cite[Proposition~2]{FM1}.
Let us briefly recall the argument. One direction is clear. Let us suppose that 
there exists $\varphi \in V_\pi$ such that $B_{e_\lambda, \psi, \chi_V^{-1}}(\varphi) \ne 0$.
Since $R_{e_\lambda}^\prime(F_v) \backslash G(F_v) \simeq  G(F_v) \cdot (e_{-1}, \dots,e_{-n+1}, e_\lambda)$ 
is locally closed in $V(F_v)^n$, as in the proof of \cite[Proposition~2]{FM1}, 
we can reduce our assertion to the non-vanishing of 
\[
I_{\Sigma} =  \int_{R_{e_\lambda}^\prime(\mA_\Sigma) \backslash G(\mA_\Sigma)}
   \omega_{\psi, \chi}(g, 1)
  \phi(e_{-1}, \dots,e_{-n+1}, e_\lambda)  B_{e_\lambda, \psi, \chi_V^{-1}}(\pi(g) \varphi) \, dg
\]
for $\mA_{\Sigma} = \prod_{v \not \in \Sigma} F_v$
where $\Sigma$ is a sufficiently large finite set of places of $F$.
Indeed, let us take $\Sigma$ so that for $v \not \in \Sigma$, 
$\psi_v, \chi_{V, v}, \chi_{W, v}$ are unramified, $E_v \slash F_v, $ is unramified or split, 
$\varphi$ is $G(\mathcal{O}_v)$-invariant, $G(F_v) \simeq \mathbb{G}(F_v)$, 
$e_\lambda \in L(\mathcal{O}_v)$ and $\phi_v$ is the characteristic function of $V(\mathcal{O}_v)^n$.
As in \cite{FM1}, the decomposition in \cite[Lemma~2.4]{Su} is applicable to our case since 
\[
\mathrm{U}(L)(F_v) \simeq \mathbb{G}_1(F_v) \simeq \{ (g, a) \in \mathrm{GL}_2(F_v) \times E_v \colon \det g \cdot N_{E_v \slash F_v}(a) =1\}
\]
when $v \not \in \Sigma$.
Hence, we obtain
\begin{multline*}
I_{\Sigma} = 
\int_{S^{\prime \prime}(\mA_\Sigma)} \int_{T_n(\mA_\Sigma)} \int_{(\mA_\Sigma \otimes E)^{n-1}}
\\
   \omega_{\psi, \chi}(r_1(a_1)^{-1} \cdots r_{n-1}(a_{n-1})^{-1} ut , 1)
  \phi(e_{-1}, \dots,e_{-n+1}, e_\lambda)  
  \\
  \times B_{e_\lambda, \psi, \chi_V^{-1}}(\pi(r_1(a_1)^{-1} \cdots r_{n-1}(a_{n-1})^{-1} ut) \varphi) \,da_1 \dots da_{n-1}\, dt \, du
\end{multline*}
where $T_n$ is the subgroup of $\mathbb{G}$ consisting of diagonal matrices.
By an argument similar to that in  the proof of \cite[Proposition~2]{FM1} with the action \eqref{act r_i} of $r_i(a_i)$, 
the  integrand of this integral is supported on $(t, u, (a_1, \dots, a_{n-1})) \in T_n(\mathcal{O}_\Sigma) \times S^{\prime \prime}(\mathcal{O}_\Sigma) \times 
(\mathcal{O}_{E, \Sigma})^{n-1}$
with $\mathcal{O}_\Sigma = \prod_{v \not \in \Sigma}\, \mathcal{O}_v$ and $\mathcal{O}_{E, \Sigma}= \prod_{v \not \in \Sigma}\, \mathcal{O}_{E, v}$.
Thus
\begin{align}\label{I_Sigma =B}
I_{\Sigma}&=
\int_{S^{\prime \prime}(\mathcal{O}_\Sigma)} \int_{T_n(\mathcal{O}_\Sigma)} \int_{(\mathcal{O}_{E, \Sigma})^{n-1}}
\\
  & \qquad\omega_{\psi, \chi}(r_1(a_1)^{-1} \cdots r_{n-1}(a_{n-1})^{-1}t u , 1)
  \phi(e_{-1}, \dots,e_{-n+1}, e_\lambda)  \notag
  \\
&\quad\times B_{e_\lambda, \psi, \chi_V^{-1}}(\pi(r_1(a_1)^{-1} \cdots r_{n-1}(a_{n-1})^{-1}tu) \varphi) \,da_1 \dots da_{n-1} \, dt \, du
\notag
  \\
&=
B_{e_\lambda, \psi, \chi_V^{-1}}(\varphi) \ne 0.
\notag
\end{align}
\end{proof}
%
%
%
%
%
%
%
%
%
%
%
%
%
%
%
%
%
%
%
%
\section{Proof of Theorem~\ref{period to L-value thm} and Theorem~\ref{opp GGP}}
\label{s:proof thm}
%
\subsection{Proof of Theorem~\ref{period to L-value thm}}
First we note that our argument is similar to the proof of \cite[Theorem~1]{FM1}.

The following proposition which is a unitary group case analogue of \cite[Proposition~2.3]{MR} (cf.  (\cite[Proposition~9.4]{GS}) holds,
by a  computation similar to that in the proof of Theorem~\ref{pullback whittaker prp} (1).
Indeed, the computation there
remains valid for the theta lift in the opposite direction.
%
\begin{proposition}
Let $k$ be a non-archimedean local field of characteristic zero
and $l$ be a quadratic extension of $k$.
Let $\tau$ be a non-trivial additive character of $k$, and $\nu_1, \nu_2$ be characters of $l^\times$
such that $\nu_i|_{k^\times} = 1$. Let $\omega_{\tau, (\nu_1, \nu_2)}^{(n,m)}$ be the Weil representation
of $\mathbb{G}_n^+(k) \times \mathbb{G}_m^-(k)$ corresponding to $\tau$ and $(\nu_1, \nu_2)$.
Let $\mathcal{U}_n$ be the group of
upper triangular unipotent matrices in $\mathbb{G}_n^+(k)$,
 which is the unipotent radical of 
a Borel subgroup of $\mathbb{G}_n^+(k)$,
and $\mu$ be a non-degenerate character of $\mathcal{U}_n$.

Suppose that $m<n$. Then as a representation of $\mathbb{G}_m^-(k)$, we have
\[
\left( \omega_{\tau, (\nu_1, \nu_2)}^{(n,m)} \right)_{\mathcal{U}_n, \mu} = 0
\]
where the left-hand side denotes the twisted Jacquet module for $(\mathcal{U}_n, \mu)$.

\end{proposition}
From this proposition, the following vanishing result on the
local theta correspondence readily follows. 
Note that it also should follow from Baki\'{c}
and Hanzer~\cite{BH}.
%
\begin{corollary}
\label{vanishing cor}
If $\pi$ is a $\mu$-generic irreducible representation of $\mathbb{G}_n^+(k)$,
then $\pi$ does not participate in the theta correspondence with $\mathbb{G}_m^-(k)$ for $m <n$.
\end{corollary}
A similar proposition holds in the split case 
thanks to
the analogous
 computations by Watanabe~\cite[Proposition~1]{TW}.
\begin{proposition}
Let $\omega_{\tau}^{(n,m)}$ be the Weil representation
of $\mathrm{GL}_n(k) \times \mathrm{GL}_m(k)$.
Suppose that $m<n$. Then as a representation of $\mathrm{GL}_m(k)$, we have
\[
\left( \omega_{\tau}^{(n,m)} \right)_{N_{n, s}, \tau_{N_{n, s}}} = 0
\]
where $N_{n, s}$ denotes the group of upper triangular unipotent matrices in $\mathrm{GL}_n(k)$,
$\tau_{N_{n,s}}$ is a non-degenerate character of $N_{n, s}$, and 
the left-hand side denotes the twisted Jacquet module for $(N_{n, s}, \tau_{N_{n, s}})$.
\end{proposition}
\begin{corollary}
\label{vanishing cor split}
If $\pi$ is a generic irreducible representation of $\mathrm{GL}_n(k)$,
then $\pi$ does not participate in the theta correspondence with $\mathrm{GL}_m(k)$ for $m <n$.
\end{corollary}

Before proceeding to a proof of Theorem~\ref{period to L-value thm},
we recall the following local-global criterion concerning
the  non-vanishing of theta correspondences.
\begin{theorem}[Theorem~10.1 in \cite{Yam}]
\label{yam thm}
Let $\pi$ be an irreducible cuspidal automorphic representation of $G(\mA)$ with $G \in \mathcal{G}_n$.
Suppose that the theta lift $\Theta_{V, \mathbb{W}_m}(\pi, \psi, \chi)$ of $\pi$ to $\mathbb{G}_{m}^-(\mA)$ is zero for any $m <n$.

Then  $\Theta_{V, \mathbb{W}_n}(\pi, \psi, \chi)$ is non-zero if and only if the following two conditions hold:
\begin{enumerate}
\item $L\left( \frac{1}{2}, \pi \times \chi_V^{-1}\right) \ne 0$.
\item  For any place $v$ of $F$, the local theta lift $\theta_{V, \mathbb{W}_n}(\pi_v, \psi_v, \chi_v)$ is non-zero.
\end{enumerate}
\end{theorem}
%
%
%
%
Now we are ready to prove Theorem~\ref{period to L-value thm}.
\begin{proof}[Proof of Theorem~\ref{period to L-value thm}]
As a pair of characters to define the splitting, we may take 
\[
\chi_\Lambda^\Box \coloneqq (\Lambda^{-1}, \Lambda^{-1}).
\]
Since $\pi_w$ is generic at some finite place $w$ of $F$, local theta lift  $\theta_{V, \mathbb{W}_m}(\pi_w, \psi_w, \chi_{\Lambda, w}^\Box)$ is zero for $m <n$
by Corollary~\ref{vanishing cor} and Corollary~\ref{vanishing cor split}. Hence, the global theta lift $\Theta_{V, \mathbb{W}_m}(\pi, \psi, \chi_{\Lambda}^\Box)$  
is also zero for $m<n$.
On the other hand, since $\pi$ has the
$(e, \psi, \Lambda)$-Bessel period, 
$\Theta_{V, \mathbb{W}_n}(\pi, \psi, \chi_{\Lambda}^\Box)$ has the $\psi_{N_n, \lambda}$-Whittaker period
by Proposition~\ref{pullback whittaker prp}
where $\lambda = (e, e)_V$.
Hence in particular, 
$\Theta_{V, \mathbb{W}_n}(\pi, \psi, \chi_{\Lambda}^\Box) \ne 0$.
Therefore, by Theorem~\ref{yam thm}, we obtain $L\left( \frac{1}{2}, \pi \times \Lambda \right) \ne 0$ .
\end{proof}
\subsection{Proof of Theorem~\ref{opp GGP}}
\label{proof of theorem ggp}
In this section, we shall prove Theorem~\ref{opp GGP}.
First we recall here the claim of  the condition (2)
in Theorem~\ref{opp GGP}.
%
\begin{Claim}
\label{opp GGP lem}
There exists  a pair $\left(G^\prime, \pi^\prime\right)$ where
$G^\prime \in \mathcal{G}_n$ and $\pi^\prime$ 
an irreducible 
cuspidal tempered automorphic representation 
of $G^\prime(\mA)$
such that $\pi^\prime$ is nearly equivalent to $\pi$
and $\pi^\prime$ admits the $(e, \psi, \Lambda)$-Bessel period.
\end{Claim}
%
First, we show that the condition (2) implies the condition (1).
We note that $\pi$ and $\pi^\prime$ in Claim
have the same $A$-parameter which may be proved in the same manner as Atobe and Gan~\cite[Proposition~7.4]{AG2}.
Moreover, since $\pi$ is tempered, the $A$-packet containing $\pi$ coincides with the $L$-packet of the associated $L$-parameter,
and thus $\pi$ and $\pi^\prime$ have the same local $L$-parameter.
Therefore, we find that 
\[
L \left(s, \pi \times \Lambda \right) = L \left(s, \pi^\prime \times \Lambda \right).
\]
As in \cite[Remark~2]{FM1}, when an irreducible cuspidal automorphic representation is tempered,
then its local component is  generic at some finite place.
Hence, by Theorem~\ref{period to L-value thm}, we obtain the condition (1), i.e., 
\[
L \left(\frac{1}{2}, \pi \times \Lambda \right) = L \left(\frac{1}{2}, \pi^\prime \times \Lambda \right) \ne 0.
\]
By the same argument, the condition (b) implies the condition (a).

Then our remaining task is to prove the implications (1) $\Longrightarrow$ (2) 
and (a) $\Longrightarrow$ (b).

By the following proposition, it is sufficient to 
prove the implication (1) $\Longrightarrow$ (2).
%
%
%
\begin{proposition}
\label{opp GGP prp}
Suppose that the implication (1) $\Longrightarrow$ (2) holds.

Then  the implication (a) $\Longrightarrow$ (b) holds.
\end{proposition}
\begin{proof}
%
Suppose that the condition (a) holds.
Then the condition (1) holds and hence the condition (2) also holds.
As we remarked above, $\pi$ and $\pi^\prime$ have the same $L$-parameter.
When $E_v \simeq F_v \oplus F_v$, then $G(F_v)  \simeq G^\prime(F_v) \simeq \mathrm{GL}_{2n}(F_v)$ and thus $\pi_v \simeq \pi_v^\prime$
because all $L$-packets of $ \mathrm{GL}_{2n}(F_v)$ are singleton. Suppose that $E_v$ is a quadratic extension field of $F_v$.
Then by Beuzart-Plessis~\cite{BP1, BP2}, we know that in each Vogan local $L$-packet,
only one element has the given Bessel model. This shows that $G(F_v) \simeq G^\prime(F_v)$ and $\pi_v \simeq \pi_v^\prime$. 
Hence, we obtain $G=G^\prime$ and $\pi \simeq \pi^\prime$.
Finally we have $\pi = \pi^\prime$
since the multiplicities of $\pi, \pi^\prime$ in the generic part of the discrete spectrum 
are one by  \cite[Theorem~2.5]{CZ} (see also \cite{KMSW, Mok}).
\end{proof}
%
Thus we are reduced to proving the implication
(1) $\Longrightarrow$ (2).

We shall use an argument similar to the proof
of \cite[Proposition~5]{FM1}
and the  epsilon dichotomy for the local theta correspondence of the unitary dual pairs in the equal rank case.
Konno and Konno~\cite{KK} studied a relation between non-vanishing of local theta correspondence and 
$\varepsilon$-factor, which essentially follows from the computation by Paul~\cite{Pau} 
(see Remark~\ref{IP rem} as for the comparison between our local theta lifts
and those in Paul~\cite{Pau}). 
For our purpose, we shall reformulate their result as follows.
%
%
%
\begin{proposition}
\label{dicho archi}
Let $V_{p,q}$ (resp. $V_{p^\prime,q^\prime}$) be a
skew-Hermitian (resp. Hermitian) space over $\mC$ with 
a skew-Hermitian
(resp. Hermitian ) form
of signature $(p,q)$ (resp. $(p^\prime,q^\prime)$) and $p+q=2n$ (resp. $p^\prime+q^\prime=2n$).
Let $\sigma$ be an irreducible admissible representation of 
$\mathrm{U}(V_{p,q})$,
and $\tau$  a non-trivial additive character of $\mR$.
Denote $\tau_\mC = \tau \circ \mathrm{tr}_{\mC \slash \mR}$.
Let $\xi_1$ and $\xi_2$ be characters of $\mC^\times$ such that $\xi_i = \chi_{m_i}$
with $m_1 \equiv m_2 \equiv 0$ (mod $2$).
Here for $a\in\mathbb Z$, $\chi_a$ 
is a character of $\mathbb C^\times$ defined by
\[
\chi_{a}(z) = \bar{z}^{-a} (z \bar{z})^{a\slash 2}.
\]
We note that $\chi_a\mid_{\mathbb R^\times}
=\left(\mathrm{sgn}\right)^a$.

Suppose that
 the theta lift $\theta_{\tau, (\xi_1, \xi_2)}(\sigma)$ of $\sigma$ to $\mathrm{U}(V_{p^\prime,q^\prime})$ is non-zero.
 Then we have
\begin{equation}\label{e: archimedean dichotomy}
\varepsilon \left(\frac{1}{2}, \sigma \times \xi_1^{-1}, \tau_\mC \right)
= \omega_{\sigma}(-1) \xi_1(i)^{2n} \varepsilon(V_{p,q}) \varepsilon(V_{p^\prime, q^\prime})
\end{equation}
where $\varepsilon(V_{p,q})$ and $ \varepsilon(V_{p^\prime,q^\prime})$ denote the Hasse invariants of $V_{p,q}$ and $V_{p^\prime,q^\prime}$, respectively, 
 and $\omega_{\sigma}$ denotes the central character of $\sigma$.
\end{proposition}
%
\begin{proof}
First we note that the identity \eqref{e: archimedean dichotomy}
in the general case readily follows from the one in
the case of (limit of) discrete series by 
\cite[Theorem~6.1]{Pau} and 
the multiplicativity of $\varepsilon$-factors.

Thus suppose that $\sigma$ is limit of discrete series.
We may take the Harish-Chandra parameter of $\sigma$ so that it is of the form
\begin{multline*}
\left(\frac{m_1}{2}, \dots, \frac{m_1}{2} ; \frac{m_1}{2}, \dots, \frac{m_1}{2} \right) 
\\
-([a_1]_{k_1}, \dots, [a_r]_{k_r}, [b_1]_{\ell_1}, \dots, [b_s]_{\ell_s}; [a_1]_{\bar{k}_1}, \dots, [a_r]_{\bar{k}_r}, [b_1]_{\bar{\ell}_1}, \dots, [b_s]_{\bar{\ell}_s})
\end{multline*}
where $(\sum_{x=1}^r k_i )+ (\sum_{y=1}^s\ell_y)=p$, $(\sum_{x=1}^r \bar{k}_i )+ (\sum_{y=1}^s\bar{\ell}_y)=q$,
$[a]_k$ denotes $k$-tuple $(a, \dots, a)$ and  $a_i, b_i \in \mZ + \frac{1}{2}$ satisfying 
\[
\varepsilon_{\tau}(a_i) >0, \quad  \varepsilon_{\tau} (b_i) <0
\]
and
\[
|a_1|> |a_2|>\cdots > |a_r|, \quad |b_1| > |b_2| > \cdots >|b_s|.
\]
Here, for $a \in \mZ+\frac{1}{2}$, we introduce the sign
\begin{equation}
\label{epsilon sgn}
\varepsilon_{\tau}(a) = \varepsilon\left( \frac{1}{2}, \chi_{2a}, \tau_\mC \right) \chi_{2a}(-i) =  \varepsilon_\tau  \mathrm{sgn}(a) 
\end{equation}
where we write $\tau(x) = e^{d_\tau  x}$ with $d_\tau \in i\, \mR$ and 
\[
\varepsilon_\tau =\frac{d_\tau}{|d_\tau|\, i} \in \{\pm 1\}.
\]
Then the Langlands parameter of $\sigma$
is given by
\[
\left( \oplus_{i=1}^r k_i \chi_{m_1-2a_i} \right) \oplus \left( \oplus_{i=1}^s \ell_i \chi_{m_1-2b_i} \right)
\oplus \left( \oplus_{i=1}^r \bar{k}_i \chi_{m_1-2a_i} \right)
\oplus \left( \oplus_{i=1}^s \bar{\ell}_i \chi_{m_1-2b_i}  \right)
\]
(for example, see Atobe~\cite[Theorem~2.1]{HA}).
Hence, we have 
\[
\varepsilon(s, \sigma \times \xi_1^{-1}, \tau_\mC)  = \prod_{i=1}^r  \varepsilon(s,  \chi_{-2a_i}, \tau_\mC)^{k_i+\bar{k}_i} 
 \prod_{i=1}^s  \varepsilon(s,  \chi_{-2b_i}, \tau_\mC)^{\ell_i+\bar{\ell}_i}.
\]
Hence, by \eqref{epsilon sgn}, we obtain
\begin{align}
\label{dicho pf 1}
&\varepsilon \left(\frac{1}{2}, \sigma \times \xi_1^{-1}, \tau_\mC \right) 
\\
=& \prod_{x=1}^r (\chi_{2a_x}(i)  \varepsilon_\tau(a_x) )^{k_x+\bar{k}_x} \times
 \prod_{y=1}^s (\chi_{2b_y}(i)  \varepsilon_\tau(b_y) )^{\ell_y+\bar{\ell}_y}\notag
 \\
 =&\prod_{x=1}^r \chi_{2a_x}(i)^{k_x+\bar{k}_x}   \prod_{y=1}^s \chi_{2b_y}(i)^{\ell_y+\bar{\ell}_y}
 \prod_{y=1}^s (-1)^{\ell_y+\bar{\ell}_y}.\notag
 \end{align}
 \color{black}
Moreover, we have
\begin{align}
\label{dicho pf 2}
\omega_\sigma(-1) &=\prod_{x=1}^r \chi_{m_1-2a_x}(i)^{k_x+\bar{k}_x}    \prod_{y=1}^s \chi_{m_1-2b_y}(i)^{\ell_y+\bar{\ell}_y}
\\
&=
i^{2m_1 n} 
\prod_{x=1}^r \chi_{2a_x}(i)^{k_x+\bar{k}_x}   \prod_{y=1}^s \chi_{2b_y}(i)^{\ell_y+\bar{\ell}_y}
\notag
\\
&=
\xi_1(i)^{2n} 
\prod_{x=1}^r \chi_{2a_x}(i)^{k_x+\bar{k}_x}   \prod_{y=1}^s \chi_{2b_y}(i)^{\ell_y+\bar{\ell}_y}.
\notag
\end{align}
For simplicity, we set 
\[
k=\sum_{x=1}^r k_x, \, \bar{k}=\sum_{x=1}^r \bar{k}_x, \,
\ell= \sum_{y=1}^s \ell_y, \, \bar{\ell} = \sum_{y=1}^s \bar{\ell}_y.
\]
Then by \cite[Theorem~6.1]{KK}, we know that the theta lift $\theta_{\tau, (\xi_1, \xi_2)}(\sigma)$ is non-zero if and only if 
\[
p^\prime =k+ \bar{\ell}.
\]
On the other hand, from the definition, 
we have $p=k+\ell$.
Therefore, if $p^\prime = k+\overline{\ell}$, then we have
\begin{align}
\label{dicho pf 3}
 &\prod_{y=1}^s (-1)^{\ell_i+\bar{\ell}_i} = (-1)^{\ell+\bar{\ell}} = (-1)^{k+\ell} (-1)^{k+\bar{\ell}}
 =(-1)^{n-p} (-1)^{n-p^\prime}
 \\
  = &\varepsilon(V_{p,q}) \,\varepsilon(V_{p^\prime, q^\prime})
  \notag
\end{align}
since $(-1)^{n-p} =  \varepsilon(V_{p,q})$ and $(-1)^{n-p^\prime} = \varepsilon(V_{p^\prime, q^\prime})$.
Then our claim for $\sigma$ readily follows from \eqref{dicho pf 1}, \eqref{dicho pf 2} and \eqref{dicho pf 3}. 
%
\end{proof}
\color{black}
\begin{proof}[Proof of (1) implies (2)]
Suppose that
\[
L\left(\frac{1}{2},\pi\times \Lambda\right)\ne 0.
\]

By Mok~\cite{Mok} and Kaletha, Minguez, Shin and White~\cite{KMSW}, we have the base change lift of $\pi$
to $\mathrm{GL}_{2n}(\mA_E)$. Then the global descent method by Ginzburg, Rallis and Soudry~\cite{GRS} shows that 
we obtain a globally 
$\psi_{N_n, (e,e)_V}$-generic
irreducible cuspidal automorphic representation $\pi^\circ$ of $\mathbb{G}_n^-(\mA)$
which is nearly equivalent to $\pi$ since $\mathbb{G}_n^- \simeq \mathbb{G}_n$ by \eqref{pm iso}.
Then $\pi$ and $\pi^\circ$ have the same $A$-parameter
as in the proof of Proposition~\ref{opp GGP prp},
and thus $\pi$ and $\pi^\circ$ has the same $L$-parameter and we obtain $L \left(\frac{1}{2}, \pi^\circ \times \Lambda \right) \ne 0$.

Let us take $\xi \in E^\times$ such that $\bar{\xi} = -\xi$.
Let $v$ be a place of $F$ which is non-split in $E$.
Let $V^\prime_v$ be a 
$2n$-dimensional 
Hermitian space over $E_v$ such that 
the local theta lift $\theta(\pi_v^\circ, \psi_v^{-1}, \chi_{\Lambda_v^{-1}}^\Box)$ of $\pi^\circ_v$ to $\mathrm{U}(V^\prime_v)$ is not zero.
Note that such $V_v^\prime$ exists because of Gan and Ichino~\cite[Theorem~11.1]{GI1}
and Paul~\cite[Theorem~0.1]{Pau} (cf \cite{HA}).
Then by Harris, Kudla and Sweet~\cite[Theorem~6.1]{HKS} and Gan
and Ichino~\cite[Theorem~11.1]{GI1} when $v$ is finite and 
by Proposition~\ref{dicho archi} when $v$ is real, we have
%
\begin{equation}
\label{dicho eq 1}
\varepsilon \left(\frac{1}{2}, \pi^\circ_v \times \Lambda_v, \psi_{F_v}^{-1} \right)
= \omega_{\pi_v^\circ}(-1) \Lambda_v(\xi)^{-2n} \varepsilon(V_v^\prime) \varepsilon(\mathbb{W}_v)
\end{equation}
\color{black}
where $\varepsilon(V_v^\prime)$ and $\varepsilon(\mathbb{W}_v)$ denote the Hasse invariants of $V_v^\prime$ and $\mathbb{W}$, respectively.

Let $v$ be a place of $F$ which splits in $E$. In this case, we put $ \varepsilon(V_v^\prime)= \varepsilon(\mathbb{W})=1$.
Then $\pi_v^\circ$ is an irreducible representation of $\mathrm{GL}_{2n}(F_v)$,
and the character $\Lambda_v$ is a character of $F_v^\times \times F_v^\times$.
Since $\Lambda_v((a, a))=1$, we may write this character as $(\mu, \mu^{-1})$ with a character $\mu$ of $F_v^\times$.
Then we know that the theta lift of $\pi^\circ_v$ to $\mathrm{GL}_{2n}(F_v)$ is isomorphic to $\pi^\circ_v$, in particular non-zero.
Further, by the local functional equation of local $L$-factors, we have
\[
\varepsilon \left(\frac{1}{2}, \pi^\circ_v \times \mu, \psi_{F_v}^{-1} \right)
\varepsilon \left(\frac{1}{2}, (\pi^\circ_v)^\vee \times \mu^{-1}, \psi_{F_v}^{-1} \right)
= \omega_{\pi_v^\circ}(-1).
\]
Since we have 
\[
\varepsilon \left(\frac{1}{2}, \pi_v^\circ, \times \Lambda_v, \psi_{F_v}^{-1} \right)=
\varepsilon \left(\frac{1}{2}, \pi^\circ_v \times \mu, \psi_{F_v}^{-1} \right)
\varepsilon \left(\frac{1}{2}, (\pi^\circ_v)^\vee \times \mu^{-1}, \psi_{F_v}^{-1} \right),
\]
we get
\begin{equation}
\label{dicho eq 2}
\varepsilon \left(\frac{1}{2}, \pi^\circ_v \times \Lambda_v, \psi_{F_v}^{-1} \right)
= \omega_{\pi_v^\circ}(-1) \mu(-1)^{2n} \varepsilon(V_v^\prime) \varepsilon(\mathbb{W}_v).
\end{equation}

Since $L \left(\frac{1}{2}, \pi^\circ \times \Lambda \right) \ne 0$,
we should have $\prod_v \varepsilon \left(\frac{1}{2}, \pi^\circ_v \times \Lambda_v, \psi_v^{-1} \right) =1$.
Thus, by \eqref{dicho eq 1} and \eqref{dicho eq 2}, we get
\[
\prod_v \varepsilon(V_v^\prime) =1.
\]
By the Hasse principle, there is a Hermitian space $V^\prime$ over $E$
such that $V^\prime \otimes F_v$ is isomorphic to $V_v^\prime$.
Then we put $G^\prime = \mathrm{U}(V^\prime)$.
Let $v$ be a non-split real place of $F$.
Then by \cite[Theorem~0.1, 6.1 (a)]{Pau} (see also \cite[Theorem~3.13]{HA}),
$G^\prime(F_v)$ should be one of
\[
\mathrm{U}(n, n), \quad \mathrm{U}(n+1, n-1), \quad  \text{or} \quad \mathrm{U}(n-1, n+1).
\]
Moreover, when $v$ is a finite place which is non-split in $E$,
it is clear that the Witt index of $V_v^\prime$ is at least $n-1$.
Hence, we should have $G^\prime \in \mathcal{G}_n$.
Let us denote by $\Pi =\Theta(\pi^\circ, \psi^{-1}, \chi_{\Lambda^{-1}}^\Box)$
the theta lift of $\pi^\circ$ to $\mathrm{U}(V^\prime)$.
From our choice of $V_v^\prime$, the local theta lifts $\theta(\pi^\circ_v, \psi_v^{-1}, \chi_{\Lambda_v^{-1}}^\Box) \ne 0$
and we have $L \left(\frac{1}{2}, \pi^\circ \times \Lambda \right) \ne 0$. Then by Theorem~\ref{yam thm}, we have $\Pi \ne 0$.
Further, from Corollary~\ref{vanishing cor split} and Rallis' tower property,
$\Pi$ should be cuspidal.

Let us consider the theta lift $\Theta(\Pi, \psi, \chi_{\Lambda}^\Box)$ of $\Pi$
to $\mathbb{G}_n^-(\mA)$.  We note that at a split place $v$, we know that $\Pi_v \simeq \pi^\circ_{v}$ (for example, see \cite[Th\'{e}or\`{e}m~1]{Min}).
Hence we have
\begin{equation}\label{e:gamma identity}
\gamma(s, \Pi_v \times \Lambda_v, \psi_v) = \gamma(s, \pi_v^\circ \times \Lambda_v, \psi_v).
\end{equation}
Moreover
Gan and Ichino~\cite[Theorem~11.5]{GI1} proved \eqref{e:gamma identity}
 for every finite place $v$ of $F$.
Since we define local $L$-factors $L\left(s, \Pi_v \times \Lambda_v \right)$ and $L \left(s, \pi^\circ_v \times \Lambda_v \right)$
using the $\gamma$-factors as in \cite{LR2}, we have 
\begin{equation}
\label{L match}
L\left(s, \Pi_v \times \Lambda_v \right)=L \left(s, \pi^\circ_v \times \Lambda_v \right)
\end{equation}
at split places and finite places.
At real non-split places,  we have an explicit description of local theta lift 
by Paul~\cite{Pau} (cf. Ichino~\cite[Theorem~4.1, Lemma~8.1]{Ich20}).
Here we refer to Remark~\ref{IP rem} for the notational difference with \cite{Ich20,Pau}.
Then by Adams and Vogan~\cite{AV},
we see that $\pi_v^\circ$ and $\Pi_v$ have the same local $L$-parameter
and   \eqref{L match} holds also at real non-split places
since our local $L$-factors coincide with the $L$-factors defined by local Langlands parameters by \cite{Yam}.
Thus we obtain
\[
L\left(\frac{1}{2}, \Pi \times \Lambda \right) = L \left(\frac{1}{2}, \pi^\circ \times \Lambda \right) \ne 0.
\]
By \cite[Proposition~13.20, 13.22, 15.10]{GKT} and the Howe duality, 
$\theta(\Pi_v, \psi_v, \chi_{\Lambda_v}^\Box) \simeq \pi_v^\circ$.
In particular, $\theta(\Pi_v, \psi_v, \chi_{\Lambda_v}^\Box)$ is non-zero. 
Then by Theorem~\ref{yam thm}, we see that $\Theta(\Pi, \psi, \chi_{\Lambda}^\Box) \ne 0$,
and it is cuspidal by Corollary~\ref{vanishing cor split} and Rallis' tower property as above.
Further, by Wu~\cite[Theorem~5.1]{Wu}, we have $\Theta(\Pi, \psi, \chi_{\Lambda}^\Box) = \pi^\circ$.

Therefore, by Proposition~\ref{pullback whittaker prp}, $\Pi$ should have the $(e, \psi, \Lambda)$-Bessel period.
As we noted above, $\Pi_v \simeq \pi^\circ_v$ at split place $v$. Moreover, 
we have $\Pi_v \simeq \pi^\circ_v$ by Kudla~\cite{Ku86} when $v \nmid 2$
and $\Pi_v, \Lambda_v, \psi_v$ are unramified.
Hence, $\Pi$ is nearly equivalent to $\pi^\circ$ and to $\pi$.
Finally, we note that as in the proof of (2)$\Longrightarrow$(1), we see that $\Pi$ and $\pi$ have the same local $L$-parameter.
In particular, $\Pi$ has a generic $A$-parameter.
Hence, as $G^\prime$ and $\pi^\prime$ in Theorem~\ref{opp GGP}, we may take $\mathrm{U}(V^\prime)$ and $\Pi$.
\end{proof}
%
%
%
%
%
%
%
%
%
%
%
%
%
%
%
%
%
%
%
%
%
%
%
%
%
%
%
%
%

%
%
%
%
%
%
%
%
%
%
%
%
%
%
%
%
%
%
%
%
\section{Refined Gan--Gross--Prasad conjecture for $(\mathrm{U}(2n), \mathrm{U}(1))$}
\label{s:proof ref}
In this section, we prove Theorem~\ref{refined ggp thm}, namely we prove the refined Gan--Gross--Prasad conjecture
 for Bessel periods for $(\mathrm{U}(2n), \mathrm{U}(1))$
under certain assumptions. Throughout this section, 
we set $G = \mathrm{U}(V) \in \mathcal{G}_n$.
\begin{lemma}
For an irreducible cuspidal tempered automorphic representation $\pi$ of $G(\mA)$ with $G \in \mathcal{G}_n$
such that $B_{e, \psi, \Lambda}\equiv 0$  on $V_\pi$, Theorem~\ref{refined ggp thm} holds.
\end{lemma}
\begin{proof}
This is proved by an   argument 
similar to  the proof of \cite[Corollary~1]{FM2},
using the
local Gan--Gross--Prasad conjecture by \cite{BP1, BP2} and the endoscopic classifications by \cite{KMSW, Mok}.
\end{proof}
%
%
%
Hence we  assume that $B_{e, \psi, \Lambda}\not\equiv 0$  on 
$V_\pi$.
We prove Theorem~\ref{refined ggp thm} by 
 an argument similar to the proof of \cite[Theorem~1]{FM2}.
As stated in the beginning of \cite[Section~2.4]{FM2}, 
 it suffices for us to prove \eqref{formula uni}
for $\varphi = \otimes_v\, \varphi_v \in V_\pi$ such that $B_{e, \psi, \Lambda}(\varphi) \ne 0$
because of \cite[Theorem~5]{BP2} and the multiplicity one theorem \eqref{uniqueness}.

In this case, by Proposition~\ref{pullback whittaker prp}, $\sigma \coloneqq \Theta_{V, \mathbb{W}_n}(\pi, \psi, \chi_\Lambda^\Box)$
is a $\psi_{N_n, \lambda}$-generic irreducible cuspidal automorphic representation of $\mathbb{G}_n^-(\mA)$
where $\lambda=\left(e,e\right)_V$.
Then we have $\sigma=\otimes_v\,\sigma_v$ with $\sigma_v = \theta(\pi_v, \psi_v, \chi_{\Lambda_v}^\Box)$ by the Howe duality.
Let 
\[
\theta_{\psi_v, \chi_{\Lambda_v}^\Box}\colon\mathcal S\left(V\left(F_v\right)^n\right)\otimes V_{\pi_v}\to
V_{\sigma_v}
\]
be a non-zero $G(F_v)\times\mathbb{G}_n^-\left(F_v\right)$-equivariant
linear map, which is unique up to multiplication by a scalar.
Since the mapping
\[
\mathcal S\left(V\left(\mA\right)^n\right)\otimes V_\pi
\ni\left(\phi^\prime,\varphi^\prime\right)\mapsto
\theta_{\psi, \chi_\Lambda^\Box}^{\phi^\prime}\left(\varphi^\prime\right)\in V_\sigma
\]
is $G(F_v)\times \mathbb{G}_n^-\left(F_v\right)$-equivariant  at any place $v$,
by the uniqueness of $\theta_{\psi_v, \chi_{\Lambda_v}^\Box}$, we may adjust $\left\{\theta_{\psi_v, \chi_{\Lambda_v}^\Box}\right\}_v$ so that 
\[
\theta_{\psi, \chi_\Lambda^\Box}^{\phi^\prime}\left(\varphi^\prime\right)=\otimes_v\,
\theta_{\psi_v, \chi_{\Lambda_v}^\Box}\left(\phi_v^\prime\otimes\varphi_v^\prime\right)
\]
for $\varphi^\prime=\otimes_v\,\varphi_v^\prime\in V_\pi$ \, $\phi^\prime=\otimes_v\, \phi_v^\prime
\in\mathcal S\left(V\left(\mA\right)^n\right)$.

%
%
%
%
\subsection{Rallis inner product formula}
Fix a polarization 
$\mathbb{W} = \mathbb{W}_n \coloneqq Y_n^+ \oplus Y_n^-$. 
Then we realize the Weil representation $\omega_{\psi, \chi_\Lambda^\Box}$ of 
$\mathrm{U}(V)(\mA) \times \mathrm{U}(\mathbb{W}_n)(\mA)$
on the space $\mathcal{S} \left( \left(V \otimes Y_n^+ \right) (\mA) \right)$.
Let $V^\Box$ be the Hermitian space $V\oplus \left(-V\right)$,
i.e. 
$V^\Box$ is a direct sum $V\oplus V$
as a vector space
and its  Hermitian form $\left(\, \,, \,\,\right)_{V^\Box}$
on $V^\Box$
is defined by
\[
\left(v_1\oplus v_2,v_1^\prime\oplus v_2^\prime\right)_{V^\Box}
\coloneqq\left(v_1,v_1^\prime\right)_{V}-\left(v_2,v_2^\prime\right)_V.
\]
Let $V^\Box_{\pm}$ be maximal isotropic subspaces of $V^\Box$
defined by
\[
V^\Box_+\coloneqq\left\{v\oplus v\in V^\Box \colon v\in V\right\}
\quad
\text{and}
\quad
V^\Box_-\coloneqq\left\{v\oplus -v\in V^\Box \colon v\in V\right\},
\]
respectively.
\color{black}
We note that there is a natural embedding
\[
\iota\colon\mathrm{U}\left(V\right)\times\mathrm{U}\left(-V\right)
\hookrightarrow
\mathrm{U}\left(V^\Box\right)
\quad
\text{such that}\quad 
\iota\left(g_1,g_2\right)\left(v_1\oplus v_2\right)=
g_1v_1\oplus g_2 v_2.
\]
For $\phi \in \mathcal{S}\left(\left(V^\Box\otimes Y_n^+\right)\left(\mA\right)\right)$,
its partial Fourier transform $\hat{\phi} \in \mathcal{S}\left(\left(V^\Box_+ \otimes \mathbb{W}\right)\left(\mA\right)\right)$ 
is defined by
\[
\hat{\phi}(u \oplus v)\coloneqq \int_{(V^\Box_- \otimes Y_n^+)(\mA)} \phi(x \oplus u) \psi (\langle x, v \rangle ) \, dx
\]
where $u \in (V^\Box_+ \otimes Y_n^+)(\mA)$ and $v \in (V^\Box_+ \otimes Y_n^-)(\mA)$, and $\langle \,, \, \rangle$
denotes the pairing on $(V^\Box_- \otimes Y_n^+)(\mA) \times (V^\Box_+ \otimes Y_n^-)(\mA)$ given by the pairing on $V^\Box$ and $\mathbb{W}$.
Then there exists a $\mathrm{U}(V, \mA) \times \mathrm{U}(-V, \mA)$-intertwining map
\[
\tau\colon \mathcal{S}\left(\left(V \otimes Y^+_n\right)\left(\mA\right)\right) \hat{\otimes}\,
 \mathcal{S}\left(\left(\left(-V\right) \otimes Y^+_n\right)\left(\mA\right)\right)
 \rightarrow 
 \mathcal{S}\left(\left(V^\Box_+ \otimes \mathbb{W}\right)\left(\mA\right)\right)
\]
with respect to the Weil representations,
obtained by composing the natural map
\[
\mathcal{S}\left(\left(V \otimes Y^+_n\right)\left(\mA\right)\right) \hat{\otimes}\,
 \mathcal{S}\left(\left(\left(-V\right) \otimes Y^+_n\right)\left(\mA\right)\right)
 \rightarrow 
 \mathcal{S}\left(\left(V^\Box \otimes Y_n^+\right)\left(\mA\right)\right)
 \]
 with the partial Fourier transform
\[
\mathcal{S}\left(\left(V^\Box\otimes Y_n^+\right)\left(\mA\right)\right)
\rightarrow
\mathcal{S}\left(\left(V^\Box_+ \otimes \mathbb{W}\right)\left(\mA\right)\right).
\]
 Namely we have 
 \begin{align*}
 &\tau\left(\omega_{\psi, \chi_\Lambda^\Box, V, \mathbb{W}}\left(g_1\right)\phi_{+}
 \otimes\omega_{\psi, \chi_\Lambda^\Box, -V, \mathbb{W}}\left(g_2\right)\phi_{-}\right)
 \\
 =&\omega_{\psi, \chi_\Lambda^\Box, V^\Box, \mathbb{W}}\left(\iota\left(g_1,g_2\right)\right)
 \tau\left(\phi_{+}\otimes\phi_{-}\right)
 \end{align*}
 for $\left(g_1,g_2\right)\in\mathrm{U}(V, \mA) \times \mathrm{U}(-V, \mA)$
 and
 $\phi_{\pm}\in\mathcal{S}\left(\left(\left(\pm
 V\right) \otimes Y^+_n\right)\left(\mA\right)\right)$.
 We also consider the local counterparts of $\tau$.
 %
 %
 %
 \subsubsection{Local doubling zeta integrals}
 Let $P$ be the maximal parabolic subgroup of
 $\mathrm{U}\left(V^\Box \right)$
 defined as the stabilizer of the isotropic subspace
 $V_+^\Box$.
 Then the Levi subgroup of $P$
 is isomorphic to $\mathrm{Res}_{E \slash F} \,\mathrm{GL}\left(V\right)$.

 At each place $v$ of $F$, we consider the degenerate principal series
 representation 
 \[
 I_v\left(s\right)\coloneqq\mathrm{Ind}_{P\left(F_v\right)}^{
 \mathrm{U}\left(V^\Box, F_v\right)}\, \Lambda_v^{-1}\, |\,\cdot\,|_v^s
 \quad\text{for $s\in\mathbb C$.}
 \]
 Here the induction is normalized and
  $\Lambda_v^{-1}\, |\,\cdot\,|_v^s$ denotes the character of $P\left(F_v\right)$
given by $g_v \mapsto \Lambda_v^{-1}(\det g_v) |\det g_v |_v^s$ on its Levi subgroup 
  $\mathrm{Res}_{E \slash F}\mathrm{GL}\left(V\right) \left(F_v \right)$
  and trivial on its unipotent radical.
  
  For $\varphi_v,\varphi_v^\prime\in V_{\pi_v}$
  and $\Phi_v\in I_v\left(s\right)$, the local doubling zeta integral
    is defined by
  \begin{equation}\label{e: local doubling}
  Z_v\left(s,\varphi_v,\varphi_v^\prime,\Phi_v,\pi_v\right)\coloneqq
  \int_{G(F_v)}\langle\pi_v\left(g_v\right)\varphi_v,\varphi_v^\prime\rangle_v
  \,\Phi_v\left(\iota\left(g_v, I_v\right)\right)\,dg_v
 \end{equation}
 where $I_v$ denotes the unit element of $G(F_v)$.
 We recall that the integral~\eqref{e: local doubling} converges absolutely
 when $\mathrm{Re}\left(s\right)> -\frac{1}{2}$ by
 Yamana~\cite[Lemma~7.2]{Yam}
 since $\pi_v$ is tempered.
 
 For $\phi_v\in\mathcal{S}\left(\left(V^\Box_+ \otimes \mathbb{W}\right)
 \left(F_v\right)\right)$, we define $\Phi_{\phi_v}
 \in I_v\left(0\right)$
 by
 \begin{equation}
 \label{local section}
 \Phi_{\phi_v}\left(g_v\right)=
 \left(\omega_{\psi_v, \chi_{\Lambda_v}^\Box}\left(g_v\right)\phi_v\right)
 \left(0\right)
 \quad\text{for}\quad g_v\in\mathrm{U}\left(V^\Box, F_v\right).
 \end{equation}
 Then we define $Z_v^\circ\left(\varphi_v, \phi_v,\pi_v\right)$
 for $0 \ne \varphi_v\in V_{\pi_v}$ and $\phi_v \in
 \mathcal{S}\left(\left(V \otimes Y_n^+\right)
 \left(F_v\right)\right)$ by
 \begin{multline}\label{e: normalized doubling}
Z_v^\circ\left(\varphi_v, \phi_v,\pi_v\right)
\coloneqq
\frac{\prod_{j=1}^{2n} L\left(j, \chi_{E, v}^j \right)}{L\left(1\slash 2, \pi_v \times \Lambda_v \right)}
\cdot \frac{1}{\langle\varphi_v,\varphi_v\rangle_v}
\\
\times
Z_v\left(0, \varphi_v,\varphi_v,\Phi_{\tau_v\left(\phi_v\otimes \phi_v\right)},
\pi_v\right).
\end{multline}
%
%
%
\subsubsection{Rallis inner product formula}   
%
Let $(\, , \,)$ denote the $\mathbb{G}_n^-\left(\mA\right)$-invariant Hermitian inner product on $V_\sigma$ given by
the Petersson inner product, i.e.
\begin{equation}\label{e: - Petersson}
\left(\varphi_1, \varphi_2\right)\coloneqq
\int_{\mathbb{G}_n^- \left(F\right)\backslash \mathbb{G}_n^-\left(\mA\right)}
\varphi_1\left(g\right)
\overline{\varphi_2\left(g\right)}\, dg
\qquad \text{for } \varphi_1, \varphi_2 \in V_\sigma,
\end{equation}
where we take  the Tamagawa measure  $dg$.

Recall that the Haar measure constant $C_G>0$
is defined by \eqref{e: measure comparison}.
Then we have the following Rallis inner product formula.
%
%
 \begin{theorem}[Lemma~10.1 in \cite{Yam}]
 \label{Rallis inner}
 For any non-zero decomposable vectors $\varphi=\otimes_v\, \varphi_v\in V_\pi$
 and $\phi=\otimes_v\, \phi_v\in 
 \mathcal{S}\left(\left(V \otimes Y_n^+\right)\left(\mA\right)\right)$,
 we have
 \begin{equation}\label{e: rallis inner product}
 \frac{\left(\theta_{\psi, \chi_\Lambda^\Box}^\phi\left(\varphi\right),\theta_{\psi,  \chi_\Lambda^\Box}^\phi\left(\varphi\right)
 \right)}{\langle\varphi,\varphi\rangle}
 =C_G\cdot
 \frac{L\left(1/2,\pi \times \Lambda \right)}{\prod_{j=1}^{2n}\, L\left(j, \chi_{E}^j \right)}
 \cdot\prod_v Z_v^\circ\left(\varphi_v, \phi_v,\pi_v\right)
 \end{equation}
 where $Z_v^\circ\left(\phi_v, \varphi_v,\pi_v\right)=1$
 for almost all $v$.
\end{theorem}
%
\subsection{Lapid--Mao's conjecture on the  Whittaker periods for $\mathbb{G}_n^-(\mA)$}
\label{LM-conjecture}
Let us recall Lapid--Mao's conjecture concerning  an explicit formula for the Whittaker periods.
Recall that $(\sigma, V_\sigma)$ is an irreducible cuspidal $\psi_{N_n, \lambda}$-generic automorphic representation of $\mathbb{G}_n^-(\mA)$.
Let $\Sigma$ be the base change lift of $\sigma$ to $\mathrm{GL}_{2n}(\mA_E)$ constructed by 
Kim and Krishnamurthy~\cite{KK}. It is an isobaric sum $\sigma_1 \boxplus \cdots \boxplus \sigma_k$
where $\sigma_i$ is an irreducible cuspidal automorphic representation of $\mathrm{GL}_{n_i}(\mA_E)$
with $n_1 + \cdots +n_k = 2n$, such that $L^S(s, \sigma_i, \mathrm{As}^{-})$ has a simple pole at $s =1$.

Choose an 
Hermitian inner product $(\,,\,)_v$ on $V_{\sigma_v}$ 
at each place $v$ of $F$ so that 
the Petersson inner product \eqref{e: - Petersson}
becomes
\[
(\varphi_1, \varphi_2) = \prod_v (\varphi_{1, v}, \varphi_{2,v})_v
\]
for any decomposable vectors $\varphi_1 = \otimes_v \varphi_{1, v}, \varphi_2 = \otimes_v \varphi_{2, v} \in V_\sigma$.
At a finite place $v$, we define a local Whittaker period 
$I_v(\varphi_{1, v}, \varphi_{2, v})$ for $\varphi_{1, v}, \varphi_{2,v}  \in V_{\sigma_v}$ as
\[
I_v(\varphi_{1, v}, \varphi_{2, v})\coloneqq
 \int_{N_n(F_v)}^{st} ( \sigma_v(u)\varphi_{1, v}, \varphi_{2, v} )_v
 \, \psi_{N_n, \lambda, v}(u)^{-1} \, du
 \]
by the stable integral (see Lapid and Mao~\cite{LM} and Liu~\cite{Liu2}).
At an infinite place $v$, we define as follows.
 First
  we put $u_i=u_{i, i+1}$ $\,\left(1\le i\le n-1\right)$ and $u_{n} = u_{n,2n}$ for $u \in N_n$.
Then  for $\gamma \geq -\infty$, we define 
$N_{n, \gamma} \coloneqq \{ u \in N_n(F_v) \colon |u_i| \leq e^\gamma \text{ for $1 \leq i \leq n-1$} \}$.
 As in the definition of $\alpha_v(\phi_v, \phi_v)$
in \ref{local integral sec}, for $u_0 \in N_{n}(F_v)$, we define 
 \[
 I_{\varphi_{1,v}, \varphi_{2,v}}(u_0) \coloneqq \int_{N_{n, -\infty}} ( \sigma_v(u_0 u)\varphi_{1, v}, \varphi_{2, v} )_v \, du.
 \]
 This integral converges absolutely and,  
 by \cite[Proposition~3.10]{Liu2}, it
 gives a tempered distribution on 
 $N_{n}(F_v) \slash N_{n, -\infty}$.
Then we define 
\[
 I_v(\varphi_{1, v}, \varphi_{2, v})\coloneqq \widehat{I_{\varphi_{1, v}, \varphi_{2, v}}}(\psi_{N_n, \lambda, v}).
 \]
 When $\varphi_{1, v} = \varphi_{2, v}=\varphi_v$, we simply write
 \[
  I_{\varphi_{v}}(u_0) =  I_{\varphi_{v}, \varphi_{v}}(u_0)\quad \text{and}\quad I_v(\varphi_{v})=  I_v(\varphi_{v}, \varphi_{v}).
 \]
 Lapid and Mao~\cite[Theorem~5.5]{LM1} showed that there exists a non-zero constant $c_{\sigma_v}$ depending on $\sigma_v$ 
 (implicitly on the choice of measures and $\psi$)
 such that 
\begin{equation}
\label{conj pre}
\frac{|W_{\psi, \lambda}(\varphi)|^2}{( \varphi, \varphi )}
=2^{-k} \, \left(\prod_v c_{\sigma_v}^{-1}\right) \cdot \frac{\prod_{j=1}^{2n} L \left(j, \chi_E^j \right)}{L \left(1, \Sigma, \mathrm{As}^+ \right)} 
\,
\prod_v 
I_v^\natural(\varphi_v)
\end{equation}
for any non-zero decomposable vector $\varphi = \otimes \varphi_v \in V_\sigma$.
 Here $W_{\psi, \lambda}$ is defined with respect to the Tamagawa measure and we define
\begin{equation}
\label{def I_v}
I_v^\natural(\varphi_v) \coloneqq \frac{L \left(1, \Sigma_v, \mathrm{As}^+ \right)}{\prod_{j=1}^{2n} L \left(j, \chi_{E, v}^j \right)} 
\cdot
\frac{I_v(\varphi_v)}{( \varphi_v, \varphi_v )_v}.
\end{equation}
\begin{Conjecture}[Conjecture~1.2, 5.1 in \cite{LM}]
\label{main conj}
Let $\sigma$ and $\Sigma$ be as above.
Then $c_{\sigma_v} = \omega_{\sigma_v}(-1)$. In particular, 
for any non-zero decomposable vector $\varphi = \otimes_v \,\varphi_v \in V_\sigma$, we have
\begin{equation}
\label{whittaker formula ref}
\frac{|W_{\psi, \lambda}(\varphi)|^2}{( \varphi, \varphi )}
=2^{-k}\,  \frac{\prod_{j=1}^{2n} L \left(j, \chi_E^j \right)}{L \left(1, \Sigma, \mathrm{As}^+ \right)} 
\,\prod_v I_v^\natural(\varphi_v).
\end{equation}
\end{Conjecture}
This is proved under certain assumptions by the second author in \cite[Corollary~1.1]{Mo2}.
Then we may apply this result to $\sigma$ under suitable assumptions.
%
\begin{theorem}
\label{wh ex thm}
Let $\pi$ be as in Theorem~\ref{refined ggp thm}, and 
let us denote $\Theta_{V, \mathbb{W}_n}(\pi, \psi, \chi_\Lambda^\Box)$ by $\sigma$.
Then \eqref{whittaker formula ref} holds for $\sigma$.
\end{theorem}
%
\begin{proof}
By \cite{Pau} (or \cite{Ich20}) with Remark~\ref{IP rem}, any
 archimedean component of $\sigma_\infty$ should be a discrete series representation.
Moreover, because of our assumptions, $\sigma_v$ should be unramified 
when $v$ is a split finite place. Moreover, $v$ is non-split in $E$ when 
$v$ divides $2$. 
Thus $\sigma, \psi, F$ and $E$ satisfy the conditions in  \cite[Corollary~1.1]{Mo2},
and the required formula holds.
\end{proof}
%
%
%
\subsection{Reduction to a local identity}
Combining Proposition~\ref{pullback whittaker prp},Theorem~\ref{Rallis inner} and \eqref{whittaker formula ref},
we may reduce the required identity to a certain local identity by an argument
similar to that in  \cite[Section~2]{FM2}.
Recall that we assume $B_{e, \psi, \Lambda} \not \equiv 0$ on $V_\pi$.
By Theorem~\ref{opp GGP}, we have 
$\alpha_v \not \equiv 0$ on $V_{\pi_v} \times V_{\pi_v}$ at any place $v$ of $F$
as we remarked in Section~\ref{intro ref ggp}, which also follows from Corollary~\ref{non-zero cor}.
Then by the multiplicity one theorem for Bessel models, as in \cite[(2.16)]{FM2},
there is a constant $C \in \mC^\times$ such that 
\[
B_{e, \psi, \Lambda}(\varphi_1) \cdot  \overline{B_{e, \psi, \Lambda}(\varphi_2)}
= C \cdot \prod_{v} \alpha_v^\natural(\varphi_{1,v}, \varphi_{2,v})
\]
for any $\varphi_i = \otimes_v \varphi_{i, v} \in V_\pi$.
In particular, if $B_{e, \psi, \Lambda}(\varphi) \ne 0$ for $\varphi =\otimes_v \varphi_v \in V_\pi$,
then $\alpha_v^\natural(\varphi_{v}, \varphi_{v}) \ne 0$ at any place $v$.
Fix $\varphi = \otimes_v \varphi_v \in V_\pi$ such that $B_{e, \psi, \Lambda}(\varphi) \ne 0$. Then we shall prove the formula 
\eqref{formula uni} for $\varphi$.

Let us set $x_0\coloneqq\left(e_{-1},\cdots, e_{-n+1},e \right)\in V\left(F \right)^n$
and $\mathcal X_{\lambda_e, v}\coloneqq G(F_v) \cdot x_0 \subset V \left(F_v \right)^n$.
Then $\mathcal X_{\lambda_e, v}$
is isomorphic to $R^\prime_e(F_v) \backslash G(F_v)$
as $G(F_v)$-homogeneous spaces, and $\mathcal X_{\lambda_e, v}$ is locally closed in $V\left(F_v\right)^n$.
Hence, the function 
\[
R_e^\prime(F_v) \backslash G(F_v) \ni g \mapsto \phi(g^{-1} \cdot x_0)
\]
is compactly supported for any $\phi \in C_c^\infty(V \left(F_v \right)^n)$.
Therefore, the following integral converges absolutely for any $\phi_v \in C_c^\infty(V \left(F_v \right)^n)$ $\colon$
\begin{equation}\label{local identity}
W_v(\phi_v, \varphi_v)\coloneqq \int_{R_e^\prime(F_v) \backslash G(F_v)}  
  \phi_v( g^{-1} \cdot \left( e_{-1}, \dots,e_{-n+1}, e \right))  \alpha_v^\circ(g_v; \varphi_v) \, dg_v.
\end{equation}
Here, we set
\[
\alpha_v^\circ(g_v, \varphi_v) \coloneqq \alpha_v^\natural(\pi_v(g_v)\varphi_v, \varphi_v) \slash \alpha_v^\natural(\varphi_v, \varphi_v).
\]
We note that from the proof of Proposition~\ref{pullback whittaker prp}, we have
\[
W_v(\phi_v, \varphi_v)= \int_{R_e^\prime(\mathcal{O}_v) \backslash G(\mathcal{O}_v)}  
  \phi_v( g^{-1} \cdot \left( e_{-1}, \dots,e_{-n+1}, e \right))  \alpha_v^\circ(g_v; \varphi_v) \, dg_v
 =1
\]
at almost all places $v$ of $F$ since $\mathrm{vol}(R_e^\prime(\mathcal{O}_v) \backslash G(\mathcal{O}_v))=1$ by our choice of the measure.
\begin{proposition}
Keep the above notation.
Then for any $\phi =\otimes_v \phi_v\in C_c^\infty(V(\mA)^n)$, we have 
\[
W_{\psi, \lambda}(\theta_{\psi, \chi_\Lambda^\Box}^\phi(\varphi))
= (C_G C_e^{-1}) \cdot B_{e, \psi, \Lambda}(\varphi) \prod_{v}W_v(\phi_v, \varphi_v)
\]
where $C_G$ and $C_e$ are Haar measure constant defined in Section~\ref{intro ref ggp}.
\end{proposition}
\begin{proof}
Recall that $B_{e, \psi, \Lambda}(\varphi) \ne 0$ and $\alpha_v^\natural(\varphi_{v}, \varphi_{v}) \ne 0$ for any $v$.
Then for $g = (g_v) \in G(\mA)$, we obtain 
\begin{align*}
B_{e, \psi, \Lambda}(\pi(g)\varphi) \cdot  \overline{B_{e, \psi, \Lambda}(\varphi)}
&= C \cdot \prod_{v} \alpha_v^\natural(\pi_v(g_v)\varphi_{v}, \varphi_{v})\\
&=B_{e, \psi, \Lambda}(\varphi) \cdot  \overline{B_{e, \psi, \Lambda}(\varphi)}
\cdot \prod_{v} \frac{\alpha_v^\natural(\pi_v(g_v)\varphi_{v}, \varphi_{v})}{\alpha_v^\natural(\varphi_{v}, \varphi_{v})}.
\end{align*}
Hence, we have 
\begin{equation}
\label{decomp Bessel pre}
B_{e, \psi, \Lambda}(\pi(g)\varphi)
=B_{e, \psi, \Lambda}(\varphi) 
\cdot \prod_{v} \frac{\alpha_v^\natural(\pi_v(g_v)\varphi_{v}, \varphi_{v})}{\alpha_v^\natural(\varphi_{v}, \varphi_{v})}.
\end{equation}
Now, putting the decomposition of Bessel periods \eqref{decomp Bessel pre} into \eqref{pullback whittaker prp identity},
we obtain the required identity
taking the decomposition of measure into account.
\end{proof}
From the proof of Proposition~\ref{pullback whittaker prp}, we may take 
$\varphi = \otimes \varphi_v \in V_\pi$ and $\phi =\otimes \phi_v \in C_c^\infty(V(\mA)^n)$ so that 
\[
W_{\psi, \lambda}(\theta_{\psi, \chi_\Lambda^\Box}^\phi(\varphi))  \ne 0
\quad
\text{and} 
\quad
\alpha_v^\natural(\varphi_v, \varphi_v) \ne 0.
\]
\color{black}

As in \cite[Section~3.1]{FM2}, let us fix local measures as follows.
Let $\omega$ and $\omega_G$
be non-zero gauge forms on $V^n$ and $G$ respectively.
Let $\omega_0$ be the gauge form on $\mathcal X_\lambda$
given by pulling back $\omega$ via the inclusion $\mathcal{X}_\lambda \hookrightarrow V^n$.
We choose a gauge form $\omega_\lambda$ on $R^\prime_e$
such that $\omega_G=\omega_0\,\omega_\lambda$.

At each place $v$ of $F$, we specify  $dg_v$  to be the 
local Tamagawa measure on $G(F_v)$ corresponding to $\omega_G$.
See Gross~\cite{Gro} and Rogawski~\cite[1.7]{Rog} for the details concerning the definition of local Tamagawa
measures.
Then we have
\[
dg=C_G\prod_v\, dg_v\quad\text{where
$C_G=\left(\prod_{j=1}^{2n}\,L\left(j, \chi_E^j \right)\right)^{-1}$}
\]
(see \cite[Section~1.5]{Xue19}).
We also specify $dr^\prime_v$ to be the local Tamagawa
measure on $R^\prime_{e, v}$ at each place $v$
corresponding to $\omega_\lambda$.
Also $dr^\prime_v=dt_v\, ds^\prime_v$
where $dt_v$ and $ds^\prime_v$ are the local 
Tamagawa measures on $D_{e}(F_v)$ and
$S^\prime_{e}(F_v)$, respectively.
Since $ds^\prime=\prod_v\,ds^\prime_v$ where $ds^\prime$
is the Tamagawa measure on $S_e^\prime\left(\mA\right)$,
we have
\begin{equation}
\label{u measure}
dr^\prime=C_e\prod_v\,
dr^\prime_v\quad
\text{where $C_e=\frac{1}{L\left(1,\chi_E\right)}$.}
\end{equation}
We take the quotient measure $dh_v$
on $R_{e, v}^\prime \backslash G(F_v)$ so that 
\[
dg_v=dr^\prime_v\, dh_v.
\]

Then by an argument similar to that in
 \cite[Section~3]{FM2}, we observe that  
 the following proposition is sufficient to 
 prove Theorem~\ref{refined ggp thm}.
 %
 %
 %
\begin{proposition}
\label{prp: local equality}
At any place $v$ of $F$, 
for a given non-zero $\varphi_v \in V_{\pi_v}^\infty$, there exists $\phi_v \in C_c^\infty(V(F_v)^n)$
such that $W_v\left(\phi_v,\varphi_v\right) \ne 0$ and the local equality
\begin{equation}\label{e: local equality}
\frac{Z_v^\circ\left(\varphi_v,\phi_v,\pi_v\right)\, I_v^\natural\left(
\theta_{\psi_v, \chi_{\Lambda_v}^\Box}\left(\phi_v\otimes\varphi_v\right)\right)}{
\left| W_v\left(\phi_v, \varphi_v\right)\right|^2}
=\frac{\alpha_v^\natural\left(\varphi_v,\varphi_v\right)}{
\langle\varphi_v,\varphi_v\rangle_v}
\end{equation}
holds with respect to the  local measures
specified as above.
\end{proposition}
We note that  for the reduction to Proposition~\ref{prp: local equality}, we need to remark the following relations.
\begin{itemize}
\item $L(s, \Sigma, \mathrm{As}^+) = L(s, \pi, \mathrm{Ad})$.
\item $|S(\Psi(\pi))| = |S(\Psi(\sigma))| = 2^{k}$.
\end{itemize}
Since $\sigma$ is nearly equivalent to $\pi$ from the local computation by Kudla~\cite{Ku}, 
the base change lifts of $\sigma$ and $\pi$ coincide, namely $\Pi = \Sigma$.
Then the first relation follows from \eqref{Lfct bc as}
and we also have $|S(\Psi(\pi))| = |S(\Psi(\sigma))|$. Further, we note that 
$|S(\Psi(\sigma))|=2^k$ since the base change $\Sigma$ of $\sigma$ is $\sigma_1 \boxplus \cdots \boxplus \sigma_k$ and
 $\Psi(\sigma) = \sigma_1 \oplus \cdots \oplus \sigma_k$.
\\

Let us define a Hermitian inner product 
$\mathcal B_{\omega_v}$ on $\mathcal S\left(V(F_v)^n\right)$ by
%
\begin{equation}\label{e: weil Hermitian1}
\mathcal B_{\omega_v}\left(\phi_v,\phi_v^\prime\right)\coloneqq
\int_{V(F_v)^n}
\phi_v\left(x\right)\,\overline{\phi_v^\prime\left(x\right)}\,
dx
\quad\text{for $\phi_v,\phi_v^\prime\in \mathcal S\left(V(F_v)^n\right)$}
\end{equation}
where $dx$ denotes the local Tamagawa measure
corresponding to the gauge form $\omega_{v}$ on $V(F_v)^n$.
Then Liu~\cite[Lemma~3.19]{Liu2} proved that the integral
\begin{align}\label{e: weil Hermitian2}
&Z_v^\flat\left(\varphi_v, \varphi_v^\prime; \phi_v, \phi_v^\prime\right) 
= \int_{G(F_v)} 
\langle\pi\left(g\right)\varphi_v,\varphi_v^\prime\rangle\,
\mathcal B_{\omega_v}\left(\omega_{\psi_v, \chi_{\Lambda_v}^\Box}\left(g\right)\phi_v,\phi_v^\prime\right)
\, dg\\
=
 &\int_{G(F_v)} \int_{V(F_v)^n}
\langle\pi\left(g\right)\varphi_v,\varphi_v^\prime\rangle\,
\left(\omega_{\psi_v, \chi_{\Lambda_v}^\Box}\left(g\right)\phi_v \right)(x)\, \overline{\phi_v^\prime(x)}
\, dx \, dg
\notag
\end{align}
converges absolutely for $\varphi_v,\varphi_v^\prime\in V_{\pi_v}$ and 
$\phi_v,\phi_v^\prime\in\mathcal S\left(V\left(F_v\right)^n\right)$.
We note that our setting belongs to Case 2 
in the proof of \cite[Lemma~3.19]{Liu2}.

As in \cite[Section~3.2]{FM2} (see also Gan and Ichino~\cite[16.5]{GI1}), we may define
an $\mathbb{G}_n^-(F_v)$-invariant
Hermitian inner product  
$\mathcal B_{\sigma_v} \colon V_{\sigma_v}\times
V_{\sigma_v}\to\mathbb C$ by 
\begin{equation}
\label{inner def local}
\mathcal B_{\sigma_v}\left(\theta_{\psi_v, \chi_{\Lambda_v}^\Box}\left(\phi_v \otimes\varphi_v\right),
\theta_{\psi_v, \chi_{\Lambda_v}^\Box}\left(\phi_v^\prime\otimes\varphi_v^\prime\right)\right)\coloneqq Z_v^\flat\left(\varphi_v, \varphi_v^\prime; \phi_v, \phi_v^\prime\right).
\end{equation}
In the definition of $I_v(\theta_{\psi_v, \chi_{\Lambda_v}^\Box}(\phi_v \otimes \varphi_v))$, we may take $\mathcal{B}_{\sigma_v}$ as a Hermitian inner product.
Then  we denote by $\mathcal{W}_v(\theta_{\psi_v, \chi_{\Lambda_v}^\Box}(\phi_v \otimes \varphi_v), \theta_{\psi_v, \chi_{\Lambda_v}^\Box}(\phi_v \otimes \varphi_v))$  the resulting local Whittaker period. Namely, 
when $v$ is non-archimedean, we let 
\begin{align*}
&\mathcal{W}_v(\theta_{\psi_v, \chi_{\Lambda_v}^\Box}(\phi_v \otimes \varphi_v), \theta_{\psi_v, \chi_{\Lambda_v}^\Box}(\phi_v \otimes \varphi_v))
\\
&\coloneqq
 \int_{N_n(F_v)}^{st} \mathcal{B}_{\sigma_v}( \sigma_v(u)\theta_{\psi_v, \chi_{\Lambda_v}^\Box}(\phi_v \otimes \varphi_v), \theta_{\psi_v, \chi_{\Lambda_v}^\Box}(\phi_v \otimes \varphi_v) )
 \, \psi_{N_n, \lambda, v}(u)^{-1} \, du
 \end{align*}
and, when $v$ is archimedean,
 we let 
\begin{align*}
&\mathcal{W}_v(\theta_{\psi_v, \chi_{\Lambda_v}^\Box}(\phi_v \otimes \varphi_v), \theta_{\psi_v, \chi_{\Lambda_v}^\Box}(\phi_v \otimes \varphi_v))
\\
&\coloneqq \widehat{I_{\theta_{\psi_v, \chi_{\Lambda_v}^\Box}(\phi_v \otimes \varphi_v), \theta_{\psi_v, \chi_{\Lambda_v}^\Box}(\phi_v \otimes \varphi_v)}}(\psi_{N_n, \lambda, v})
 \end{align*}
where 
 \begin{align*}
&I_{\theta_{\psi_v, \chi_{\Lambda_v}^\Box}(\phi_v \otimes \varphi_v), \theta_{\psi_v, \chi_{\Lambda_v}^\Box}(\phi_v \otimes \varphi_v)}(u_0) 
\\
&\coloneqq 
\int_{N_{n, -\infty}} \mathcal{B}_{\sigma_v}( \sigma_v(u_0 u)\theta_{\psi_v, \chi_{\Lambda_v}^\Box}(\phi_v \otimes \varphi_v), \theta_{\psi_v, \chi_{\Lambda_v}^\Box}(\phi_v \otimes \varphi_v) )\, du
 \end{align*}
with $u_0 \in N_n(F_v)$.
 
By an argument
  similar to that in  \cite[Section~3.2]{FM2}, we 
  may further reduce Proposition~\ref{prp: local equality} to the following identity.
\begin{lemma}[cf. Lemma~2 in \cite{FM2}]\label{l: main lemma}
For any $\varphi,\varphi^\prime\in V_{\pi_v}^\infty$ and any $\phi,\phi^\prime\in C_c^\infty\left(V(F_v)^n\right)$,
we have
\begin{multline}\label{e: weil Hermitian13}
\mathcal{W}_v(\theta_{\psi_v, \chi_{\Lambda_v}^\Box}(\phi_v \otimes \varphi_v), \theta_{\psi_v, \chi_{\Lambda_v}^\Box}(\phi_v \otimes \varphi_v))
\\
=
C_{E_v \slash F_v} 
\int_{R_e^\prime(F_v)\backslash G(F_v)}
\int_{R_e^\prime (F_v)\backslash G(F_v)}
\alpha_v\left(\pi_v\left(h_v\right)\varphi_v,\pi_v\left(h_v^\prime\right)\varphi_v^\prime\right)
\\
\times
\left(\omega_{\psi_v, \chi_{\Lambda_v}^\Box}\left(h_v,1\right)\phi_v\right)\left(x_0\right)\,
\overline{\left(\omega_{\psi_v, \chi_{\Lambda_v}^\Box}\left(h_v^\prime,1\right)\phi_v^\prime\right)\left(x_0\right)}\,
dh_v\, dh^\prime_v
\end{multline}
\color{black}
where we set 
\[
C_{E_v \slash F_v} = \frac{L\left( 1, \chi_{E_v} \right)}{\prod_{j=1}^{2n}\,L\left(j, \chi_{E_v}^j \right)}.
\]
\end{lemma}
We note that by an argument similar to that for
 \cite[Corollary~2]{FM2}, we have the following corollary as a consequence of  this identity.
\begin{corollary}
\label{non-zero cor}
Let $\pi$ be an irreducible tempered representation of $G(F_v)$. 
Then for $\sigma = \theta(\pi, \psi_v, \chi_{\Lambda_v}^\Box)$, we have
\[
\mathrm{Hom}_{N_n(F_v)}(\sigma, \psi_{N_n, \lambda, v}) \ne 0 \Longleftrightarrow \alpha_v \not \equiv 0 \text{ on $V_\pi\times V_\pi$}.
\]
\end{corollary}
\subsection{Proof of Lemma~\ref{l: main lemma}}
 In the non-archimedean case, Lemma~\ref{l: main lemma} 
 is proved
 by an argument similar to that in \cite{FM2}. Hence
 here we consider only the archimedean case. 
Indeed we prove Lemma~\ref{l: main lemma} when $\pi_v$ and $\sigma_v$ are tempered.
It is clear that this is 
sufficient for Theorem~\ref{refined ggp thm}.
We remark that in \cite{FM2}, we proved a similar identity for
discrete series representations.
Here we adapt the argument in \cite{FM2}
to tempered representations.

In our proof of Theorem~\ref{refined ggp thm}, we use the assumption only to apply Theorem~\ref{wh ex thm}.
Actually, our proof remains valid if we assume the explicit formula for the Whittaker periods.
For the sake of our use in \cite{FM3}, we record it as a corollary.
%
%
%
\begin{corollary}
\label{Bessel under ass}
Let $\pi$ be an irreducible cuspidal tempered automorphic representation of $G(\mA)$ where $G \in \mathcal{G}_n$.
\begin{enumerate}
\item
If $B_{e, \psi, \Lambda} \equiv 0$ on $\pi$, then 
the formula \eqref{formula uni} holds for $\pi$.
\item
If $B_{e, \psi, \Lambda} \not \equiv 0$ on $\pi$, 
then the formula \eqref{formula uni} holds for $\pi$,
under the assumption that 
Conjecture~\ref{main conj} holds for $\theta_{\psi, \chi_{\Lambda}^\Box}(\pi)$. 
\end{enumerate}
\end{corollary}

 \color{black}
 
For simplicity, we abbreviate the subscript $v$ from the notation and write, for example,  $X(F_v)$ as $X$ for any algebraic group $X$ over $F$.
Also, we simply write $\theta(\phi \otimes \varphi)$ for  $\theta_{\psi, \chi_{\Lambda}^\Box}(\phi \otimes \varphi)$.
Recall that for a $\mathbb{G}_n^-$-invariant Hermitian inner product $(\,,\,)$ on $V_\sigma$, we have
\[
 I(\varphi_0, \varphi_0^\prime)  = \widehat{ I_{\varphi_0, \varphi_0^\prime, \infty}}( \psi_{N_n, \lambda}),
\]
where for $\varphi_0, \varphi_0^\prime \in V_\sigma^\infty$ and $u^\prime \in N_n$, we set
 \[
 I_{\varphi_0, \varphi_0^\prime, \infty}(u^\prime) = \int_{N_{n,  -\infty}}( \sigma(u^\prime u)\varphi_0, \varphi_0^\prime ) \, du,
 \]
which converges absolutely by \cite[Proposition~3.10]{Liu2}. 
Let us define 
\[
U_{\mathbb{G}_n^-, -\infty} \coloneqq U_{\mathbb{G}_n^-} \cap N_{n, -\infty}\quad\text{and}\quad U_{n, -\infty} \coloneqq
 \{u \in U_n \,\colon \, \widehat{u} \in N_{n, -\infty}\}.
\]
Then from the definition, we may write 
\[
 I_{\varphi_0, \varphi_0^\prime, \infty}(u^\prime) =\int_{U_{\mathbb{G}_n^-, -\infty}}\int_{U_{n, -\infty}} ( \sigma(u^\prime \widehat{u_1} u_2)\varphi_0, \varphi_0^\prime )\, du_1 \, du_2.
\]
We note that $ I_{\varphi_0, \varphi_0^\prime, \infty}(\widehat{u_1} u_2) = I_{\varphi_0, \varphi_0^\prime, \infty}(u_2 \widehat{u_1})$ for $u_1 \in U_n$ and $u_2 \in U_{\mathbb{G}_n^-}$ 
since $\widehat{u_1}^{-1} u_2^{-1} \widehat{u_1} u_2 \in U_{\mathbb{G}_n^-, -\infty}$.
Then it is a tempered distribution on $\left( \widehat{U_{n}} \slash \widehat{U_{n, -\infty}} \right)  \times \left( U_{\mathbb{G}_n^-} \slash U_{\mathbb{G}_n^-, -\infty} \right)$ by \cite[Corollary~3.13]{Liu2}.
We define partial Fourier transforms $\widehat{I_{\varphi_0, \varphi_0^\prime, \infty}}^i$ of $ I_{\varphi_0, \varphi_0^\prime, \infty}$ for $i=1,2$ by 
\begin{align*}
\langle  I_{\varphi_0, \varphi_0^\prime, \infty}, \widehat{f_1} \otimes f_2 \rangle& = \langle  \widehat{I_{\varphi_0,\varphi_0^\prime, \infty}}^1, f_1 \otimes f_2 \rangle,
\\
\langle  I_{\varphi_0, \varphi_0^\prime, \infty}, f_1 \otimes \widehat{f_2} \rangle &= \langle  \widehat{I_{\varphi_0, \varphi_0^\prime, \infty}}^2, f_1 \otimes f_2 \rangle ,
\end{align*}
where $f_1 \in \mathcal{S} \left( \widehat{U_{n}} \slash \widehat{U_{n, -\infty}} \right)$ and 
$f_2 \in \mathcal{S} \left( U_{\mathbb{G}_n^-} \slash U_{\mathbb{G}_n^-, -\infty} \right)$.
Then we have 
\begin{equation}
\label{hat 1 2}
\widehat{ \widehat{I_{\varphi_0, \varphi_0^\prime, \infty}}^2}^1 = \widehat{ \widehat{I_{\varphi_0, \varphi_0^\prime, \infty}}^1}^2 
= \widehat{I_{\varphi_0, \varphi_0^\prime, \infty}}.
\end{equation}
Further, we note that $\widehat{I_{\varphi_0, \varphi_0^\prime, \infty}}^1\left(\,\cdot \otimes f_2\right)$ and 
$\widehat{I_{\varphi_0, \varphi_0^\prime, \infty}}^2\left(f_1 \otimes \cdot\,\right)$
are smooth on the regular locus of the Pontryagin dual of  $\widehat{U_{n}} \slash \widehat{U_{n, -\infty}}$ and 
$U_{\mathbb{G}_n^-} \slash U_{\mathbb{G}_n^-, -\infty}$,
respectively, in the sense described
in  \ref{local integral sec}.

Now, let us take $\mathcal B_{\sigma}$ (see \eqref{inner def local} and \eqref{e: weil Hermitian2} for the definition) 
as a $\mathbb{G}_n^-$-invariant Hermitian inner product $(\,,\,)$ on $V_\sigma$.
%
%
%
In this case, we write $I_{\varphi_0, \varphi_0^\prime, \infty}$ with 
$\varphi_0=\theta(\phi \otimes \varphi), \varphi_0^\prime=\theta(\phi^\prime \otimes \varphi^\prime)$ 
by $f_{\phi, \phi^\prime, \varphi, \varphi^\prime, \infty}$. Hence, we obtain
\[
\mathcal{W}(\theta(\phi \otimes \varphi), \theta(\phi^\prime \otimes \varphi^\prime))
= \widehat{f_{\phi, \phi^\prime, \varphi, \varphi^\prime, \infty}}\,\left( \psi_{N_n, \lambda}|_{\widehat{U_{n}}} \otimes \psi_{N_n, \lambda}|_{U_{\mathbb{G}_n^-} }\right).
\]
As in \cite[Lemma~4]{FM2}, for any $f_1 \in \mathcal{S} \left( \widehat{U_{n}} \slash \widehat{U_{n, -\infty}} \right)$, 
we may prove the following identity
\begin{multline}
\label{e: hat 2}
 \widehat{f_{\phi, \phi^\prime, \varphi, \varphi^\prime, \infty}}^2\left(f_1 \otimes \psi_{N_n, \lambda}|_{U_{\mathbb{G}_n^-} } \right) 
 = C_{E \slash F} \int_{U_{n} \slash U_{n, -\infty}} \int_{U_n, -\infty} \int_{G}  \int_{R_e^\prime \backslash G} 
  \\
 f_1(v)\,
 \left( \omega_{\psi, \chi_{\Lambda}^\Box}(hg, \hat{v} \hat{u}) \phi \right)(x_0) \,
 \overline{\phi^\prime(h^{-1} \cdot x_0)}
\, \langle \pi(g)\varphi, \varphi^\prime \rangle \,dh \,dg \,du \, dv ,
\end{multline}
by an argument  similar to that for \cite[Corollary~3.23]{Liu2}.
For $u_1 \in U_n$, we define
\begin{multline}
\label{J-integral}
J_{\phi, \phi^\prime, \varphi, \varphi^\prime}(\widehat{u_1}) \coloneqq \int_{U_{n, -\infty}} \int_{G} \int_{R_e^\prime \backslash G} 
 \left( \omega_{\psi, \chi_{\Lambda}^\Box}(hg, \widehat{u u_1}) \phi \right)(x_0)  
 \, \overline{\phi^\prime(h^{-1} \cdot x_0)}
 \\
 \times
 \langle \pi(g)\varphi, \varphi^\prime \rangle \,dh \,dg \, du.
\end{multline}
Here the absolute convergence of the right-hand side of
\eqref{J-integral}
is shown in the following way.
Since $\phi, \phi^\prime \in C_c^\infty(V^n)$, we may take $f, f^\prime \in \mathcal{S}(V^n)$ such that $|\phi| \leq f$ and $|\phi^\prime| \leq f^\prime$. 
Then it is sufficient to show the convergence of the integral
\[
 \int_{U_{n, -\infty}} \int_{G} \int_{R_e^\prime \backslash G} 
 \left( \omega_{\psi, \chi_{\Lambda}^\Box}(hg, \widehat{u u_1}) f \right)(x_0)  
 \, \overline{f^\prime(h^{-1} \cdot x_0)}
 |\langle \pi(g)\varphi, \varphi^\prime \rangle| \,dh \,dg \, du,
\]
which follows from \cite[Corollary~3.13]{Liu2}.

Now, by telescoping the $G$-integration, the 
right-hand side of  \eqref{J-integral} is equal to
\begin{multline*}
 \int_{U_{n, -\infty}}  \int_{R_e^\prime \backslash G} \int_{R_e^\prime \backslash G}  \int_{R_e^\prime}
 \left( \omega_{\psi, \chi_{\Lambda}^\Box}(g,  \widehat{uu_1}) \phi \right)(x_0) 
 \,
 \overline{\phi^\prime(h^{-1} \cdot x_0)}
 \\
 \times
 \langle \pi(r^\prime g)\varphi, \pi(h)\varphi^\prime \rangle \, dr^\prime \,dg \,dh \, du.
\end{multline*}
Then by \eqref{hat 1 2} and \eqref{e: hat 2}, the above computation gives
\[
\mathcal{W}(\theta(\phi \otimes \varphi), \theta(\phi^\prime \otimes \varphi^\prime)) 
= C_{E \slash F}  \cdot \widehat{J_{\phi, \phi^\prime, \varphi, \varphi^\prime}}^1(\psi_{N_n, \lambda}).
\]
From the definition of $\widehat{J_{\phi, \phi^\prime, \varphi, \varphi^\prime}}^1$, for any $f \in \mathcal{S} (U_n  \slash U_{n, -\infty})$, we have
\begin{align*}
&(\widehat{J_{\phi, \phi^\prime, \varphi, \varphi^\prime}}^1, f) = (J_{\phi, \phi^\prime, \varphi, \varphi^\prime}, \widehat{f}) = \int_{U_n  \slash U_{n, -\infty}} 
J_{\phi, \phi^\prime, \varphi, \varphi^\prime}(\widehat{n}) \widehat{f}(n) \, dn
\\
= &\int_{U_n  \slash U_{n, -\infty}}  
\int_{U_{n, -\infty}}  \int_{R_e^\prime \backslash G} \int_{R_e^\prime \backslash G}  \int_{R_e^\prime}
 \left( \omega_{\psi, \chi_{\Lambda}^\Box}(g,  \widehat{nu}) \phi \right)(x_0) 
 \\
&\qquad\qquad \times \overline{\phi^\prime(h^{-1} \cdot x_0)}
 \langle \pi(r^\prime g)\varphi, \pi(h)\varphi^\prime \rangle \widehat{f}(n) \, dr^\prime \,dg \,dh \, du \, dn
\end{align*}
which converges absolutely from the definition. 
Since this integral converges absolutely and we may change the order of integration,  in order to prove Lemma~\ref{l: main lemma}, it suffices to show that 
for any $f \in \mathcal{S}(S \slash S_{-\infty})$, 
\[
 \int_{U_n \slash U_{n, -\infty}}  
\int_{U_{n, -\infty}}  \int_{R_e^\prime}
 \left( \omega_{\psi, \chi_{\Lambda}^\Box}(g,  \widehat{nu}) \phi \right)(x_0) 
 \langle \pi(r^\prime g)\varphi, \pi(h)\varphi^\prime \rangle \widehat{f}(n) \, dr^\prime \, du \, dn
\]
is equal to 
\[
\phi(g^{-1} \cdot x_0) \cdot \langle \alpha_{\pi(g)\varphi, \pi(h)\varphi^\prime}, \widehat{f} \,\rangle.
\]
In the  integral above, we regard $f \in \mathcal{S}(S \slash S_{-\infty})$ as an element of 
$\mathcal{S} \left(U_n \slash U_{n, -\infty} \right)$ by 
\[
U_n \ni u u_0^\prime \mapsto f(\check{u}r_{n-1}(a_{n-1}) \dots r_{1}(a_1))
\]
for $u \in U_{n-1}$ and $u_0^\prime \in U_0^\prime$ where we write $u_0^\prime$ 
as  in \eqref{def u_0^prime} with $a_{n-1}=0$
and $r_i(a_i)$ are defined as in the proof of Proposition~\ref{pullback whittaker prp}.
See Section~\ref{local integral sec} for the definition of  $\alpha_{\varphi, \varphi^\prime}$.
Indeed, by a computation  similar to the proof of Proposition~\ref{pullback whittaker prp}, i.e. by
a computation analogous to that in \cite{Fu} (see also \cite[Section~3.4.3]{FM2}),
we may prove that this is equal to
\begin{align*}
 &\int_{S \slash S_{-\infty}}  
\int_{S_{-\infty}}  \int_{D_e}
 \left( \omega_{\psi, \chi_{\Lambda}^\Box}(g,  1) \phi \right)(x_0) 
 \langle \pi(r^\prime s s^\prime g)\varphi, \pi(h)\varphi^\prime \rangle \widehat{f}(s) \, dr^\prime \, ds^\prime \, ds
 \\
 =
&\,
\phi(g^{-1} \cdot x_0) \cdot  \langle \alpha_{\pi(g)\varphi, \pi(h)\varphi^\prime}, \widehat{f} \rangle
\end{align*}
and this completes our proof of Lemma~\ref{l: main lemma} and Theorem~\ref{refined ggp thm}.
Here, we note that this computation is valid because our integral converges since $f, \hat{f} \in \mathcal{S}(S \slash S_{-\infty})$.
\qed
\begin{Remark}
For the convenience of the reader, we recall the relation between $\psi_{N_n, \lambda}$ and $\chi_{e, \Lambda}$ 
explored in the proof of Proposition~\ref{pullback whittaker prp}, which is also crucial
in the computation above $\colon$
\[
\psi_{N_n, \lambda}(\widehat{u} \widehat{u_0^\prime}) = \chi_{e, \Lambda}\left(\check{u} r_1(a_1) \cdots r_{n-1}(a_{n-1}) \right)
\]
where $u \in U_{n-1}$ and  $u_0^\prime \in U_0^\prime$.
\end{Remark}
%
%
%
%
%
%
%
%
%
%
%
%
%
%
%
%
%
%
%
%
%
%
%
%
%
%
%
%
%
%
%
%
%
%
%
%
%
%
%
%
%
%
%
%
%
%
%
%
%
%
%
%
%
%
%
\section{Refined Gan--Gross--Prasad conjecture of Bessel periods for $(\mathrm{GL}_{2n}, \mathrm{GL}_1)$}
\label{s:RGGP GL}
In this section we shall prove a split analogue of Theorem~\ref{refined ggp thm},  namely
the
 Ichino--Ikeda-type formula for Bessel periods on
  $(\mathrm{GL}_{2n}, \mathrm{GL}_1)$, in a similar manner.
In order for that we shall prove the
Rallis inner product formula for the theta lift from $\mathrm{GL}_k$ to $\mathrm{GL}_k$.
Further we shall compute the
pull-back of Whittaker periods for the theta lift from $\mathrm{GL}_{2n}$ to $\mathrm{GL}_{2n}$.
%
%
%
%
\subsection{Weil representation and theta lift for $(\mathrm{GL}_\ell, \mathrm{GL}_m)$}
Let us quickly recall some generalities concerning 
the theta lift for the dual pair $\left(\mathrm{GL}_\ell, \mathrm{GL}_m\right)$.

We denote $\mathcal{S}(\mathrm{Mat}_{\ell, m}(\mA))$,
the space of Schwartz functions on $\mathrm{Mat}_{\ell, m}(\mA)$,
by $\mathcal{S}_{\ell, m}(\mA)$.
We also write $\mathcal{S}(\mathrm{Mat}_{\ell, m}(F_v))$ by
$\mathcal{S}_{\ell, m}(F_v)$ for any place $v$ of $F$.
Then the Weil representation $\omega_{\ell,m}$
of $\mathrm{GL}_\ell(\mA) \times \mathrm{GL}_m(\mA)$ on $\mathcal{S}_{\ell, m}(\mA)$
is given by
\[
(\omega_{\ell, m}(g, h)\phi)(x) =\alpha(g)^{-\frac{m}{2}} \alpha(h)^{\frac{\ell}{2}}   \phi(g^{-1}xh)
\]
for $\left(g,h\right)\in \mathrm{GL}_\ell(\mA) \times \mathrm{GL}_m(\mA)$ and $\phi \in \mathcal{S}_{\ell, m}(\mA)$,
where $\alpha\left(g\right)=\left|\det g\right|$.

Let us recall another realization of the Weil representation (see \cite[Section~3]{Min}).
Let $j$ be an integer such that $0<j\le \min \{\ell, m\}$
and fix such $j$.
For each $j$, we have an isomorphism 
$\mathcal{S}_{\ell, m}(\mA)\ni \phi
\mapsto\mathcal{F}_{j, \psi}(\phi)\in \mathcal{S}_{\ell, m}(\mA)$
given by the partial Fourier transform
\[
\mathcal{F}_{j, \psi}(\phi) \begin{pmatrix} a\\ b \end{pmatrix}
= \int_{\mathrm{Mat}_{j, m}(\mA)} \phi\begin{pmatrix} a^\ast \\ b \end{pmatrix}
\psi \left(\mathrm{tr}({}^{t} a a^\ast) \right) \, da^\ast
\]
where $a\in\mathrm{Mat}_{j,m}\left(\mA\right)$ and
$b\in\mathrm{Mat}_{\ell-j,m}\left(\mA\right)$. 
Then we note that $\mathcal{F}_{j, \psi}^{-1}=\mathcal{F}_{j, \psi^{-1}}$.
Often we simply write $\mathcal{F}_{j, \psi} = \mathcal{F}_{j}$.

Now we have another action $\omega_{\ell, m, j}^\prime$ of $\mathrm{GL}_\ell(\mA) \times \mathrm{GL}_m(\mA)$ 
on $\mathcal{S}_{\ell, m}(\mA)$ defined by
\begin{multline}\label{another realization}
(\omega_{\ell, m, j}^\prime(g, h) \mathcal{F}_j(\phi)) \begin{pmatrix} x\\ y \end{pmatrix}
=\alpha(g)^{-\frac{m}{2}}\alpha(h)^{\frac{\ell}{2}}  \int_{\mathrm{Mat}_{j,m}(\mA)} \phi \left(g^{-1} \begin{pmatrix} a^\ast \\ y \end{pmatrix} h\right)
\\
\times
\psi (\mathrm{tr}({}^t x a^\ast)) \, da^\ast
\end{multline}
where $x \in \mathrm{Mat}_{j, m}(\mA)$ and $y \in \mathrm{Mat}_{\ell-j, m}(\mA)$.
Then the isomorphism $\mathcal{F}_j$ gives an isomorphism 
from $\omega_{\ell, m}$ to $\omega_{\ell, m, j}^\prime$,
i.e.
\[
\mathcal{F}_j(\left(\omega_{\ell,m}\left(g,h\right)\phi\right))
=\omega^\prime_{\ell,m,j}
\left(g,h\right)\mathcal{F}_j(\phi)
\,\,\,
\text{for}\,\,\,
\left(g,h\right)\in \mathrm{GL}_\ell(\mA) \times \mathrm{GL}_m(\mA).
\]

Let $\xi$ be a unitary character of $\mA^\times \slash F^\times$.
For a positive integer $k$, let $Z_k$ denote the center of $\mathrm{GL}_k$
and let $z_k\left(a\right)=a\cdot 1_k\in Z_k\left(\mA\right)$ for $a \in \mA^\times$.
For $s \in \mC$ and $\phi \in \mathcal{S}_{\ell, m}(\mA)$, we define  a
theta function
$\theta(s, \xi, \phi)$
by
\begin{multline*}
\theta(s, \xi, \phi) \coloneqq \int_{F^\times \backslash \mA^\times} \xi(a)\, \alpha(z_m(a))^{s+\frac{\ell}{2}}
 \sum_{x \in \mathrm{Mat}_{\ell, m}\left(F\right),\, x \ne 0} \phi(ax) \, da
 \\
 =
 \int_{F^\times \backslash \mA^\times} \xi(a)\, \alpha(z_m(a))^{s}
 \sum_{x \in \mathrm{Mat}_{\ell, m}\left(F\right),\, x \ne 0} (\omega_{\ell, m}(1, z_m(a))\phi)(x) \, da,
\end{multline*}
which converges absolutely for $\mathrm{Re}(s) \gg 0$ by \cite[Lemma~11.5, 11.6]{GJ}. 
From the definition, we have
\[
\theta(s, \xi, \omega_{\ell, m}(1_\ell, z_m(a) )\phi) 
=  \xi(a)^{-1} \alpha(z_m(a))^{-s}\theta(s, \xi, \phi)
\]
for $a \in \mA^\times$ when $\mathrm{Re}(s) \gg0$.
%

Since
\[
\omega^\prime_{\ell, m, j}(1, z_m(a))\mathcal{F}_j(\phi)
= \mathcal{F}_j(\omega_{\ell, m}(1, z_m(a))\phi),
\]
the following identity holds  by the Poisson summation formula :
\begin{equation}
\label{PSF}
\sum_{x \in \mathrm{Mat}_{\ell, m}\left(F\right)} 
(\omega^\prime_{\ell, m, j}(1, z_m(a)) \mathcal{F}_j(\phi))(x)
=
 \sum_{x \in \mathrm{Mat}_{\ell, m}\left(F\right)} 
\omega_{\ell, m}(1, z_m(a))\phi(x)
\end{equation}
for $a \in \mA^\times$.
In particular, when $\phi$ satisfies $\phi(0) = \mathcal{F}_j(\phi)(0)=0$, we have 
\begin{multline*}
\theta(s, \xi, \phi) 
\\ =
 \int_{F^\times \backslash \mA^\times} \xi(a)\, \alpha(z_m(a))^{s}
 \sum_{x \in \mathrm{Mat}_{\ell, m}\left(F\right),\, x \ne 0} (\omega^\prime_{\ell, m, j}(1, z_m(a))\mathcal{F}_j(\phi))(x) \, da.
\end{multline*}

From now on, we only consider the case when $\ell=m$.
Let $\left(\pi,V_\pi\right)$ be an irreducible cuspidal automorphic 
representation of $\mathrm{GL}_\ell\left(\mA\right)$
with $\mu_\pi$ its central character.
Then for $\phi\in\mathcal S_{\ell,\ell}\left(\mA\right)$ and $s\in\mathbb C$,
we define the 
theta lift $\theta\left(\varphi,\phi,s\right)$ of
$\varphi\in V_\pi$ to $\mathrm{GL}_\ell\left(\mA\right)$
by
\[
\theta(\varphi, \phi, s)(g) \coloneqq \int_{Z_\ell(\mA) \mathrm{GL}_\ell(F) \backslash \mathrm{GL}_\ell(\mA)}
\varphi(h) \alpha(h)^{s} \theta(s, \mu_\pi, \omega_{\ell, \ell}(g, h)\phi) \, dh.
\]
When it converges absolutely, we have
\begin{multline}\label{ex theta}
\theta(\varphi, \phi, s)(g) = \int_{\mathrm{GL}_\ell(F) \backslash \mathrm{GL}_\ell(\mA)}
\varphi(h) \alpha(h)^{s}
\\
\times \sum_{x \in \mathrm{Mat}_{\ell, \ell}\left(F\right), \,x \ne 0}\omega_{\ell, \ell}(g,h)\phi(x)\, dh.
\end{multline}
Actually  the integral converges absolutely for $\mathrm{Re}(s) \gg 0$ and it has a holomorphic continuation to $\mC$ by \cite[Lemma~3]{TW2}.
Moreover, as remarked in \cite[p.707]{TW2},  
$\theta(\varphi, \phi, s)$ is a cusp form (possibly, zero) on $\mathrm{GL}_{\ell}\left(\mA\right)$
 for  $\mathrm{Re}(s) \gg 0$.
We note that
\[
\theta(\varphi, \phi, s)(z_\ell(a)g)
 =\alpha(z_m(a))^s \mu_\pi(a) \theta(\varphi, \phi, s)(g).
\]
We note that for a given $g, h \in \mathrm{GL}_{\ell}(\mA)$, if we have 
\[
(\omega_{\ell, \ell}(g, h)\phi)(0) =0
\qquad
\text{and}
\qquad
(\omega_{\ell, \ell, j}^\prime(g, h) \mathcal{F}_j(\phi))(0) =0,
\]
then we have the following another expression of the theta lift
\begin{multline}\label{another ex theta}
\theta(\varphi, \phi, s)(g) = \int_{\mathrm{GL}_\ell(F) \backslash \mathrm{GL}_\ell(\mA)}
\varphi(h) \alpha(h)^{s}
\\
\times \sum_{x \in \mathrm{Mat}_{\ell, \ell}\left(F\right), \,x \ne 0}\omega_{\ell, \ell, j}^\prime(g,h) \mathcal{F}_j(\phi)(x)\, dh.
\end{multline}

For an irreducible cuspidal automorphic representation $(\pi, V_\pi)$ of $\mathrm{GL}_\ell(\mA)$,
we define  its theta lift $\theta\left(\pi\right)$
to $\mathrm{GL}_\ell(\mA)$ by 
\begin{equation}\label{one theta}
\theta(\pi) \coloneqq \langle \theta(\varphi, \phi, 0) \colon \varphi \in V_\pi, \phi \in \mathcal{S}_{\ell, \ell}(\mA) \rangle
\end{equation}
where the right-hand side means the
automorphic representation of $\mathrm{GL}_\ell\left(\mA\right)$
generated by
$\left\{\theta(\varphi, \phi, 0) \colon \varphi \in V_\pi, \phi \in \mathcal{S}_{\ell, \ell}(\mA)\right\}$.

We recall  the following fact \cite[Theorem~1]{TW2}
concerning
 the non-vanishing of the 
 global theta lift from $\mathrm{GL}_\ell$
 to $\mathrm{GL}_\ell$.
 %
\begin{theorem}
\label{TW thm1}
Let $\pi$ be an irreducible cuspidal automorphic representation of $\mathrm{GL}_\ell(\mA)$.
Then its theta lift $\theta(\pi)$ to $\mathrm{GL}_\ell(\mA)$ is non-zero if and only if 
$L \left(\frac{1}{2}, \pi  \right) \ne 0$. 
Moreover when $\theta\left(\pi\right)$ is non-zero,
we have
$\theta(\pi)=\pi$. 
\end{theorem}
%
Note that, in particular,
 $\theta(\varphi, \phi, 0)$ is cuspidal (possibly zero) 
 for any $\varphi \in V_\pi$ and $\phi \in \mathcal{S}_{\ell, \ell}(\mA)$ 
since $\theta\left(\pi\right)$ is zero or $\pi$
by Theorem~\ref{TW thm1}.
Hereafter we shall simply write $ \theta(\varphi, \phi)$ for $\theta(\varphi, \phi, 0)$.

In the following subsections, 
we use the realization of theta lifts as \eqref{ex theta}
for our proof of the Rallis inner product formula for the
dual pair $\left(\mathrm{GL}_k,\mathrm{GL}_k\right)$
in \ref{rif for gl}
and use the realization \eqref{another ex theta}
for the pull-back computation of the Whittaker period
in \ref{pull-back gl}.
%
%
%
%
\subsection{Rallis inner product formula for the dual pair $(\mathrm{GL}_k, \mathrm{GL}_k)$}\label{rif for gl}
For an irreducible unitary cuspidal automorphic representation $(\pi, V_\pi)$ of $\mathrm{GL}_k(\mA)$,
let us prove the Rallis inner product formula for the theta lift $\theta(\pi)$ of $\pi$ to $\mathrm{GL}_{k}(\mA)$.

For cusp forms $\varphi_1$ and $\varphi_2$ on $\mathrm{GL}_k(\mA)$ with the same unitary central character,
we define the Petersson inner product   by
\[
\langle \varphi_1, \varphi_2 \rangle = \int_{Z_k(\mA) \mathrm{GL}_k(F) \backslash  \mathrm{GL}_k(\mA)}
\varphi_1(g) \,\overline{\varphi_2(g) } \,dg.
\]
Here we normalize the measure $dg$ so that 
$\mathrm{Vol}\left(Z(\mA) \mathrm{GL}_k(F) \backslash  \mathrm{GL}_k(\mA), dg \right)=1$.
For $\phi_1, \phi_2 \in \mathcal{S}_{k,k}(\mA)$, let us  define an inner product on $\mathcal{S}_{k,k}(\mA)$ by 
\begin{align}
\label{def IP SS}
(\phi_1, \phi_2) \coloneqq &
\int_{\mathrm{Mat}_{k \times k}(\mA)} \phi_1(x) \overline{\phi_2(x)} \, dx \\
=&\int_{\mathrm{GL}_k(\mA)} (\phi_1 \cdot \alpha^{k \slash 2})(g) \overline{(\phi_2 \cdot \alpha^{k \slash 2})(g)} \, dg
\notag
\end{align}
where $dx$ denotes the measure on $\mathrm{Mat}_{k \times k}(\mA)$ given by the relation $\frac{dx}{|\det x|^k_\mA} = dg$.
Then we note that 
\[
(\omega_{k,k}(g, h)\phi_1, \omega_{k,k}(g, h)\phi_2) =(\phi_1, \phi_2).
\]
We recall that  a local measure  $dg_v$ on $\mathrm{GL}_k(F_v)$ 
is chosen 
so that $\mathrm{Vol}(K_v, dg_v)=1$ at almost all places where $K_v$ is  a maximal compact subgroup of $\mathrm{GL}_k(F_v)$. 
Then we have the decomposition $dg= C_1 \, \prod_v dg_v$ with the Haar measure constant $C_1 > 0$.
For simplicity, we set 
\begin{equation}
\label{holo section}
\Phi_{\phi}^s(g) =(\omega_{k,k}(1,g)(\phi \cdot \alpha^s), \phi \cdot \alpha^s)
\end{equation}
for $\phi \in \mathcal{S}_{k,k}(\mA)$.
For $\phi_v \in \mathcal{S}_{k,k}(F_v)$, we define an inner product $(\phi_v,\phi_v)_v$
and $\Phi_{\phi_v}^s(g)$ in a way similar to  \eqref{def IP SS} and \eqref{holo section}, respectively.
\begin{proposition}
\label{RIP GL}
Let $\phi = \otimes \,\phi_v \in \mathcal{S}_{k,k}(\mA)$ and $\varphi = \otimes \,\varphi_v \in V_\pi$. 
Then we have
\[
\frac{\langle \theta(\varphi, \phi), \theta(\varphi, \phi) \rangle}{\langle \varphi, \varphi \rangle}
=C_1 \cdot L \left(\frac{1}{2}, \pi  \right)
\, L \left(\frac{1}{2}, \pi^\vee \right) \,\prod_{v} Z_v(\varphi_v, \phi_v, 0)
\]
where $\pi^\vee$ denotes the contragredient of $\pi$.
Here we decompose the Petersson inner product into
a product of local Hermitian pairing as 
$\langle \,, \, \rangle = \prod_v \langle \,, \,\rangle_v$
and
we define 
\begin{multline*}
Z_v(\varphi_v, \phi_v, s)\coloneqq 
\frac{1}{L \left(s+\frac{1}{2}, \pi_v  \right)L \left(s+\frac{1}{2}, \pi_v^\vee \right)}
\\
\times
\int_{\mathrm{GL}_k(F_v)} \frac{\langle \pi_v(g_v)\varphi_v, \varphi_v \rangle_v}{\langle \varphi_v, \varphi_v \rangle_v} 
\cdot
\Phi_{\phi_v}^s(g_v) \, dg_v.
\end{multline*}
\end{proposition}
\begin{proof}
Before investigating the global formula, we would like to remark on the
 local integral $Z_v(\varphi_v, \phi_v, s)$ at each place $v$ of $F$.
 By Yamana~\cite[Theorem~5.2]{Yam}, we know that $Z_v(\varphi_v, \phi_v, s)$
 converges absolutely for $\mathrm{Re}(s) \gg 0$ and it has a holomorphic continuation to $\mC$.
 Moreover, from the unramified computation in \cite{LR2}, 
 we have
 $Z_v(\varphi_v, \phi_v, s)=1$ at almost all places $v$.

Let us consider the global formula.
From the proof of \cite[Lemma~3]{TW2}, we have the following 
equality
\begin{equation}
\label{TW2 lemma3}
\theta(\varphi, \phi, s)(h)= \alpha(h)^{-k \slash 2} \int_{\mathrm{GL}_k(\mA)} \varphi(g) \alpha(g)^{s+k \slash 2}
\phi(h^{-1} g) \, dg
\end{equation}
where the integral converges absolutely for $\mathrm{Re}(s) \gg 0 $.
Moreover, as we remarked in the previous section, for $\mathrm{Re}(s) \gg 0$, $\theta(\varphi, \phi, s)$ is a cusp form.
Now for $s\in\mathbb C$,  we define a pairing 
\[
\langle f_1, f_2 \rangle_s \coloneq \int_{Z_k(\mA) \mathrm{GL}_k(F) \backslash \mathrm{GL}_k(\mA)}
(f_1(g) \cdot \alpha(g)^{-s}) \overline{f_2(g) \cdot \alpha(g)^{-s}} \, dg
\]
for  smooth functions $f_1$, $f_2$ on $\mathrm{GL}_k(F) \backslash \mathrm{GL}_k(\mA)$ which have the central character $\eta \cdot \alpha^s$
with a unitary character $\eta$ of $Z_k(\mA)$, provided that the integral converges.
Then applying \eqref{TW2 lemma3}, for $\mathrm{Re}(s) \gg 0$, we have
\begin{multline*}
\langle \theta(\varphi, \phi, s), \theta(\varphi, \phi, s) \rangle_s
= \int_{Z_k(\mA) \mathrm{GL}_k(F) \backslash \mathrm{GL}_k(\mA)} \int_{\mathrm{GL}_k(\mA)}
 \int_{\mathrm{GL}_k(\mA)}
 \\
 \times
 \alpha(h)^{-\frac{k}{2}-s} \overline{\alpha(h)^{-\frac{k}{2}-s}}
 \,
 \alpha(g_1)^{s+k \slash 2} \overline{\alpha(g_2)^{s+k \slash 2}}
 \\
 \times
 \varphi(g_1) \, \overline{\varphi(g_2)}  \,
\phi(h^{-1} g_1)\,
\overline{\phi(h^{-1} g_2)}   \, dg_2 \, dg_1 \, dh.
\end{multline*}
Then by changes of variables $g_i \mapsto hg_i$
for $i=1,2$, this integral is equal to
\begin{multline*}
\int_{Z_k(\mA) \mathrm{GL}_k(F) \backslash \mathrm{GL}_k(\mA)} \int_{\mathrm{GL}_k(\mA)}
 \int_{\mathrm{GL}_k(\mA)}
 \alpha(g_1)^{s+k \slash 2}\, \overline{\alpha(g_2)^{s+k \slash 2}}
 \\
 \times
 \varphi(h g_1) \, \overline{\varphi(h g_2)} \, \phi(g_1)\,
\overline{\phi(g_2)}   \, dg_2 \, dg_1 \, dh.
\end{multline*}
Furthermore, by changes of variable
 $g_1 \mapsto g_2g_1$, we obtain
\begin{multline*}
\int_{Z_k(\mA) \mathrm{GL}_k(F) \backslash \mathrm{GL}_k(\mA)} \int_{\mathrm{GL}_k(\mA)}
 \int_{\mathrm{GL}_k(\mA)}
 \alpha(g_2g_1)^{s+k \slash 2}\, \overline{\alpha(g_2)^{s+k \slash 2}}
 \\
 \times
 \varphi(h g_2g_1) \, \overline{\varphi(h g_2)} \, \phi(g_2 g_1)\,
\overline{\phi(g_2)}   \, dg_2 \, dg_1 \, dh.
\end{multline*}
Then the above integral becomes
\begin{align*}
 &\int_{Z_k(\mA) \mathrm{GL}_k(F) \backslash \mathrm{GL}_k(\mA)} \int_{\mathrm{GL}_k(\mA)} \varphi(h g)  \,\overline{\varphi(h )} \, \Phi_\phi^s(g) \, dg \, dh
 \\
 =
& \int_{\mathrm{GL}_k(\mA)}
\langle \pi(g)\varphi, \varphi \rangle
\,
\Phi_\phi^s(g) \, dg.
\end{align*}
As we noted in the beginning of the proof, from the unramified computation in \cite{LR2}, we obtain
\begin{multline*}
\langle \theta(\varphi, \phi, s),  \theta(\varphi, \phi, s) \rangle_s
= C_1 \cdot 
\langle \varphi, \varphi \rangle \cdot
L \left(s+\frac{1}{2}, \pi\right)L \left(s+\frac{1}{2}, \pi^\vee  \right)
\\
\times
\prod_{v} Z_v(\varphi_v, \phi_v, s).
\end{multline*}
Since the both sides have holomorphic continuations, we obtain our desired formula by taking $s=0$.
\end{proof}
%
%
%
%
%
\subsection{Pull-back computation of the Whittaker period}
\label{pull-back gl}
In this section, we shall give a split analogue of Proposition~\ref{pullback whittaker prp} 
for the dual pair $(\mathrm{GL}_{2n}, \mathrm{GL}_{2n})$.

Since this subsection is essentially self contained, 
we shall use  notation which is valid only in
\ref{pull-back gl} and
might not be 
consistent with the rest of the paper.
We prioritize simplicity and hope that it would not 
cause any confusion.
%
%
\subsubsection{Bessel periods on general linear groups}
First we recall the definition of
Bessel periods in the general case 
following \cite[Section~13]{GGP} and 
also \cite[Section~2.1]{Liu1}.

Let $V$ be a  vector space over $F$
and 
$V=X\oplus W\oplus E\oplus X^\ast$
be a decomposition of $V$ where
\[
\dim X=\dim X^\ast =r,\quad
\dim W=m,\quad \dim E=1.
\]
Let $P^\prime$ be the parabolic subgroup of 
$\mathrm{GL}\left(V\right)$ stabilizing the
flag $0\subset X\subset X\oplus W\oplus E \subset
V$ and
$P^\prime=M^\prime S^\prime$ its Levi decomposition
where $M^\prime$ is a Levi part and $S^\prime$ 
is the unipotent radical of $P^\prime$.
Thus
\[
M^\prime \simeq \mathrm{GL}\left(X\right)\times
\mathrm{GL}\left(W\oplus E\right)\times
\mathrm{GL}\left(X^\ast\right)
\]
and $S^\prime$ fits in the exact sequence
\[
0\to
\mathrm{Hom}\left(X^\ast , X\right)
\to
S^\prime
\to
\mathrm{Hom}\left(W\oplus E, X\right)+
\mathrm{Hom}\left(X^\ast, W\oplus E\right)
\to 0.
\]
We may write the above exact sequence as
\[
0\to
\left(X^\ast\right)^\vee\otimes X
\to
S^\prime
\to
\left(W^\vee \oplus E^\vee\right)\otimes X
+
\left(X^\ast\right)^\vee \otimes \left(W\oplus E\right)
\to 0
\]
where we denote by $Y^\vee$  the dual space of an $F$-vector
space
$Y$.

Let $\ell_X \colon X\to F$ (resp. $\ell_{X^\ast} \colon F\to X^\ast$)
be a non-trivial $F$-linear map and 
$U_X$ (resp. $U_{X^\ast}$) a maximal unipotent subgroup
of $\mathrm{GL}\left(X\right)$ 
(resp. $\mathrm{GL}\left(X^\ast\right)$) stabilizing
$\ell_X$ (resp. $\ell_{X^\ast}$).
Moreover let
\[
\ell_W\colon \left(W^\vee \oplus E^\vee\right)
+
 \left(W\oplus E\right)\to F
 \]
 be an $F$-linear map which is trivial on
 $W+W^\vee$ but non-trivial on $E$ and $E^\vee$.
Then we define a homomorphism
$\ell\colon S^\prime\to F$ by composing
\begin{multline*}
S^\prime\to 
\left(W^\vee \oplus E^\vee\right)\otimes X
+
\left(X^\ast\right)^\vee \otimes \left(W\oplus E\right)
\\
\xrightarrow{\ell_X+\ell_{X^\ast}^\vee}
\left(W^\vee \oplus E^\vee\right)
+
 \left(W\oplus E\right)
 \xrightarrow{\ell_W}
 F.
 \end{multline*}
 Since $U_X\times\mathrm{GL}\left(W\right)\times U_{X^\ast}$
 fixes $\ell$ by conjugation, we may extend $\ell$ trivially to
 a homomorphism
 \[
 \ell \colon R \coloneqq S^\prime \rtimes
 \left(U_X\times\mathrm{GL}\left(W\right)\times U_{X^\ast}\right)
 \to F.
 \]
Let $\lambda$
be a generic character of $U_X\times U_{X^\ast}$.
Then a character $\chi$ of $R$
is defined by
$\chi \coloneqq \left(\psi\circ \ell\right)\otimes\lambda$
where $\psi$ is a non-trivial character of $F$.
%
\begin{Remark}\label{dependence Bessel model}
Note that the pair $\left(R,\chi\right)$ depends
only on the spaces $W\subset V$, 
up to conjugacy by $\mathrm{GL}\left(V\right)$.
\end{Remark}
%
\subsubsection{$\left(\mathrm{GL}_{2n},\mathrm{GL}_1\right)$
case}
Now we focus on the Bessel periods of our concern,
$\left(\mathrm{GL}_{2n},\mathrm{GL}_1\right)$-Bessel periods.
Since we need to perform explicit pull-back computation,
we realize the period concretely in terms of matrices.

Let $V$ be $F^{2n}$, the space of $2n$-dimensional column vectors
over $F$
and let 
\[
e_{-1},\, e_{-2}, \,\dots, \,e_{-n},\, e_n, \dots, e_1
\] 
be the standard basis of $F^{2n}$
indexed in this way, e.g.
$e_{-1}={}^t\left(1,0,\cdots , 0\right)$,
$e_{-2}={}^t\left(0,1,0,\cdots ,0\right)$ and $e_1={}^t\left(0,\cdots , 0, 1\right)$.
We use this basis for the matrix representation
of elements of $\mathrm{GL}\left(V\right)$.
Let $X=Fe_{-1}+Fe_{-2}+\cdots +Fe_{-n+1}$,
$E=Fe_{-n}$, $W=Fe_n$
and $X^\ast=Fe_{n-1}+Fe_{n-2}+\cdots +Fe_1$.
Let $\left\{e_{\pm i}^\vee\colon 1\le i\le n\right\}$
be the basis of $V^\vee$ dual to 
the basis $\left\{e_{\pm i}\colon 1\le i\le n\right\}$ of $V$.

Then $P^\prime=M^\prime S^\prime$ where
\[
M^\prime=
\left\{\begin{pmatrix}g_1&&\\&h&\\&&g_2\end{pmatrix}\colon
g_1,g_2\in\mathrm{GL}_{n-1},\, h\in \mathrm{GL}_2
\right\}\simeq
\mathrm{GL}_{n-1}\times\mathrm{GL}_2\times\mathrm{GL}_{n-1}.
\]
Let us define $\ell_X\colon X\to F$ and $\ell_{X^\ast} \colon F\to X^\ast$ by
\[
\ell_X\left(\sum_{i=1}^{n-1}a_{-i}e_{-i}\right)=a_{-n+1}
\quad
\text{and}\quad
\ell_{X^\ast}\left(a\right)=ae_{n-1},
\]
respectively.
Then $\ell_{X^\ast}^\vee\left(x^\vee\right)=
x^\vee\left(e_{n-1}\right)$ for $x^\vee\in \left(X^\ast
\right)^\vee$.
We define 
\[
\ell_W\colon \left(W^\vee \oplus E^\vee\right)
+
 \left(W\oplus E\right)\to F
 \] by
 $
 \ell_W\left(a\,e_n^\vee+b\,e_{-n}^\vee+c\,e_n+d\,e_{-n}\right)
 \coloneqq b+d$.
For
\[
s^\prime=\begin{pmatrix}1_{n-1}&A&B\\&1_2&C\\&&1_{n-1}\end{pmatrix}
\in S^\prime
\quad\text{where}\quad
\text{$A=\left(a_{i,j}\right)$
and $C=\left(c_{i,j}\right)$},
\]
we may regard
\[
A=e_{-n}^\vee\otimes
\left(\sum_{i=1}^{n-1} a_{i,1}\,e_{-i}\right)
+
e_n^\vee\otimes
\left(\sum_{i=1}^{n-1}a_{i,2}\,e_{-i}\right)
\in 
\left(W^\vee \oplus E^\vee\right)\otimes
X
\]
and
\[
C=
\left(\sum_{j=1}^{n-1}
c_{1,j}\, e_{n-j}^\vee\right)\otimes
e_{-n}
+
\left(\sum_{j=1}^{n-1}
c_{2,j}\, e_{n-j}^\vee\right)\otimes
e_{n}
\in \left(X^\ast\right)^\vee\otimes\left(W\oplus E\right),
\]
respectively.
Hence
\[
\ell\left(s^\prime\right)=
\ell_W\left(
a_{n-1,1}\,e_{-n}^\vee+a_{n-1,2}\,e_n^\vee
+c_{1,1}\,e_{-n}+c_{2,1}\,e_n
\right)=a_{n-1,1}+c_{1,1}.
\]
Let $\psi$ be a non-trivial character of $\mathbb A\slash F$
and we define a character $\chi$ of $S^\prime\left(\mA\right)$ by
\[
\chi\left(s^\prime\right)=\psi\left(\ell\left(s^\prime\right)\right)
\quad\text{for $s^\prime\in S^\prime\left(\mA\right)$}.
\]

For a positive integer $r$, let $N_r$ denote the group of upper triangular unipotent matrices in $\mathrm{GL}_r$,
which is the unipotent radical of a Borel subgroup of $\mathrm{GL}_r$.
Then for $u,v\in N_{n-1}$, we define
$\left(u,v\right)^\wedge \in \mathrm{GL}_{2n}$ by
\[
(u, v)^\wedge \coloneqq \begin{pmatrix}u&&\\ &1_2&\\ &&v \end{pmatrix}
\]
and let $S^{\prime\prime}\coloneqq\left\{\left(u,v\right)^\wedge\colon
u,v\in N_{n-1}\right\}$.
Let us take a generic character $\lambda$ of $S^{\prime\prime}\left(\mA\right)$
given by
\[
\lambda\left(\left(u,v\right)^\wedge\right)=
\psi\left(u_{1,2}+\cdots +u_{n-2,n-1}-v_{1,2}-\cdots -v_{n-2,n-1}\right)
\]
where $u=\left(u_{i,j}\right),v=\left(v_{i,j}\right)\in N_{n-1}\left(\mA\right)$.
Then we may extend $\chi$ to a character of $S\left(\mA\right)$ where
$S\coloneqq S^\prime S^{\prime\prime}$ by 
\[
\chi\left(s^\prime\,s^{\prime\prime}\right)=
\chi\left(s^\prime\right)\lambda\left(s^{\prime\prime}\right)
\quad
\text{for $s^\prime\in S^\prime\left(\mA\right)$ and
$s^{\prime\prime}\in S^{\prime\prime}\left(\mA\right)$}.
\]
Let us define a subgroup $D$ of $\mathrm{GL}_{2n}$ by
\[
D\coloneqq\left\{d\left(a\right)\colon
a\in\mathrm{GL}_1\right\}
\quad
\text{where
$d\left(a\right)\coloneqq\begin{pmatrix}1_{n-1}&&\\
&\left(\begin{smallmatrix}1&\\&a\end{smallmatrix}\right)&\\
&&1_{n-1}\end{pmatrix}$
for 
$a\in\mathrm{GL}_1$}
\]
and let $R\coloneqq D S$.
Then for a  character $\eta$ of $\mathbb A^\times\slash F^\times$,
a character $\chi_\eta$ of $R\left(\mA\right)$ is defined by
\[
\chi_\eta\left(d\left(a\right)s\right)=
\eta\left(a\right)^{-1}\chi\left(s\right)
\quad
\text{where $a\in\mathrm{GL}_1\left(\mA\right)$
and $s\in S\left(\mA\right)$}.
\]

For a cusp form $\varphi$ on $\mathrm{GL}_{2n}\left(\mA\right)$,
its $\left(\psi,\eta\right)$-Bessel period $B_{\psi,\eta}\left(\varphi\right)$
is defined by
\begin{equation}\label{e: def of bessel for GL}
B_{\psi,\eta}\left(\varphi\right) \coloneqq
\int_{D\left(F\right)\backslash D\left(\mA\right)}
\int_{S\left(F\right)\backslash S\left(\mA\right)}
\varphi\left(st\right)\,\chi_\eta\left(st\right)^{-1}
\, ds\, dt,
\end{equation}
which converges absolutely (cf. \cite[Section~3]{JS}). 
We say that an irreducible cuspidal automorphic representation
$\left(\pi, V_\pi\right)$ has the $\left(\psi,\eta\right)$-Bessel period
when $B_{\psi,\eta}\not\equiv 0$ on $V_\pi$.
%
\begin{Remark}\label{Bessel twist}
By definition, we have
\begin{equation}\label{bessel gl change}
B_{\psi,\eta}\left(\varphi\right)=B_{\psi, 1}\left(\varphi\otimes \eta\right)
\end{equation}
where $\left(\varphi\otimes\eta\right)\left(g\right)=
\varphi\left(g\right)\eta\left(\det g\right)$
for $g\in \mathrm{GL}_{2n}\left(\mA\right)$.
\end{Remark}
%
%
%
%
\subsubsection{Pull-back of Whittaker periods}
Let $N_{i}$ be the group of upper unipotent matrices in 
$\mathrm{GL}_{i}$ and we write $N= N_{2n}$ for simplicity.
We define a generic character $\psi_0$
of $N(\mA)$ by 
\[
\psi_0 \left( u \right) \coloneqq \psi\left(
u_{1, 2}+ \cdots + u_{n-1, n}+u_{n,n+1}- u_{n+1, n+2} - \cdots - u_{2n-1, 2n}
\right)
\]
for $u=\left(u_{i,j}\right)\in N\left(\mA\right)$.

For a cusp form $\varphi$ on $\mathrm{GL}_{2n}(\mA)$, we define its Whittaker period $W\left(\varphi\right)$  by 
\[
W\left(\varphi\right) \coloneqq \int_{N(F) \backslash N(\mA)} \varphi(n) \psi_{0}(n)^{-1} \, dn.
\]

For the pull-back computation of the Whittaker period,
as a realization of the Weil representation,
we employ $\omega^\prime_{2n,2n,j}$ with $j=n$
in \eqref{another realization}.
Hereafter we simply write $\omega^\prime$ for $\omega^\prime_{2n,2n, n}$.
We recall that for $\phi \in \mathcal{S}_{2n,2n}(\mA)$
and $g,h\in\mathrm{GL}_{2n}\left(\mA\right)$, 
we have
\[
\left(\omega^\prime(g, h) \mathcal{F}_n(\phi)\,\right) \begin{pmatrix} x\\ y \end{pmatrix}
=\alpha(g)^{-n} \,\alpha(h)^{n}  \int_{\mathrm{Mat}_{n, 2n}(\mA)} \phi \left(g^{-1} \begin{pmatrix} x^\ast \\ y \end{pmatrix} h\right)
\psi (\mathrm{tr}({}^t x x^\ast)) \, dx^\ast
\]
where $x,y\in \mathrm{Mat}_{n,2n}\left(\mA\right)$.
Thus, in particular, we have
\begin{align}
\label{weil rep GL}
\begin{split}
\omega^\prime(1, h)\phi\begin{pmatrix} x\\ y \end{pmatrix} &= 
 \phi \begin{pmatrix} x\, {}^{t}h^{-1}\\ yh \end{pmatrix}
\,\, \text{for $h \in \mathrm{GL}_{2n}(\mA)$};
\\
\omega^\prime \left(\begin{pmatrix} a&\\ &b\end{pmatrix} , 1\right) \phi \begin{pmatrix} x\\ y \end{pmatrix}
&=\alpha(a b^{-1})^n \phi \begin{pmatrix} {}^t a x\\ b^{-1}y \end{pmatrix}
\,\, \text{for $a, b \in \mathrm{GL}_n(\mA)$};
\\
 \omega^\prime\left(\begin{pmatrix} 1_n&S\\ &1_n\end{pmatrix} ,1 \right) \phi\begin{pmatrix} x\\ y \end{pmatrix}
&=\psi\left[\mathrm{tr}\left({}^{t}xSy\right)\right]\phi\begin{pmatrix} x\\ y \end{pmatrix}
\\
&=\psi\left[\mathrm{tr}\left(y\,{}^{t}xS\right)\right]\phi\begin{pmatrix} x\\ y \end{pmatrix}
\,\,\text{for $S \in \mathrm{Mat}_{n, n}(\mA)$},
\end{split}
\end{align}
where
 $x, y \in \mathrm{Mat}_{n, 2n}(\mA)$.
This is a split analogue of \eqref{weil rep}.
From the above explicit action of $\omega^\prime$, the following lemma readily follows.
\begin{lemma}
Let $\mathcal{S}(\mathrm{GL}_{2n}(\mA))$ denote the space of Schwartz functions on $\mathrm{GL}_{2n}(\mA)$
which we may naturally regard as a subspace of $\mathcal{S}_{2n,2n}(\mA)$.

Then for $\phi \in \mathcal{S}(\mathrm{GL}_{2n}(\mA))$, we have $(\omega^\prime(u,h) \phi)(0)=0$
with $u \in N(\mA)$ and $h \in \mathrm{GL}_{2n}(\mA)$.
\end{lemma}
Then the following proposition holds.
%
%
%
%
%
\begin{proposition}
\label{pullback GL}
Let $(\pi, V_\pi)$ be an irreducible cuspidal automorphic representation of $\mathrm{GL}_{2n}(\mA)$.
For a cusp form $\varphi \in V_\pi$ and $\phi = \mathcal{F}_{n, \psi^{-1}}(\phi_0)$ with $\phi_0 \in \mathcal{S}(\mathrm{GL}_{2n}(\mA))$,
we have 
\begin{equation}
\label{e:pullback GL}
W\left(\theta(\varphi, \phi, s)\right) 
=  \int_{R^\prime(\mA) \backslash \mathrm{GL}_{2n}(\mA)}
\omega^\prime(1, g)
 \phi_0(z_{0})
\alpha(g)^{s}  B_{\psi, \alpha^s}(\pi(g) \varphi) \, dg
\end{equation}
for any $s \in \mC$. Here
$z_0\in \mathrm{Mat}_{2n,2n}\left(F\right)$ is given by
\begin{equation}\label{e: special point}
z_0=\begin{pmatrix}x_0\\y_0\end{pmatrix}
\quad\text{where}\quad
x_0=\begin{pmatrix}{}^{t}e_{-1}\\{}^{t}e_{-2}\\ \vdots \\{}^{t}e_{-n+1}\\{}^{t}e_{-n}\end{pmatrix}
\quad
\text{and}\quad
y_0=\begin{pmatrix}{}^{t}e_{-n}\\ {}^te_{n-1}\\ \vdots\\{}^te_{2}
 \\{}^te_{1}\end{pmatrix}
\end{equation}
and
\begin{equation}\label{e: fixed group}
R^\prime=
\left\{g\in\mathrm{GL}_{2n} \colon
x_0\, {}^tg=x_0,\, y_0\,g^{-1}=y_0\right\}.
\end{equation}

More generally, for any character $\eta$ of $\mA^\times \slash F^\times$, we have
\begin{multline}
\label{e:pullback GL gen}
W(\theta(\varphi \otimes \eta, \phi, s)) 
\\
=
\int_{R^\prime(\mA) \backslash \mathrm{GL}_{2n}(\mA)}
\omega^\prime(1, g)
 \phi_0(z_0)
\alpha(g)^{s} \eta(\det g) B_{\psi, \eta \alpha^s}(\pi(g) \varphi) \, dg.
\end{multline}
In particular, $\pi$ has the 
$(\psi, \eta)$-Bessel period for a character $\eta$ if and only if 
$\theta(\pi\otimes\eta)$ has the Whittaker period.
\end{proposition}
\begin{Remark}
Please be aware that the first row of $y_0$ is not ${}^t e_n$ but ${}^te_{-n}$.
\end{Remark}
%
%
%
%
%
%
%
%
%
%
%
%
%
%
%
%
%
%
%
%
\begin{proof}
By \eqref{bessel gl change}, the equality \eqref{e:pullback GL gen}
readily follows from \eqref{e:pullback GL} by replacing $\pi$ by $\pi\otimes\eta$.
Hence it is enough for us to prove \eqref{e:pullback GL}, i.e.
the case where $\eta$ is trivial.

The following argument is similar to the proof of Proposition~\ref{pullback whittaker prp}.

Suppose  $\mathrm{Re}(s) \gg 0$ so that $\theta(\varphi, \phi, s)$ converges absolutely and 
it is cuspidal. 
Since $N=N_M U_n$ where
\[
U_{n} =\left\{ \begin{pmatrix} 1_n&X\\ &1_n \end{pmatrix} \right\}, 
\qquad N_{M} = \left\{ \begin{pmatrix} n_1&\\ &n_2 \end{pmatrix} \colon n_i \in N_{n} \right\},
\]
we have
\[
W(\theta(\varphi, \phi, s))
= \int_{N_{M}(F) \backslash N_{M}(\mA)} \int_{U_{n}(F) \backslash U_{n}(\mA)} \theta(\varphi, \phi, s)(vu)
\psi_{0}(vu)^{-1} \, du \, dv.
\]
Further, by \eqref{another ex theta}, this integral is equal to
\begin{multline*}
 \int_{N_{M}(F) \backslash N_{M}(\mA)} \int_{U_{n}(F) \backslash U_{n}(\mA)} 
 \int_{\mathrm{GL}_{2n}(F) \backslash \mathrm{GL}_{2n}(\mA)} 
\alpha(g)^{s}
\\
\times
 \sum_{x, y \in  \mathrm{Mat}_{n, 2n}(F), \, (x,y) \ne (0, 0)} \omega^\prime(vu, g) \phi_0 \begin{pmatrix}x\\y \end{pmatrix}
 \varphi(g) \psi_{0}(vu)^{-1} \,dg \, du \, dv.
\end{multline*}
By  a computation similar to that in \cite[p.95]{Fu}, this integral is equal to
\begin{multline}
\label{e:p95}
 \int_{N_{M}(F) \backslash N_{M}(\mA)}  \int_{\mathrm{GL}_{2n}(F) \backslash \mathrm{GL}_{2n}(\mA)}
\alpha(g)^{s}
\\
\times  \sum_{\begin{pmatrix}x\\y\end{pmatrix} \in \mathcal{X}} \omega^\prime(v, g) \phi_0 \begin{pmatrix}x\\y \end{pmatrix}
\varphi(g) \psi_{0}(v)^{-1} \, dg  \, dv
\end{multline}
where
\[
\mathcal X=
\left\{\begin{pmatrix}x\\y\end{pmatrix} \colon
x,y\in \mathrm{Mat}_{n,2n}\left(F\right),
\,
y\, {}^tx=E_0
\right\}
\quad\text{with}\quad
E_0=\begin{pmatrix}0&\cdots&0&1\\ 0&\cdots&0&0\\
\vdots&&\vdots&\vdots\\  0&\cdots&0&0 \end{pmatrix}.
\]
%
%
%
\begin{lemma}
In the integral \eqref{e:p95}, only $ \begin{pmatrix}x\\y \end{pmatrix}$ such that 
$\mathrm{rank}\left(x\right)=\mathrm{rank}\left(y\right)=n$ contributes to $W(\theta(\varphi, \phi, s))$.
\end{lemma}
\begin{proof}
First we note that $N_M(F)$ acts on $\mathcal X$ via 
\[
\begin{pmatrix}n_1&\\ &n_2  \end{pmatrix} \cdot \begin{pmatrix}x\\y \end{pmatrix}
\coloneqq\begin{pmatrix}{}^{t} n_1x\\ n_2^{-1}y \end{pmatrix},
\]
since $n_2^{-1}y\, {}^t({}^{t} n_1x)= n_2^{-1} y\,{}^{t}x n_1= y\,{}^{t}x$ for $\begin{pmatrix}x\\y \end{pmatrix} \in \mathcal X$ and $n_i \in N_{n}(F)\,\left(i=1,2\right)$.
When we write $x ={}^t ({}^t x_1, {}^tx_2, \dots, {}^tx_n)$ and  $y ={}^t ({}^t y_1, {}^ty_2, \dots, {}^ty_n)$ with $x_i , y_i \in \mathrm{Mat}_{1, 2n}(F)$
for $\begin{pmatrix}x\\y \end{pmatrix} \in \mathcal X$,
we have $y_1 {}^{t}x_n=1$
and hence $x_n \ne 0$, $y_1 \ne 0$
in particular.

Suppose that $\mathrm{rank}(x) \leq n-1$. If 
there exists a linear relation $\sum_{i=1}^{n-1} a_i x_i +x_n = 0$ with $a_i \in F\,\left(1\le i\le n-1\right)$, then 
there exists $n_0 \in N_{n}(F)$ such that ${}^t n_0\, x = {}^t ({}^t x_1, \dots, {}^t x_{n-1}, 0)$. As we noted above, this case does not occur. 
Hence, $x_{1}, \dots, x_{n-1}$ are linearly dependent.
Then we may find $n_1 \in N_{n}(F)$ such that 
${}^t n_1 x = {}^{t}({}^t x_1, \dots, {}^t  x_{\ell-1}, 0, {}^t  x_{\ell+1}, \dots, {}^t  x_n)$ with some $1 \leq \ell \leq n-1$.
Hence we may suppose that $x$ is of this form. In this case, let us define 
\[
N_M^\prime \coloneqq \left\{ u \in N_{M} \colon u \cdot \begin{pmatrix}x\\y \end{pmatrix}
= \begin{pmatrix}x\\ y \end{pmatrix} \right\}
\]
and the subgroup $N_{M, \ell}$ of $N_M^\prime$ by
\[
N_{M, \ell} \coloneqq \left\{ \begin{pmatrix} u&\\ &1_n \end{pmatrix} \colon u \in N_{n} \text{ such that $u_{i,j}=0$ when $i \ne \ell$ and $j>i$}\right\}.
\]
We consider the integral
\begin{align}
& \int_{N_{M}(F) \backslash N_{M}(\mA)} \psi_{0}(v)^{-1} \sum_{u_0 \in N_M^\prime(F) \backslash N_M(F)}
 \omega^\prime(v, g) \phi_0 \left( u_0 \begin{pmatrix}x\\y \end{pmatrix} \right)  \, dv
\\
=& \notag
 \int_{N_{M}^\prime(F) \backslash N_{M}(\mA)} 
 \omega^\prime(v, g) \phi_0 \begin{pmatrix}x\\y \end{pmatrix} \psi_{0}(v)^{-1} \, dv
 \\
 =& \notag
  \int_{N_{M}^\prime(\mA) \backslash N_{M}(\mA)} \int_{N_{M}^\prime(F) \backslash N_{M}^\prime(\mA)} 
 \omega^\prime(v, g) \phi_0  \begin{pmatrix}x\\y \end{pmatrix} \psi_{0}(vv_0)^{-1} \, dv_0 \, dv.
 \end{align}
 Since $\psi_0$ is not trivial on $N_{M, \ell} (\mA)$ and hence on $N_{M}^\prime(\mA)$, the inner integral
 vanishes.
 
 The case when $\mathrm{rank}(y) \leq n-1$ is proved in a similar way.
\end{proof}
Let us define 
\[
\mathcal X^\prime\coloneqq
\left\{\begin{pmatrix}x\\y\end{pmatrix}
\in \mathcal X \colon
\mathrm{rank}(x) = \mathrm{rank}(y) = n
\right\}.
\]
\begin{lemma}
$\mathrm{GL}_{2n}(F)$ acts transitively on $\mathcal X^\prime$ 
by the action
\[
g\cdot \begin{pmatrix}x\\y\end{pmatrix}
\coloneqq
\begin{pmatrix}x\,\,{}^tg\\ y\,g^{-1}\end{pmatrix}
\quad\text{for $g\in\mathrm{GL}_{2n}\left(\mA\right)$ and
$x,y\in\mathrm{Mat}_{n,2n}\left(\mA\right)$}.
\]
\end{lemma}
\begin{proof}
Since $\mathrm{rank}(x)=n$, we may take $g_1 \in \mathrm{GL}_{2n}(F)$
such that 
\[
g_1\cdot \begin{pmatrix}x\\y\end{pmatrix}=\begin{pmatrix} x_0\\ y^\prime \end{pmatrix}.
\]
By the condition $y^\prime\, {}^{t}x_0= E_0$,
we see that $y^\prime$ is of the form $\left( E_0 , B \right)$ with some $B \in \mathrm{Mat}_{n, n}(F)$.
Further, $\mathrm{rank}(y^\prime) =n$ implies that there exists $g \in \mathrm{GL}_n(F)$ so that 
\[
B g^{-1} = {}^t({}^t b, {}^{t}f_{2}, \dots, {}^t f_{n})
\]
with some $b \in \mathrm{Mat}_{1, n}(F)$ and the standard basis $f_1=(1,0, \dots, 0), \dots ,f_n=(0, \dots, 0,1)$ of $\mathrm{Mat}_{1, n}(F)$.
Then 
\[
\begin{pmatrix} 1_n&\\ &g \end{pmatrix} \cdot \begin{pmatrix} x_0\\ y^\prime \end{pmatrix} = \begin{pmatrix} x_0\\ (E_0, ({}^t b, {}^{t}f_2, \dots, {}^t f_{n})) \end{pmatrix}.
\]
Finally, when we put 
\[
u_b = \begin{pmatrix}1_n &\begin{pmatrix} 0_{n-1, n}\\ b\end{pmatrix} \\ &1_n \end{pmatrix},
\]
we obtain 
\[
u_b \cdot \begin{pmatrix} x_0\\ (E_0, ({}^t b, {}^{t}f_2, \dots, {}^t f_{n})) \end{pmatrix}  =  \begin{pmatrix} x_0\\ y_0\end{pmatrix}.
\]
\end{proof}
By these two lemmas, we may write the integral \eqref{e:p95} as 
\begin{align*}
&\int_{N_{M}(F) \backslash N_{M}(\mA)}  \int_{\mathrm{GL}_{2n}(F) \backslash \mathrm{GL}_{2n}(\mA)} 
 \\
 &\qquad\qquad
 \times
 \sum_{\gamma \in R^\prime(F) \backslash \mathrm{GL}_{2n}(F)} 
\alpha(g)^{s}\, \omega^\prime(v, g) \phi_0 \left( \gamma^{-1} \cdot z_0\right)
\,\varphi(g)\, \psi_0(v)^{-1}\, dg  \, dv
\\
=
 &\int_{N_{M}(F) \backslash N_{M}(\mA)}  \int_{R^\prime(F) \backslash \mathrm{GL}_{2n}(\mA)} 
\alpha(g)^{s}\,
 \omega^\prime(v, g)
\,\phi_0 \left( z_0\right)\,\varphi(g)
\, \psi_{0}(v)^{-1}\, dg  \, dv
\end{align*}
where $z_0$ is given by \eqref{e: special point}
and $R^\prime$ is defined by 
\eqref{e: fixed group}.
We note that 
\[
R^\prime = D \,S^{\prime}_{0}
\]
with
\[
 S^{\prime}_{0}= \left\{ \begin{pmatrix}1_{n-1}&A&B\\ &1_2&C\\ &&1_{n-1} \end{pmatrix}\colon
 A\begin{pmatrix}1\\0\end{pmatrix}=0, \, 
 \begin{pmatrix}1&0\end{pmatrix}C=0 \right\}.
\]
Then by \eqref{weil rep GL}, we obtain
\begin{align*}
& \int_{R^\prime(\mA) \backslash \mathrm{GL}_{2n}(\mA)} \int_{D(F) \backslash D(\mA)}
  \int_{ S^{\prime}_{0}(F) \backslash  S^{\prime}_{0}(\mA)} 
  \int_{N_{M}(F) \backslash N_{M}(\mA)} 
  \\
&\qquad\qquad
\times  \alpha(h g)^{s}\,  \omega^\prime(v, h s^{\prime} g)
 \phi_0 \left( z_0\right)
\,\varphi(h s^{\prime} g)\, \psi_0(v)^{-1} \, dv \, ds^{\prime} \, dh \, dg
\\
=
& \int_{R^\prime(\mA) \backslash \mathrm{GL}_{2n}(\mA)} \int_{D(F) \backslash D(\mA)}
  \int_{ S^{\prime}_{0}(F) \backslash  S^{\prime}_{0}(\mA)} 
  \int_{N_{M}(F) \backslash N_{M}(\mA)}  
  \\
 &\qquad\qquad \times
\alpha(h g)^{s} \,\omega^\prime(v, g)  \phi_0 \left( z_0\right)
\,\varphi(h s^{\prime} g) \,\psi_0(v)^{-1}  \, dv \, ds^{\prime} \, dh \, dg.
\end{align*}
Then by an argument similar to that
 in \cite[p.96--98]{Fu}, we have
\begin{multline}\label{before hc}
W(\theta(\varphi, \phi, s))
=
 \int_{R^\prime(\mA) \backslash \mathrm{GL}_{2n}(\mA)}
\alpha(g)^{s}\,  \omega^\prime(1, g)
 \phi_0 \left( z_0\right)
 \\
 \times
  \int_{R(F) \backslash R(\mA)}
  \chi_{\alpha^s}\left(r\right)^{-1} \varphi(rg)\, dr \, dg.
\end{multline}
 We note that in the computation above, the element $s_j(a_j)$ in \cite[p.97]{Fu} should be replaced by 
\[
s_j(a_j, b_j) = \begin{pmatrix}1&a_j&0\\ &1&-b_j\\ &&1 \end{pmatrix}
\]
as in the non-split case.
Since \eqref{before hc} holds for $\mathrm{Re}(s) \gg 0$, by taking the holomorphic continuations
of both sides,
we obtain \eqref{e:pullback GL}. Hence \eqref{e:pullback GL gen} holds for any $\eta$
as we remarked in the beginning of the proof.

From \eqref{e:pullback GL gen},
$\pi$ has the $\left(\psi,\eta\right)$-Bessel period for a character $\eta$
if and only if $\theta\left(\pi\otimes\eta\right)$ has the Whittaker period
by an argument similar to that for \cite[Proposition~2]{FM1}.
\end{proof}
%
%
%
\begin{corollary}\label{non-vanishing corollary}
Keep the  notation above. Then 
an irreducible cuspidal automorphic representation
$\pi$ of $\mathrm{GL}_{2n}\left(\mA\right)$
has the $(\psi, \eta)$-Bessel period 
if and only if $L \left( \frac{1}{2}, \pi \times \eta \right) \ne 0$.
\end{corollary}
\begin{proof}
Suppose that $\pi$ has the $(\psi, \eta)$-Bessel period.
Then by Proposition~\ref{pullback GL}, $\theta(\pi\otimes\eta)$ has the
Whittaker period.
In particular, $\theta(\pi\otimes\eta) \ne 0$, and we have $L \left( \frac{1}{2}, \pi \times \eta \right) \ne 0$ 
by Theorem~\ref{TW thm1}.

Conversely, suppose that $L \left( \frac{1}{2}, \pi \times \eta \right) \ne 0$.
Then by Theorem~\ref{TW thm1}, $\theta(\pi\otimes\eta) \ne \left\{0\right\}$.
Since $\theta(\pi\otimes\eta)$ is cuspidal, 
the Whittaker period does not vanish on $\theta(\pi\otimes\eta)$.
Thus $\pi$ has the $(\psi, \eta)$-Bessel period
by Proposition~\ref{pullback GL}.
\end{proof}
\begin{Remark}
Let $\pi$ and $\sigma$ be irreducible cuspidal automorphic representations
of $\mathrm{GL}_n(\mA)$ and $\mathrm{GL}_m(\mA)$, respectively,
with $n>m$ and $n \not \equiv m$ (mod 2).
Then Liu~\cite{Liu1} showed that $\pi$ has the $\sigma$-Bessel period
if and only if $L\left(\frac{1}{2}, \pi \times \sigma \right) \ne 0$
by considering certain Rankin--Selberg integrals.
The proof of Corollary~\ref{non-vanishing corollary} 
gives another proof in the special case $\left(\mathrm{GL}_{2n},\mathrm{GL}_1\right)$.
\end{Remark}
%
%
%
%
%
%
%
%
\subsection{Proof of Theorem~\ref{refined GGP GL}}
First we recall the following explicit formula of the
Whittaker periods for $\mathrm{GL}_m$.
\begin{theorem}[Theorem~4.1 in \cite{LM}]
\label{LM GL thm}
Let $(\pi_0, V_{\pi_0})$ be an irreducible cuspidal tempered automorphic representation of $\mathrm{GL}_{m}(\mA)$.
Then for any non-zero decomposable vector $\varphi = \otimes \varphi_v \in V_{\pi_0}$, we have 
\[
\frac{|W_{\psi_\lambda}(\varphi)|^2}{\langle \varphi, \varphi \rangle}
= \frac{\prod_{j=2}^m \zeta_F(j)}{L(1, \pi_0, \mathrm{Ad})} \prod_v \frac{\beta_v(\varphi_v)}{\langle \varphi_v, \varphi_v \rangle}
\]
where
$W_{\psi_\lambda}$ denotes the Whittaker period
with respect to a generic character $\psi_\lambda$ and
\[
\beta_v(\varphi_v) \coloneqq \frac{L(1, \pi_{0,v}, \mathrm{Ad})}{\prod_{j=2}^m \zeta_{F_v}(j)}
\int_{N_m(F_v)}^{\rm st} \langle \pi_{0, v}(n)\varphi_v, \varphi_v \rangle \psi_{\lambda,v}(n)^{-1} \, dn.
\]
\end{theorem}

Let us turn to the proof of Theorem~\ref{refined GGP GL}.
The equality \eqref{refined ggp gl}
is proved in the following way.

Suppose that $\pi$ has the $(\psi, \eta)$-Bessel period.
Then as in the proof of Theorem~\ref{refined ggp thm},
our desired equality \eqref{refined ggp gl} is reduced to 
a certain local identity by combining the
Rallis inner product formula (Proposition~\ref{RIP GL}), the pull-back formula
of the Whittaker periods 
(Proposition~\ref{pullback GL}) and Lapid--Mao's formula (Theorem~\ref{LM GL thm}).
Indeed, the local identity in question is nothing but \eqref{e: local equality},
which we already proved as Proposition~\ref{prp: local equality}.
Thus \eqref{refined ggp gl} holds.
Note that the appearance of $\mathrm{Res}_{s=1}\zeta_F(s)$ 
in  \eqref{refined ggp gl} is due to our normalization
of the measure \eqref{u measure}.

Suppose that $\pi$ does not have the $(\psi, \eta)$-Bessel period. Then by Proposition~\ref{pullback GL},
$L \left(\frac{1}{2}, \pi \times \eta \right) =0$ and our formula follows
since the both sides of \eqref{refined ggp gl} vanish.
%
%
%
%
%
%
%
%
%
%
%
%
%
%
%
%
%
%
%
%
\appendix
\section{Similitude unitary groups case}
\label{s:Similitude unitary groups case}
In this appendix, 
we shall prove the analogues of Theorem~\ref{opp GGP} and Theorem~\ref{refined ggp thm} in the case of similitude unitary groups.
They shall be of essential use in 
our forthcoming paper~\cite{FM3} on the
Gan--Gross--Prasad conjecture and its refinement
for $\left(\mathrm{SO}\left(5\right), \mathrm{SO}\left(2\right)\right)$.
%
\subsection{Set up}
As in \ref{subsection unitary group}, 
let $(V, ( \, , \,)_V )$ be
 a $2n$-dimensional Hermitian space over $E$ with a 
non-degenerate Hermitian pairing $(\, , \,)_V$
where the Witt index of $V$ is at least $n-1$.
Then we denote by $\widetilde{\mathcal{G}}_n$  the set of 
$F$-isomorphism classes of the similitude unitary groups $\mathrm{GU}(V)$ for such $V$.
By abuse of notation, we shall often identify $\mathrm{GU}(V)$ with 
its isomorphism class in $\widetilde{\mathcal{G}}_n$.
Let $\mathbb V_n$ be a $2n$-dimensional Hermitian space
with the Witt index $n$.
Let $\widetilde{\mathbb{G}}_n$ denote
the similitude unitary group $\mathrm{GU}\left(\mathbb V_n\right)$.
As in \ref{subsection unitary group}, 
we may realize $\widetilde{\mathbb{G}}_n$
as
\[
 \widetilde{\mathbb{G}}_n =\left\{ h \in \mathrm{GL}_{2n}(E) \colon 
 {}^{t}\bar{h}\, w_{2n}\, h = \nu(h) w_{2n}, \, \nu(h) \in F^\times \right\}
\]
in matrices.
We often abbreviate $\mathbb V_n$ to $\mathbb V$
and $\widetilde{\mathbb{G}}_n$ to 
$\widetilde{\mathbb{G}}$.

As in \ref{subsection skew hermtian},
let $\left(\mathbb W_n, \left(\, , \, \right)_{\mathbb W}\right)$
be a $2n$-dimensional skew-Hermitian space over $E$
with non-degenerate skew Hermitian pairing $\left(\, , \,\right)_{\mathbb W}$
whose Witt index is $n$.
Let $\widetilde{\mathbb{G}}^-_n$ denote the similitude
unitary group $\mathrm{GU}\left(\mathbb W_n\right)$.
As in \ref{subsection skew hermtian},
we may realize $\widetilde{\mathbb{G}}^-_n$
as 
\[
\widetilde{\mathbb{G}}^-_n=
\left\{h\in \mathrm{GL}_{2n}\left(E\right)\colon
{}^t \bar{h}\, J_n\, h=\nu\left(h\right)  J_n, \, \nu(h) \in F^\times
\right\}.
\]
We often abbreviate $\mathbb W_n$ to $\mathbb W$
and $\widetilde{\mathbb{G}}^-_n$ to 
$\widetilde{\mathbb{G}}^-$.
We note that as \eqref{pm iso} we have
\begin{equation}\label{unitary isomorphism}
\widetilde{\mathbb{G}}^-_n\simeq
\widetilde{\mathbb G}_n.
\end{equation}
%
%
%
\subsection{Similitude theta correspondence}
\label{similitude theta}
Let $\widetilde{G} = \mathrm{GU}(V) \in \widetilde{\mathcal{G}}_n$.
We define a subgroup $R$ of $\widetilde{G} \times \widetilde{\mathbb{G}}^-_n$
by
\[
R \coloneqq \left\{ (g, h) \in \widetilde{G} \times \widetilde{\mathbb{G}}^-_n  \colon \nu(g) = \nu(h) \right\}
\]
and we regard this as a subgroup of $\mathrm{Sp}( \mathrm{Res}_{E \slash F}(V \otimes_E \mathbb{W}_n))$.
As in Section~\ref{ss:weil rep}, let us denote by $\omega_{\psi, \chi}$ the Weil representation of $R\left(\mA\right)$
associated to $\psi$ and $\chi = (\chi_V, \chi_{\mathbb{W}_n})$,
a pair of characters of 
$\mathbb A_E^\times\slash E^\times\,\mA_F^\times$. 


For $\phi \in \mathcal{S}(V(\mA)^n)$, let $\theta_{\psi, \chi}^\phi(g, h)$ be the theta function on $R\left(\mA\right)$
defined similarly to \eqref{theta fct def}. Then for a cusp form $\varphi$ on $\widetilde{G}(\mA)$, we define its
 theta lift  by 
\[
\theta_{\psi, \chi}^\phi(\varphi)(h) = \int_{\mathrm{U}(V)(F) \backslash \mathrm{U}(V)(\mA)} \theta_{\psi, \chi}^\phi(g_1g,  h) \varphi(g_1 g) \, dg_1
\]
where
 \[
 h\in\widetilde{\mathbb{G}}^-_n(\mA)^+\coloneqq \{ h \in 
 \widetilde{\mathbb{G}}^-_n(\mA) \colon \nu(h) \in \nu(\widetilde{G}(\mA)) \}
 \]
 and
$g \in \widetilde{G}(\mA)$ is chosen so that 
$\nu(g) = \nu(h)$.
It converges absolutely and defines an automorphic form on 
$\widetilde{\mathbb{G}}^-_n(\mA)^+$.
Let us put
\[
\widetilde{\mathbb{G}}^-_n(F)^+ =\widetilde{\mathbb{G}}^-_n(\mA)^+ \cap  
\widetilde{\mathbb{G}}^-_n(F).
\]
Then we note that
\[
\widetilde{\mathbb{G}}^-_n(F)^+= \left\{ g \in \widetilde{\mathbb{G}}^-_n(F) ; \nu(g) = \nu(h)  \, \text{ for some 
$h \in \widetilde{G}(F)$}
\right\}.
\]
Indeed, it suffices to check the identity when $n=1$ since  the Witt index of $V$ is at least $n-1$.
When $n=1$, we know that $G(F) \simeq D_0^\times(F) \times E^\times \slash \{(a, a^{-1}) : a \in F^\times \}$ 
with some quaternion algebra $D_0$ over $F$. 
Then $\nu(G(\mA)) = \mathrm{N}_0(D_0^\times(\mA))$ 
where $\mathrm{N}_0$ denotes the reduced norm  of $D_0$.
By Eichler's norm theorem (see Vigneras~\cite[Th\'{e}o. III.4.1]{Vig80}), 
we obtain 
$\nu(\widetilde{G}(\mA)) \cap F^\times = \nu(\widetilde{G}(F))$
and the above identity holds.
By this identity, we see that $\theta_{\psi, \chi}^\phi(\varphi)$ is left $\widetilde{\mathbb{G}}^-_n(F)^+$-invariant.
Further, we extend $\theta_{\psi, \chi}^\phi(\varphi)$ to an automorphic form on $\widetilde{\mathbb{G}}^-_n(\mA)$ by 
the natural embedding
\[
\widetilde{\mathbb{G}}^-_n(F)^+ \backslash \widetilde{\mathbb{G}}^-_n(\mA)^+ \rightarrow 
\widetilde{\mathbb{G}}^-_n(F) \backslash 
\widetilde{\mathbb{G}}^-_n(\mA)
\]
and extension by zero.
Similarly, for a cusp form $f$ on $\widetilde{\mathbb{G}}^-_n(\mA)$,
the theta lift $\theta_{\psi, \chi}^\phi(f)$, which is an automorphic form on $\widetilde{G}(\mA)$, is defined.
Then as in the case of isometry groups, for an irreducible cuspidal automorphic representation $(\pi, V_\pi)$
of $\widetilde{G}(\mA)$, its  theta lift to
$\widetilde{\mathbb G}^-_n\left(\mA\right)$ 
is defined by
\[
\Theta_{V, \mathbb{W}_n}(\pi, \psi, \chi) \coloneqq \langle \theta_{\psi, \chi}^\phi(\varphi) ; \phi \in \mathcal{S}(V(\mA)^n), \varphi \in V_\pi \rangle
\]
where the right-hand side means the automorphic
representation of $\widetilde{\mathbb G}^-_n\left(\mA\right)$
generated.
Conversely, for an irreducible cuspidal automorphic representation $(\sigma, V_\sigma)$ of $\widetilde{\mathbb{G}}^-_n(\mA)$, its
 theta lift to $\widetilde{G}(\mA)$, denoted by
$\Theta_{\mathbb{W}_n, V}(\sigma, \psi, \chi)$, 
is similarly defined.

By taking into account the completion of the proof of the
Howe duality
conjecture by Gan and Takeda~\cite{GT},
we note that  the Howe duality  in the similitude unitary groups case
is deduced from the one in the isometry groups case
by Zhang~\cite[Theorem~4.3]{CZha}.
%
\subsection{Bessel periods}
Let $(\pi, V_\pi)$ be an irreducible cuspidal tempered automorphic representation of $\widetilde{G}(\mA)$ for $\widetilde{G} = \mathrm{GU}(V)\in \widetilde{\mathcal{G}}_n$.
Then the $\left(e,\psi,\Lambda\right)$-Bessel period
of $\varphi\in V_\pi$ is defined by exactly the
same integral as in Definition~\ref{def of bessel}
for the isometry group case, i.e.
\[
B_{e, \psi, \Lambda}(\varphi) = \int_{D_e(F) \backslash D_e(\mA)} \int_{S(F) \backslash S(\mA)}
\chi_{e, \Lambda}(ts)^{-1} \varphi(ts) \, ds \, dt.
\]
Thus we need to give a thought to
 the restriction $\pi\mid_{
\mathrm{U}\left(V\right)\left(\mA\right)}$
of $\pi$ to $\mathrm{U}\left(V\right)\left(\mA\right)$.

By Adler and Prasad~\cite[Theorem~1.4]{AP}, we have 
\[
\mathrm{dim}_\mC \, \mathrm{Hom}_{\mathrm{U}(V)(F_v)}\,(\tau, \pi_v) \leq 1
\]
for any irreducible admissible representation $\tau$ of $\mathrm{U}(V)(F_v)$.
Let $S$ be a finite set of places of $F$ such that 
at $v \not \in S$, $\pi_v$ is unramified and
$E_v$ is either an unramified quadratic extension field
of $F_v$ or 
$E_v = F_v \oplus F_v$.  
Let $\pi_1$ and $\pi_2$ be irreducible constituents of $\pi\mid_{
\mathrm{U}\left(V\right)\left(\mA\right)}$. 
Since $\pi$ is tempered, $\pi_1$ and $\pi_2$ are also tempered. 
As remarked in \cite[Remark~4.1.1]{GL}, we know that $\pi_{1,v}$ and $\pi_{2,v}$ are isomorphic at $v \not \in S$.
In particular, $\pi_1$ and $\pi_2$ are nearly equivalent.
Therefore, they have the same local $L$-parameter at each place.
Since the size of local $L$-packet is finite,
there exist only a finite number of possible
$\pi_{i,v}$ for $v \in S$.
Combining all these, we see that 
$\pi\mid_{
\mathrm{U}\left(V\right)\left(\mA\right)}$
 is a finite sum of irreducible constituents and we denote
\begin{equation}\label{restriction to isometry}
\pi\mid_{
\mathrm{U}\left(V\right)\left(\mA\right)}
=
\bigoplus_{i=1}^\ell \pi_{0, i}
\end{equation}
where $\pi_{0, i}$ ($1\le i\le \ell$)
 is an irreducible 
cuspidal tempered automorphic representations of $\mathrm{U}\left(V\right)\left(\mA\right)$
and each $\pi_{0, i}$ has the same $L$-parameter.
Thus
 we define
\[
L(s, \pi_v \times \Lambda_v) \coloneqq L(s, \pi_{0, i, v} \times \Lambda_v)
\quad 
\text{and}
\quad
L(s, \pi \times \Lambda) \coloneqq L(s, \pi_{0, i} \times \Lambda).
\]
%

We note the following multiplicity one theorem.
%
\begin{proposition}\label{pa: multiplicity one}
Keep the above notation. Then we have
\begin{equation}\label{e: local multiplicity one}
\mathrm{dim}_\mC\, \mathrm{Hom}_{R_{e}(F_v)}(\pi_v, \chi_{e, \Lambda, v}) \leq 1
\end{equation}
for any place $v$ of $F$.
\end{proposition}
%
\begin{proof}
Let $v$ be a place of $F$.
When $\pi_{v} |_{\mathrm{U}(V)(F_v)}$ is irreducible, 
\eqref{e: local multiplicity one}
 follows from the multiplicity
one theorem \eqref{uniqueness} in the  isometry group case.
Since $\pi_{v} |_{\mathrm{U}(V)(F_v)}$
is always irreducible when  $v$ is split in $E$,
we may suppose that $v$ is non-split in $E$ and that $\pi_{v} |_{\mathrm{U}(V)(F_v)} = \pi_{1, v} \oplus \pi_{2,v}$
with irreducible representations $\pi_{i, v}$ of $\mathrm{U}(V)(F_v)$.
We note that $\pi_{i, v}$ appears as a local component of some $\pi_{0,j}$ in \eqref{restriction to isometry}.
Hence 
$\pi_{1, v}$ and $\pi_{2,v}$ have the same $L$-parameter.
%
By Beuzart-Plessis~\cite{BP1,BP2},  there exists
at most  one element in a given Vogan $L$-packet
which has the $(e, \psi_v,  \Lambda_v)$-Bessel model. Thus, at most one of $\pi_{1, v}$ 
and $\pi_{2,v}$ has the
$(e, \psi_v, \Lambda_v)$-Bessel model. Hence we may suppose that 
 $\pi_{1, v}$  does not have the
 $(e, \psi_v, \Lambda_v)$-Bessel model.
 Then we have
 \[
\mathrm{Hom}_{R_{e}(F_v)}(\pi_v, \chi_{e, \Lambda, e})
=\mathrm{Hom}_{R_{e}(F_v)}(\pi_{2, v}, \chi_{e, \Lambda, v}),
\]
and our assertion follows from the multiplicity
one theorem \eqref{uniqueness} in the isometry group case.
\end{proof}
%
From the proof of Proposition~\ref{pa: multiplicity one}
 and the decomposition \eqref{restriction to isometry}, the following corollary readily follows.
%
\begin{corollary}
\label{app cor}
Suppose that $B_{e, \psi, \Lambda} \not \equiv 0$ on $V_\pi$.
Then there exists a  unique constituent $\pi_{0,i}$
in \eqref{restriction to isometry}
such that $B_{e, \psi, \Lambda} \not \equiv 0$ on $V_{\pi_{0, i}}$.
\end{corollary}
%
%
%
\subsection{The Gan--Gross--Prasad conjecture
for $\widetilde{G}$}
As an analogue of Theorem~\ref{opp GGP}, we prove the following result.
\begin{proposition}
\label{prp:a2}
Keep the above notation. 

\begin{enumerate}
\item
For an irreducible cuspidal tempered automorphic representation
$\left(\pi, V_\pi\right)$ of $\widetilde{G}\left(\mA\right)$
for $\widetilde{G}\in\widetilde{\mathcal G}_n$,
the following two conditions are equivalent.
\begin{enumerate}
\item $L(\frac{1}{2}, \pi \times \Lambda ) \ne 0$ .
\item There exists $\widetilde{G}^\prime \in \widetilde{\mathcal G}_n$ and an
irreducible cuspidal tempered automorphic representation $\pi^\prime$ of $\widetilde{G}^\prime(\mA)$ such that 
all irreducible constituents of $\pi^\prime |_{\mathrm{U}(V^\prime)(\mA)}$ 
and $\pi |_{\mathrm{U}(V)(\mA)}$ are nearly equivalent and $\pi^\prime$ has the $(e, \psi, \Lambda)$-Bessel model.
\end{enumerate}
\item
For an irreducible cuspidal tempered automorphic representation
$\left(\pi, V_\pi\right)$ of $\widetilde{G}\left(\mA\right)$
for $\widetilde{G}\in\widetilde{\mathcal G}_n$, suppose 
that there exists an irreducible cuspidal $\psi_{N_n, \lambda}$-generic automorphic representation $\pi^\circ$
of $\widetilde{\mathbb{G}}^-_n(\mA)$ such that $\pi^\circ$ is nearly equivalent to $\pi$ where $\lambda=\left(e,e\right)$.
Then the condition (a) above is equivalent
to the  following condition.
\begin{enumerate}
\item[$\textrm{(b)}^\prime$] There exists $\widetilde{G}^\prime \in \widetilde{\mathcal G}_n$ and
an irreducible cuspidal tempered automorphic representation $\pi^\prime$ of $\widetilde{G}^\prime(\mA)$ such that 
$\pi^\prime$  and $\pi$ are nearly equivalent and $\pi^\prime$ has the 
$(e, \psi, \Lambda)$-Bessel model.
\end{enumerate}
\end{enumerate}
\end{proposition}
\begin{Remark}
The condition $(b)^\prime$ is a natural analogue of 
the condition (2) in Theorem~\ref{opp GGP}. 
As we used in the proof of Theorem~\ref{opp GGP}, the condition on non-vanishing of the local Bessel period
implies the nearly equivalence between $\pi^\prime_v |_{\mathrm{U}(V^\prime)(F_v)}$ and $\pi_v |_{\mathrm{U}(V)(F_v)}$.
However, at each place $v$ of $F$ such that 
$\left[
\widetilde{G}(F_v) \colon E_v^\times\, \mathrm{U}(V)(F_v)
\right]= 2$, 
$\pi^\prime_v |_{\mathrm{U}(V^\prime)(F_v)}$ and $\pi_v |_{\mathrm{U}(V)(F_v)}$ may have different extensions
to $\widetilde{G}(F_v)$, and thus they may have extensions to $\widetilde{G}(\mA)$ and $\widetilde{G}^\prime(\mA)$ which are not nearly equivalent.
\end{Remark}
\begin{Remark}
We note that 
the temperedness of a constituent
$\pi_0$ of $\pi\mid_{\mathrm{U}\left(\mA\right)}$
implies that
$L \left(\frac{1}{2}, \pi \times \Lambda \right) = L \left(\frac{1}{2}, \pi_0 \times \Lambda \right) \ne 0$ is equivalent to 
$L^S \left(\frac{1}{2}, \pi \times \Lambda \right) =L^S \left(\frac{1}{2}, \pi_0 \times \Lambda \right) \ne 0$
for any finite set $S$ of places of $F$
containing all archimedean places 
by \cite[Lemma~10.1(1)]{Yam}.
\end{Remark}
\begin{proof}[Proof of Proposition~\ref{prp:a2}]
(1)
Suppose that the condition (a) holds.
Let $\pi_0$ be  an irreducible constituent of
$\pi |_{\mathrm{U}(V)(\mA)}$.
Then by the definition, we have 
\[
L \left(\frac{1}{2}, \pi_0 \times \Lambda \right) \ne 0.
\]
Hence, by Theorem~\ref{opp GGP}, there exists
 $\mathrm{U}(V^\prime) \in \mathcal{G}_n$
and an irreducible cuspidal tempered automorphic representation $\pi_0^\prime$ of $\mathrm{U}(V^\prime)(\mA)$ 
which is nearly equivalent to $\pi_0$ and has the $(e, \psi, \Lambda)$-Bessel period.
By \cite[Lemma~4.1.2]{GL}, there exists an  irreducible cuspidal automorphic representation 
$\pi^\prime$ of $\mathrm{GU}(V^\prime)(\mA)$ such that $\pi^\prime|_{\mathrm{U}(V^\prime)(\mA)}$ contains $\pi_0^\prime$.
Note that $\pi^\prime$ is tempered since $\pi_0^\prime$ is so.
Then $\pi^\prime$ clearly has the $(e, \psi, \Lambda)$-Bessel period
and any constituent of $\pi^\prime|_{\mathrm{U}(V^\prime)}$
is nearly equivalent to $\pi_0$ by 
\cite[Remark~4.1.1]{GL}.
Hence the condition (b) holds.

Conversely suppose that the condition (b) holds.
By Corollary~\ref{app cor}, there exists a  unique irreducible constituent $\pi_0^\prime$ of $\pi^\prime|_{\mathrm{U}(V^\prime)(\mA)}$
such that $\pi_0^\prime$ has  the 
$(e, \psi, \Lambda)$-Bessel model. Then by  Theorem~\ref{opp GGP},
we have 
\[
L \left(\frac{1}{2}, \pi^\prime \times \Lambda \right) = L \left(\frac{1}{2}, \pi_0^\prime \times \Lambda \right)  \ne 0.
\]
Meanwhile, since all irreducible constituents of $\pi^\prime |_{\mathrm{U}(V^\prime)(\mA)}$ 
and $\pi |_{\mathrm{U}(V)(\mA)}$ are nearly equivalent, we see that
\[
L \left(s, \pi \times \Lambda \right)  = L \left(s, \pi^\prime \times \Lambda \right).
\]
Hence the condition (a) holds.

(2)
First we recall that
$\widetilde{\mathbb G}^-_n\simeq\widetilde{\mathbb G}_n$
and 
the generic character $\psi_{N_n, \lambda}$ is 
defined by \eqref{generic character}.
The equivalence between the conditions (a)
and $(b)^\prime$ is proved by the same argument 
as in \ref{proof of theorem ggp}.
Indeed, if there exists an irreducible cuspidal $\psi_{N_n, \lambda}$-generic automorphic representation $\pi^\circ$
of $\widetilde{\mathbb{G}}^-_n(\mA)$ such that $\pi^\circ$ is nearly equivalent to $\pi$,
then we may construct $\pi^\prime$ in our assertion using theta lift as in the proof of the implication $(1) \Longrightarrow (2)$
in  Theorem~\ref{opp GGP}.
\end{proof}
%
%
\subsection{Refinement of the Gan--Gross--Prasad conjecture
for $\widetilde{G}_n$}
Let us study an analogue of Theorem~\ref{refined ggp thm}
in the case of similitude unitary groups.

Let us introduce some notation.
For an irreducible cuspidal automorphic representation $(\pi, V_\pi)$ of $\widetilde{G}(\mA)$ for $\widetilde{G} \in \widetilde{\mathcal{G}}_n$,
let $\langle\, \,,\,\,\rangle_\pi$ denote the $\widetilde{G}\left(\mathbb A\right)$-invariant
Hermitian inner product on $V_\pi$ defined by 
\[
\langle\varphi_1,\varphi_2\rangle_{\pi}\coloneqq
\int_{\mA_E^\times \widetilde{G}\left(F\right)\backslash \widetilde{G}\left(\mathbb A\right)}
\varphi_1\left(g\right)\overline{\varphi_2\left(g\right)}\,
d\tilde{g}
\quad\text{for $\varphi_1,\varphi_2\in V_\pi$},
\]
where $d\tilde{g}$ denotes the Tamagawa measure,
and let $\langle\, , \, \rangle_{\pi_v}$
be a $\widetilde{G}\left(F_v\right)$-invariant
Hermitian inner product on $V_{\pi_v}$
chosen so that we have
 $\langle\varphi_1 ,\varphi_2\rangle_\pi=\prod_v\langle\varphi_{1,v} ,\varphi_{2,v}\rangle_{\pi_v}$
for any decomposable vectors $\varphi_1=\otimes_v\,\varphi_{1,v}$,
$\varphi_2=\otimes_v\,\varphi_{2,v}\in V_\pi$.
The Haar measure constant $C_{\widetilde{G}}$ is defined
by
\[
d\tilde{g} = C_{\widetilde{G}} \prod_v d\tilde{g}_{v}
\]
where $d\tilde{g}_{v}$
is a measure on $\widetilde{G}_v$
such that the volume of a hyperspecial maximal compact subgroup
is $1$ at almost all places.

Let us recall Lapid--Mao's conjecture on 
Whittaker periods in the case of $\widetilde{\mathbb G}^-_n$.
%
\begin{Conjecture}[Conjecture~1.1 in \cite{LM}]
\label{app conj}
Let $(\sigma, V_\sigma)$ be an irreducible cuspidal $\psi_{N_n ,\lambda}$-generic automorphic representation of 
$\widetilde{\mathbb G}^-_n(\mA)$.

Then there exists a positive integer
 $k$ such that we have
\[
\frac{|W_{\psi, \lambda}(\varphi)|^2}{\langle \varphi, \varphi \rangle_\sigma}
=2^{-k} \cdot \frac{\prod_{j=1}^{2n} L \left(j, \chi_E^j \right)}{L \left(1, \sigma, \mathrm{Ad} \right)} \prod_v 
\frac{I_v^\natural(\varphi_v)}{{\langle \varphi_v, \varphi_v \rangle_{\sigma_v}}}
\]
for any non-zero decomposable vector $\varphi = \otimes \varphi_v \in V_\sigma$.
Here $W_{\psi,\lambda}$ denotes the 
$\psi_{N_n,\lambda}$-Whittaker period defined
by \eqref{whittaker period}
and
$I_v^\natural(\varphi_v)$ is defined by \eqref{def I_v}
using the inner product 
$\langle \,,\, \rangle_{\sigma_v}$.
\end{Conjecture}
%
%
%
Conjecture~\ref{app conj}
and the conjecture in the case of isometry groups,
i.e. Conjecture~\ref{main conj}
are related under the following assumption.

For an irreducible cuspidal automorphic representation 
$\pi$ of $\widetilde{\mathbb{G}}^-_n(\mA)$,
we define $X\left(\pi\right)$
as the set of characters of 
$\widetilde{\mathbb G}^-_n\left(\mA\right)
\slash 
\widetilde{\mathbb{G}}^-_n(F) \mathrm{U}(\mathbb{W}_n)(\mA)$
such that $\pi \otimes \chi \simeq \pi$. 
%
\begin{Assumption}
\label{ass app}
For any $\chi \in X(\pi)$, we have $\pi \otimes \chi = \pi$,
i.e. the space of automorphic forms 
for $\pi\otimes\chi$ and the one for $\pi$ coincide.
\end{Assumption}
%
We have the following proposition
for  $\widetilde{\mathbb{G}}^-_n$ by Lapid and Mao~\cite[Lemma~3.5]{LM}.
%
%
%
\begin{proposition}
\label{whittaker app}
Let $(\sigma, V_\sigma)$ be an irreducible cuspidal 
$\psi_{N_n ,\lambda}$-generic automorphic representation of 
$\widetilde{\mathbb{G}}^-_n(\mA)$,
which satisfies  Assumption~\ref{ass app}.
Then there exists a  unique $\psi_{N_n, \lambda}$-generic irreducible constituent $\sigma_0$ of $\sigma|_{\mathrm{U}(\mathbb{W}_n)(\mA)}$.

If Conjecture~\ref{main conj}  holds for $\sigma_0$,
then,  for any non-zero decomposable vector $\varphi = \otimes \varphi_v \in V_\sigma$, we have
\begin{equation}
\label{sim whittaker II}
\frac{|W_{\psi, \lambda}(\varphi)|^2}{\langle \varphi, \varphi \rangle_\sigma}
=2^{-k_0}\cdot
 |X(\sigma)|  \cdot \frac{\prod_{j=1}^{2n} L \left(j, \chi_E^j \right)}{L \left(1, \sigma_0, \mathrm{Ad} \right)} 
 \cdot 
 \prod_v 
\frac{I_v^\natural(\varphi_v)}{{\langle \varphi_v, \varphi_v \rangle_{\sigma_v}}}
\end{equation}
where
$I_v^\natural(\varphi_v)$ is defined as in \eqref{def I_v}
using the inner product $\langle \,,\, \rangle_{\sigma_v}$
and the base change lift of $\sigma_0$ to $\mathrm{GL}_{2n}(\mA_E)$ is an isobaric sum
 $\Sigma_1 \boxplus \cdots \boxplus \Sigma_{k_0}$.
 
In particular, when $\sigma_0$ satisfies the assumptions in Theorem~\ref{refined ggp thm}, the above formula holds
under Assumption~\ref{ass app}.
\end{proposition}
\begin{Remark}
Suppose that $\chi \in X(\sigma)$. Then $\chi$ factors through the similitude character $\nu$.
Moreover, $\chi$ should be trivial on $N_{E \slash F}(\mA_E^\times)$ from $\sigma \otimes \chi \simeq \sigma$.
Hence, $\chi = \chi_0 \circ \nu$ with a character $\chi_0$ of $\mA^\times \slash N_{E \slash F}(\mA_E^\times) F^\times$, namely
$\chi$ is equal to $\chi_E \circ \nu$ or the trivial character.
Hence, $|X(\sigma)|=1$ or $2$.
Further, we know that
$|\mathcal{S}_{\widetilde{\phi}} |\cdot |X(\sigma)|= |\mathcal{S}_\phi|$ by Lapid and Mao~\cite[Lemma~3.6]{LM}
where $\mathcal{S}_{\widetilde{\phi}}$ 
(resp. $\mathcal{S}_\phi$)
denotes the centralizer in $\widehat{\widetilde{G}}$ (resp. $\widehat{\mathbb{G}_n^-}$) of 
the image of the Arthur parameter $\widetilde{\phi}$ (resp. $\phi$)
of $\sigma$ (resp. $\sigma_0$).
\end{Remark}
%
%
%

For the proof of an explicit formula 
for Bessel periods,
we need to recall one more ingredient,
the Rallis inner product formula
in the case of similitude unitary groups,
which is stated as follows.
%
%
%
%
\begin{proposition}
\label{inner similitude}
Let $\pi$ be an irreducible cuspidal
automorphic representation of 
$\widetilde{G}\left(\mA\right)$.
Keep the notation in \ref{similitude theta}.

Then for any non-zero decomposable vectors $\varphi=\otimes_v\, \varphi_v\in V_\pi$
 and $\phi=\otimes_v\, \phi_v\in 
 \mathcal{S}\left(\left(V \otimes Y_n^+\right)\left(\mA\right)\right)$,
 we have
 \begin{equation}\label{e: rallis inner product2}
 \frac{\langle\theta_{\psi, \chi_\Lambda}^\phi\left(\varphi\right),\theta_{\psi,  \chi_\Lambda}^\phi\left(\varphi\right)
 \rangle_{\sigma}}{\langle \varphi,\varphi \rangle_\pi}
 =C_{\widetilde{G}}\cdot
 \frac{L\left(1/2,\pi \times \Lambda \right)}{\prod_{j=1}^{2n}\, L\left(j, \chi_{E}^j \right)}
 \cdot\prod_v Z_v^\circ\left(\varphi_v, \phi_v,\pi_v\right)
 \end{equation}
 where $\sigma = \theta_{\psi, \chi_\Lambda}^\phi\left(\pi\right)$ denotes the theta lift of 
 $\pi$ to $\widetilde{\mathbb G}^-_n\left(\mA\right)$.
 Here $Z_v^\circ\left(\varphi_v, \phi_v,\pi_v\right)$
 is the local integral defined by \eqref{e: normalized doubling} with respect to the pairing $\langle-, - \rangle_{\pi_v}$
 and is equal to $1$ at almost all $v$. 
 \end{proposition}
\begin{proof}
This is proved by an  argument similar to  that in
\cite[Section~7.8]{GI0} using the Siegel--Weil formula
\cite{Ich,KR}. 
Indeed, we may extend the Siegel--Weil formula in \cite{Ich,KR} to the case of similitude unitary groups
as in \cite[Section~4]{HK}. In this formula, we regard $\Phi_{\tau(\phi \otimes \phi)}$
as an element of the degenerate principal series representation 
$\mathrm{Ind}_{\widetilde{P}(\mA)}^{\mathrm{GU}(V^\Box)(\mA)}(\Lambda)$
with the Siegel parabolic subgroup $\widetilde{P}$ of $\mathrm{GU}(V^\Box)$, which is defined
as the stabilizer of $V_+^\Box$ in $\mathrm{GU}(V^\Box)$.
Then the argument in \cite[Section~7.8]{GI0} may be applied
word-for-word in our case.
\end{proof}
By an argument similar to
 the proof of Theorem~\ref{refined ggp thm},
 the combination of Proposition~\ref{whittaker app}, Proposition~\ref{inner similitude} and Proposition~\ref{pullback whittaker prp}
 yields the following theorem,
 which is an analogue of Theorem~\ref{refined ggp thm}
 in the case of similitude unitary groups.
 %
 %
 %
 %
 \begin{theorem}
\label{sim Bessel II}
Let $(\pi, V_\pi)$ be an irreducible cuspidal tempered automorphic representation of $\widetilde{G}(\mA)$ with $\widetilde{G} \in \widetilde{\mathcal{G}}_n$,
which has the $\left(e,\psi,\Lambda\right)$-Bessel period.
Let $\pi_0$ be the
 unique irreducible constituent of
 $\pi|_{\mathrm{U}(V)(\mA)}$
which has the $\left(e,\psi,\Lambda\right)$-Bessel period. 
We denote the base change lift of 
$\pi_0$ to $\mathrm{GL}_{2n}\left(\mA_E\right)$
as an isobaric sum
$\Sigma_1^\prime \boxplus \cdots \boxplus \Sigma_k^\prime$.
Let us write $\Sigma\coloneqq \Theta_{V, \mathbb{W}_n}(\pi_0, \psi, \chi_\Lambda^\Box)$ and 
assume that the formula \eqref{sim whittaker II} holds for $\Sigma$.

Then for any non-zero decomposable vector $\varphi = \otimes \varphi_v \in V_\pi$, we have 
\begin{multline}
\label{formula uni appendix}
\frac{|B_{e, \psi, \Lambda}(\varphi)|^2}{\langle \varphi, \varphi \rangle_\pi} = 2^{-k} \cdot
\left|X(\Sigma) \right|^{-1}\cdot
C_e \cdot
\left(\prod_{j=1}^{2n} L(j, \chi_{E}^j) \right)
\\
\times
 \frac{L(\frac{1}{2}, \pi \times \Lambda)}{L\left(1, \pi_0, \mathrm{Ad} \right) L(1, \chi_{E})}
 \cdot
\prod_v \frac{\alpha_v^\natural(\varphi_v)}{\langle \varphi_v, \varphi_v \rangle_{\pi_v}}.
\end{multline}
Here $\alpha_v^\natural(\varphi_v)$ is defined as in \eqref{e: abbreviation}
using the inner product $\langle \,,\, \rangle_{\pi_v}$.

In particular, when $\pi$ satisfies the assumptions in Theorem~\ref{refined ggp thm} and $\Sigma$ satisfies Assumption~\ref{ass app},
the formula \eqref{formula uni appendix} holds
for $\pi$.
\end{theorem}
%
%
%
%
\begin{Remark}
In our paper~\cite{FM3}, we prove \eqref{sim whittaker II} for an irreducible cuspidal automorphic representation $\Sigma^\prime$ of 
$\widetilde{\mathbb{G}}^-_2\simeq\mathrm{GU}\left(2,2\right)$ under some conditions.
Then we apply Theorem~\ref{sim Bessel II} to $\pi^\prime$,
the theta lift of  $\Sigma^\prime$ to 
$\widetilde{\mathbb{G}}_2$, 
and we obtain the formula \eqref{formula uni appendix}
for $\pi^\prime$.
\end{Remark}
\noindent
\textbf{Acknowledgement}
The authors would like to
thank Zhengyu Mao for kindly
pointing out an inaccuracy in the initial
proof of  Lemma~\ref{l: main lemma}.
The authors are tremendously grateful to the anonymous referee, 
who read the earlier version of the manuscript meticulously, 
provided numerous helpful comments and 
pointed out a number of  inaccuracies.
%
%
%
%
%
%
%
%
%
%
%
%
%
%
%
%
%
%
%
%
\section*{Declarations}

\noindent
\textbf{Competing interests}
On behalf of all authors, the corresponding author states that there is no conflict of interest.

\noindent
\textbf{Availability of data and material}
Not applicable.

\noindent
\textbf{Code availability} 
Not applicable.
%
%
%
%
%
%
%
%
%
%
%
%
%
%
%
%
%
%
%
%

%
%
%
%
%
%
%
%
%
%
%
%
%
%
%
%
%
%
%
%
\end{document}